\documentclass[reqno, 12pt]{amsart}

\usepackage{array}
\usepackage{amsmath}
\usepackage{amsfonts}
\usepackage{amssymb}
\usepackage{enumerate}
\usepackage{amsthm}
\usepackage{amsmath, amscd}
\usepackage{bm}
\usepackage{xy}
\xyoption{all}
\usepackage{color}
\usepackage{marvosym}
\usepackage{amsbsy}
\usepackage{hyperref}
\usepackage{tikz}
\usetikzlibrary{matrix,arrows,backgrounds}%, decorations, pathmorphing, positioning, fit}
	
\newtheorem{introtheorem}{Theorem}

\newtheorem{theorem}{Theorem}[subsection]

\newtheorem{lemma}[theorem]{Lemma}
\newtheorem{proposition}[theorem]{Proposition}
\newtheorem{corollary}[theorem]{Corollary}

\newtheorem{conjecture}[theorem]{Conjecture}

\theoremstyle{definition}
\newtheorem{definition}[theorem]{Definition}
\newtheorem{remark}[theorem]{Remark}

\newcommand{\op}[1]{\operatorname{#1}}

\newcommand{\leftexp}[2]{{\vphantom{#2}}^{#1}{#2}}

\newcommand{\modmod}[1]{/ \! \! / \!_{#1}}

\newcommand{\dcoh}[1]{\operatorname{D}(\operatorname{coh }#1)}
\newcommand{\dbcoh}[1]{\operatorname{D}^{\operatorname{b}}(\operatorname{coh }#1)}
\newcommand{\dqcoh}[1]{\operatorname{D}(\operatorname{Qcoh }#1)}

\newcommand{\weezer}{\leftexp{=}{\kern-0.23em\operatorname{W}}^{\kern-0.21em =}}

\newcommand{\sidenote}[1]{}

\def\Z{\op{\mathbb{Z}}}
\def\C{\op{\mathbb{C}}}
\def\R{\op{\mathbb{R}}}

\def\P{\op{\mathbb{P}}}

\def\tand{\text{ and } }

\title[VGIT and derived categories]{Variation of Geometric Invariant Theory quotients and derived categories}

\author[Ballard]{Matthew Ballard}
\address{
  \begin{tabular}{l}
   Matthew Ballard  \\ 
   \hspace{.1in} University of South Carolina, Columbia, SC, USA \\
   \hspace{.1in} Email: {\bf ballard@math.wisc.edu} \\
  \end{tabular}
}

\author[Favero]{David Favero}
\address{
  \begin{tabular}{l}
   David Favero \\
   \hspace{.1in} University of Alberta, Edmonton, AB, Canada \\
   \hspace{.1in} Email: {\bf favero@gmail.com} \\
  \end{tabular}
}

\author[Katzarkov]{Ludmil Katzarkov}
\address{
  \begin{tabular}{l}
   Ludmil Katzarkov \\
   \hspace{.1in} Universit\"at Wien, Fakult\"at f\"ur Mathematik,  Wien, \"Osterreich \\
   \hspace{.1in} Email: {\bf lkatzark@math.uci.edu} \\
  \end{tabular}
}

\numberwithin{equation}{section}
\begin{document}
\renewcommand{\labelenumi}{\emph{\alph{enumi})}}

\begin{abstract}
 We study the relationship between derived categories of factorizations on gauged Landau-Ginzburg models related by variations of the linearization in Geometric Invariant Theory. Under assumptions on the variation, we show the derived categories are comparable by semi-orthogonal decompositions and describe the complementary components.  We also verify a question posed by Kawamata: we show that $D$-equivalence and $K$-equivalence coincide for such variations. The results are applied to obtain a simple inductive description of derived categories of coherent sheaves on projective toric Deligne-Mumford stacks. This recovers Kawamata's theorem that all projective toric Deligne-Mumford stacks have full exceptional collections. Using similar methods, we prove that the Hassett moduli spaces of stable symmetrically-weighted rational curves also possess full exceptional collections. As a final application, we show how our results recover Orlov's $\sigma$-model/Landau-Ginzburg model correspondence.
\end{abstract}

\maketitle

\section{Introduction} \label{section: Introduction}

Geometric Invariant Theory (GIT) and birational geometry are closely tied. The lack of a canonical choice of linearization of the action, which may once have been viewed as a bug in the theory, has now been firmly established as a marvelous feature for constructing new birational models of a GIT quotient. As studied by M. Brion-C. Procesi \cite{BP}, M. Thaddeus \cite{Tha96}, I. Dolgachev-Y. Hu \cite{DH98}, and others, changing the linearization of the action leads to birational transformations between the different GIT quotients (VGIT). Conversely, any birational map between smooth and projective varieties can be obtained through such GIT variations \cite{Wlo00,HK99}.

In a different direction, the close relationship between birational geometry and derived categories of coherent sheaves has also enjoyed significant attention.  Indeed, work of A. Bondal \cite{BO95}, D. Orlov \cite{Orl92}, T. Bridgeland \cite{Bridgeland-flops}, A. Kuznetsov \cite{Kuz09b}, Y. Kawamata \cite{KawD-K}, and others, indicates that birational varieties should have derived categories of coherent sheaves related by semi-orthogonal decompositions.

Combining these themes motivates the study of derived categories of varieties related through VGIT. Our paper is focused on this. What we find is pleasantly systematic and well-structured. The methods and results provide a new perspective on the relationship between birational geometry and derived categories.

Let us give the main result on derived categories of sheaves, as it is the simplest to state. Let $X$ be a smooth, projective variety equipped with the action of a reductive linear algebraic group, $G$. Assume we have two $G$-equivariant ample line bundles, $\mathcal L_-$ and $\mathcal L_+$, satisfying the following conditions:
\begin{itemize}
 \item For $t \in [-1,1]$, let $\mathcal L_t = \mathcal L_-^{\frac{1-t}{2}} \otimes \mathcal L_+^{\frac{1+t}{2}}$. The semi-stable locus, $X^{\op{ss}}(\mathcal L_t)$, is constant for $-1 \leq t < 0$ and for $0 < t \leq 1$. Let $X^{\op{ss}}(-) := X^{\op{ss}}(\mathcal L_t)$ for $-1 \leq t < 0$, $X^{\op{ss}}(0) := X^{\op{ss}}(\mathcal L_0)$, and $X^{\op{ss}}(+) := X^{\op{ss}}(\mathcal L_t)$ for $0 < t \leq 1$.
 \item The set, $X^{\op{ss}}(0) \setminus \left(X^{\op{ss}}(-) \cup X^{\op{ss}}(+)\right)$, is connected.
 \item For any point, $x \in X^{\op{ss}}(0) \setminus \left(X^{\op{ss}}(-) \cup X^{\op{ss}}(+)\right)$, the stabilizer, $G_x$, is isomorphic to $\mathbb{G}_m$. 
\end{itemize}
Work of Thaddeus \cite{Tha96} and Dolgachev-Hu \cite{DH98} then tells us that there is a one-parameter subgroup, $\lambda: \mathbb{G}_m \to G$, a connected component, $Z_{\lambda}^0$, of the fixed locus of $\lambda$ in $X^{\op{ss}}(0)$, and disjoint decompositions,
\begin{align*}
 X^{\op{ss}}(0) & = X^{\op{ss}}(+) \sqcup S_{\lambda} \\
 X^{\op{ss}}(0) & = X^{\op{ss}}(-) \sqcup S_{-\lambda}.
\end{align*}
Here, $S_{\lambda}$ is the $G$-orbit of all points in $X$ that flow to $Z_{\lambda}^0$ as $\alpha \to 0$ in $\mathbb{G}_m$ and $S_{-\lambda}$ is the $G$-orbit of all points in $X$ that flow to $Z_{\lambda}^0$ as $\alpha \to \infty$ in $\mathbb{G}_m$. 

Let $C(\lambda)$ denote the centralizer of $\lambda$ and let $G_{\lambda}$ be the quotient, $C(\lambda)/\lambda$. Let $X \modmod{} + := [X^{\op{ss}}(+)/G]$ and $X \modmod{} - := [X^{\op{ss}}(-)/G]$ be the global quotient stacks of the $(+)$ and $(-)$ semi-stable loci by $G$. Finally, let $\mu$ be the weight of $\lambda$ on the anti-canonical bundle of $X$ along $Z_{\lambda}^0$.

Assume, for simplicity, that there is a splitting, $C(\lambda) \cong \lambda \times G_{\lambda}$, and let $X^{\lambda} \modmod{0} G_{\lambda}$ be the GIT quotient stack, $[(X^{\lambda})^{\op{ss}}(\mathcal L_0)/G_{\lambda}]$, of the fixed locus of $\lambda$, $X^{\lambda}$, by $G_{\lambda}$ using the equivariant line bundle, $\mathcal L_0$.

\begin{introtheorem} \label{theorem: intro}
 Fix $d \in \Z$.
 \begin{enumerate}
  \item If $\mu > 0$, then there are fully-faithful functors,
  \begin{displaymath}
   \Phi^+_d: \dbcoh{X \modmod{} -} \to \dbcoh{X \modmod{} +},
  \end{displaymath}
  and, for $d \leq j \leq \mu + d -1$,
  \begin{displaymath}
   \Upsilon_j^+:  \dbcoh{X^{\lambda} \modmod{0} G_{\lambda}} \to \dbcoh{X \modmod{} +},
  \end{displaymath}
  and a semi-orthogonal decomposition,
  \begin{displaymath}
   \dbcoh{X \modmod{} +} = \langle \Upsilon^+_{d}, \ldots, \Upsilon^+_{\mu + d -1}, \Phi^+_d \rangle.
  \end{displaymath}
  \item If $\mu = 0$, then there is an exact equivalence,
  \begin{displaymath}
   \Phi^+_d: \dbcoh{X \modmod{} -} \to \dbcoh{X \modmod{} +}.
  \end{displaymath}
  \item If $\mu < 0$, then there are fully-faithful functors,
  \begin{displaymath}
   \Phi^-_d: \dbcoh{X \modmod{} +} \to \dbcoh{X \modmod{} -},
  \end{displaymath}
  and, for $\mu + d+1 \leq j \leq d$,
  \begin{displaymath}
   \Upsilon_j^-:  \dbcoh{X^{\lambda} \modmod{0} G_{\lambda}} \to \dbcoh{X \modmod{} -},
  \end{displaymath}
  and a semi-orthogonal decomposition,
  \begin{displaymath}
   \dbcoh{X \modmod{} -} = \langle \Upsilon^-_{\mu + d+1}, \ldots, \Upsilon^-_{d}, \Phi^-_d \rangle.
  \end{displaymath}
 \end{enumerate}
\end{introtheorem}

In the statement above, for a semi-orthogonal decomposition, we use the functor to also denote its image. Theorem \ref{theorem: intro} is a case of Theorem \ref{theorem: VGIT and derived categories} from Section \ref{subsection: comparing derived categories}. Theorem \ref{theorem: intro} can be used the verify  a question posed by Kawamata \cite{KawD-K} in this setting; we prove that that $D$-domination and $K$-domination coincide for such variations, see Corollary \ref{corollary: D=K for elementary wall crossing} for the precise statement.

As another application, we provide a more streamlined proof of the following result of Kawamata \cite{Kaw06,Kaw12}.

\begin{introtheorem} \label{theorem: intro toric}
 Let $X$ be a smooth projective toric variety. The derived category, $\dbcoh{X}$, possesses a full exceptional collection.
\end{introtheorem}

 In \cite{Kap93}, Kapranov presented $\overline{M}_{0,n}$ as an iterated blow up of $\P^{n-3}$ along strict transforms of linear spaces. So, $\overline{M}_{0,n}$'s possession of a full exceptional collection was known by \cite{Orl92}. Yu. Manin and M. Smirnov used Keel's presentation \cite{Keel} to produce some exceptional collections on $\overline{M}_{0,n}$ in \cite{MaS}. We generalize this work by extending the methods employed in the proof of Theorem \ref{theorem: intro toric} to establish the following result concerning B. Hassett's moduli spaces of stable symmetrically-weighted rational curves, $\overline{M}_{0,n \cdot \epsilon}$ \cite{Has03}.

\begin{introtheorem} \label{theorem: intro moduli}
 The derived category, $\dbcoh{\overline{M}_{0,n \cdot \epsilon}}$, posseses a full exceptional collection.
\end{introtheorem}
This result is a particular case of a more general result, Lemma \ref{lemma: abyss gives SOD}, concerning various moduli spaces of weighted pointed rational curves.  

As a final application, we show how to recover, using VGIT, the following result of Orlov \cite{Orl09} relating the derived categories of projective complete intersections and singularity categories of affine cones.
Let $f_1,\ldots,f_c$ be a homogeneous regular sequence in $k[x_1,\ldots,x_n]$ of degrees, $d_1,\ldots,d_c$. Let $Y$ be the corresponding projective complete intersection and let
\begin{displaymath}
 S = k[x_1,\ldots,x_n]/(f_1,\ldots,f_c).
\end{displaymath}
Recall that the category, $\op{D}_{\op{sg}}(S,\Z)$, is the $\Z$-graded singularity category of $S$, \cite{Buc86,Orl09}.

\begin{introtheorem} \label{theorem: intro Orlov}
 Fix $d \in \Z$. Let $a = n - \sum d_i$
 \begin{enumerate}
  \item If $a > 0$, then there are fully-faithful functors,
  \begin{displaymath}
   \Phi^+_d: \op{D}_{\op{sg}}(S,\Z) \to \dbcoh{Y},
  \end{displaymath}
  and a semi-orthogonal decomposition,
  \begin{displaymath}
   \dbcoh{Y} = \langle \mathcal O_Y(d), \ldots, \mathcal O_Y(a + d -1), \Phi^+_d \rangle.
  \end{displaymath}
  \item If $a = 0$, then there is an exact equivalence,
  \begin{displaymath}
   \Phi^+_d: \op{D}_{\op{sg}}(S,\Z) \to \dbcoh{Y}.
  \end{displaymath}
  \item If $a < 0$, then there are fully-faithful functors,
  \begin{displaymath}
   \Phi^-_d: \dbcoh{Y} \to \op{D}_{\op{sg}}(S,\Z),
  \end{displaymath}
  and a semi-orthogonal decomposition,
  \begin{displaymath}
   \op{D}_{\op{sg}}(S,\Z) = \langle k(-d), \ldots, k(a-d+1), \Phi^-_d \rangle.
  \end{displaymath}
 \end{enumerate}
\end{introtheorem}

The ideas underlying this paper, in particular that of windows, see Section \ref{section: main result}, have appeared previously in the literature. While they can rightly be traced back to the fundamental paper of J.-P. Serre \cite{SerreFAC}, their first true appearance as a tool in the study of derived categories is in the work of Kawamata. In \cite{KawFF}, Kawamata treats the case of $\mathbb{G}_m$ actions; he constructs the fully-faithful functors, $\Phi^+_d$ and $\Phi^-_d$, from Theorem \ref{theorem: intro} and proves equivalences when $\mu = 0$. Kawamata also explicitly views this situation as a case of VGIT. In \cite{Kaw05}, Kawamata extends the methods to treat birational maps that are \'etale-locally VGIT for toric varieties and extends his earlier results to this setting. 

Shortly after \cite{KawFF}, and independently, M. Van den Bergh also studied $\mathbb{G}_m$ actions on affine space \cite{VdB} via windows, giving the fully-faithful functors, and criterion for equivalences. And, as this paper makes manifest, windows and VGIT are an essential underlying framework of Orlov's paper \cite{Orl09} even if not explicitly mentioned.

The physicists, M. Herbst, K. Hori, and D. Page, studied Abelian gauged linear $\sigma$-models in \cite{HHP}, where they rediscovered windows and used it to explain Orlov's theorem. Through work of E. Witten \cite{Witten}, the phases of gauged linear $\sigma$-model are exactly the different chambers of the GIT fan for the action. So VGIT is an implicit piece of \cite{HHP}. Building on the ideas of \cite{HHP}, E. Segal re-proved the Calabi-Yau hypersurface case of Orlov's theorem using VGIT, LG-models, and windows. Segal's development and presentation of the ideas left an indelible mark on the authors of this paper. Subsequently, Shipman extended Segal's methods to handle the Calabi-Yau complete intersection case of Orlov's theorem \cite{Shipman}. Independently, Herbst and J. Walcher extended Orlov's theorem to Calabi-Yau complete intersections in toric varieties \cite{HW} and used it to study auto-equivalences. Along this vein, W. 
Donovan constructed exotic derived equivalences via Grassmannian twists \cite{Donovan}. Donovan's work represents the first application of these ideas outside the Abelian realm. 

Two additional papers on related material appeared contemporaneously to this paper. Both are independent. The first is due to D. Halpern-Leistner \cite{HL12} and has significant overlap with this paper. Halpern-Leistner proves the existence of the fully-faithful functors, $\Phi^+_d$ and $\Phi^-_d$, and equivalences for $\mu = 0$ in Theorem \ref{theorem: intro} when $G$ is not necessarily Abelian. He also gives a definition of windows in the non-Abelian setting. The authors and Halpern-Leistner interacted during a conference at the University of Miami where the first and second authors presented preliminary results of this paper. Halpern-Leistner informed the authors of his work and later provided a preprint version of \cite{HL12} while the first version of this paper was in preparation. The second paper is due to Donovan and Segal \cite{DS} and builds on \cite{Donovan}. Again studying Grassmannian twists, Donovan and Segal use a different definition of window built from M. Kapranov's exceptional 
collection \cite{Kap88}. Neither 
of the 
concurrent works explicitly handles LG-models or focuses on the $\mu \not= 0$ case in VGIT; both ideas are essential to the applications of this paper.

Let us finish the introduction with a brief outline of the structure of the paper. In Section \ref{section: background}, we recall some facts and results on GIT, derived categories, and factorization categories. As our main geometric tool, we focus on stratifications of GIT quotients described by G. Kempf \cite{Kempf}, W. Hesselink \cite{Hess}, F. Kirwan \cite{Kir}, and L. Ness \cite{Ness}. In Section \ref{section: main result}, we prove that elementary wall crossings for these stratifications yield semi-orthogonal decompositions relating the two different LG-models. This result is then placed back in the context of GIT in Section \ref{sec: VGIT} and is used to prove $D$-equivalence and $K$-equivalence coincide for elementary wall crossings. The applications to exceptional collections appear in Section \ref{section: toric} for toric varieties and Section \ref{section: moduli} for moduli spaces of rational curves. In Section \ref{section: Orlov}, we show how to recover Orlov's theorem. 

\vspace{2.5mm}
\noindent \textbf{Acknowledgments:}
The authors have benefited immensely from conversations and correspondence with Yujiro Kawamata, Manfred Herbst, Colin Diemer, Gabriel Kerr, Alastair Craw, Sukhendu Mehrotra, Andrei C\u{a}ld\u{a}raru, Paolo Stellari, Ed Segal, Michael Thaddeus, Dmitri Orlov, Alexei Bondal, Kentaro Hori, Will Donovan, R. Paul Horja, Dragos Deliu, M. Umut Isik, Pawel Sosna, Emanuele Macr\`{i}, Alexander Kuznetsov, Daniel Halpern-Leistner, and Maxim Kontsevich and would like to thank them all  for their time, patience, and insight. 

The first named author was funded by NSF DMS 0636606 RTG, NSF DMS 0838210 RTG, and NSF DMS 0854977 FRG. The second and third named authors were funded by NSF DMS 0854977 FRG, NSF DMS 0600800, NSF DMS 0652633 FRG, NSF DMS 0854977, NSF DMS 0901330, FWF P 24572 N25, by FWF P20778 and by an ERC Grant.
\vspace{2.5mm}

\section{Background} \label{section: background}

In this section, we collect some of the ideas and results necessary for the arguments of Section~\ref{section: main result}. For the whole of the paper, $k$ will denote an algebraically-closed field of characteristic zero. The term, variety, means a separated, reduced scheme of finite-type over $k$. All points of a variety are $k$-points.

\subsection{HKKN stratifications}

In this section, we recall the basics of Mumford's Geometric Invariant Theory (GIT). After stating the standard definitions, we leave the basic results of the theory to \cite{MFK} and focus the majority of our attention on stratifications of varieties with group actions. These stratifications are essential to Theorem~\ref{theorem: elementary wall crossing}.

Let $G$ be a reductive linear algebraic group over $k$ and let $m: G \times G \to G$ be the group multiplication. Let $X$ be a smooth, quasi-projective variety and assume we have an action,
\begin{displaymath}
 \sigma: G \times X \to X.
\end{displaymath}
Let $\pi: G \times X \to X$ be the projection onto $X$.

First, let us recall some basic terminology. Let $Y \subset X$ be a subset. The \textbf{stabilizer} of $Y$ in $G$ is denoted by $G_Y$. The \textbf{orbit} of $Y$ in $X$ is denoted by $G \cdot Y$.

\begin{definition} \label{definition: equiv sheaf}
 A \textbf{quasi-coherent $G$-equivariant sheaf} on $X$, $\mathcal F$, is a quasi-coherent sheaf on $X$ together with an isomorphism, $\theta: \pi^* \mathcal F \to \sigma^* \mathcal F$, satisfying, 
\begin{displaymath}
 \left((\op{Id}_G \times \sigma) \circ (\tau \times \op{Id}_X \right))^*\theta \circ \left(\op{Id}_G \times \pi\right)^*\theta = \left(m \times \op{Id}_X\right)^*\theta,
\end{displaymath}
on $G \times G \times X$ where $\tau: G \times G \times X \to G \times G \times X$ switches the two factors of $G$ and
\begin{displaymath}
 s^{*} \theta = \op{Id}_{\mathcal F}
\end{displaymath}
 where $s: X \to G \times X$ is induced by the inclusion of the identity, $e$, in $G$. The isomorphism, $\theta$, is called the \textbf{equivariant structure}. 
 
 If $\mathcal F$ is coherent, respectively locally-free, as an $\mathcal O_X$-module, we say that $\mathcal F$ is a \textbf{coherent $G$-equivariant sheaf}, respectively \textbf{locally-free $G$-equivariant sheaf}. If $\mathcal F$ is locally-free and finite-rank, we say $\mathcal F$ is a \textbf{$G$-equivariant vector bundle}. If $\mathcal F$ is a locally-free and rank-one, we say that $\mathcal F$ is a \textbf{$G$-equivariant line bundle}, or a $G$-\textbf{line bundle}.
\end{definition}

\begin{remark}
 One can also think of $G$-equivariant sheaves as sheaves on the quotient stack, $[X/G]$.  Indeed, the category of quasi-coherent sheaves on $[X/G]$, $\op{Qcoh}([X/G])$, is equivalent to the category of $G$-equivariant quasi-coherent sheaves on $X$, $\op{Qcoh}(X,G)$, see, for example, \cite{Vis}, and similarly for coherent and locally-free sheaves. The stack notation will be prevalent in this paper. 
\end{remark}

\begin{definition}
 Let $V$ be a finite-dimensional vector space. Let $\chi: G \to \op{GL}(V)$ be a morphism of algebraic groups and $\mathcal E$ be a quasi-coherent $G$-equivariant sheaf on $X$.  We can create a new equivariant sheaf,
 \begin{displaymath}
  \mathcal E(\chi) := \pi^*V \otimes_{\mathcal O_X} \mathcal E.
 \end{displaymath}
 Here $\pi: X \to \op{Spec }k$ is the structure map and we view $V$ as a free $G$-equivariant sheaf on $\op{Spec }k$. We will employ this construction mainly when $\op{dim}V = 1$, in which case $\mathcal E(\chi)$ will be called the \textbf{twist of $\mathcal E$ by $\chi$}. 
\end{definition} 

For a global section, $f$, of a locally-free sheaf, $\mathcal V$, we let $X_f$ be the open subvariety of $X$ where $f$ is nonvanishing.

\begin{definition} \label{definition: semi-stable,stable,unstable loci}
 Let $\mathcal L$ be a $G$-line bundle. We consider three subsets of $X$.
 \begin{align*}
 X^{\op{ss}}(\mathcal L) & := \{x \in X \ | \ \exists f \in \op{H}^0(X,\mathcal L^n)^G \text{ with } n > 0, f(x) \not = 0, \tand X_f \text{ affine} \} \\
 X^{\op{s}}(\mathcal L) & := \{x \in X^{\op{ss}}(\mathcal L) \ | \ G \cdot x \text{ is closed in } X^{\op{ss}}(\mathcal L) \text{ and }G_x \text{ is finite}\} \\
 X^{\op{us}}(\mathcal L) & := X \setminus X^{\op{ss}}(\mathcal L).
\end{align*}
 The subsets, $X^{\op{ss}}(\mathcal L), X^{\op{s}}(\mathcal L), X^{\op{us}}(\mathcal L)$, are called the \textbf{semi-stable, stable}, and \textbf{unstable locus}, respectively, of $X$. Each subset is naturally a subvariety of $X$.
\end{definition}

\begin{definition} \label{definition: GIT quotient}
 Let $\mathcal L$ be a $G$-line bundle. The \textbf{GIT quotient} of $X$ by $G$ with respect to $\mathcal L$ is the quotient stack, $[X^{\op{ss}}(\mathcal L)/G]$. We denote the GIT quotient by $X \modmod{\mathcal L} G$.
\end{definition}

\begin{remark}
 The reader should note that this is {\em not} the definition of a GIT quotient given in \cite{MFK}. However, there is a close relation between the two notions. To distinguish between the two, we call the GIT quotient of \cite{MFK}, \textbf{Mumford's GIT quotient}. For the precise definition, we refer the reader to \cite{MFK}.
\end{remark}\begin{definition}
 Let $V$ be a finite-dimensional vector space. Let $\chi: G \to \op{GL}(V)$ be a morphism of algebraic groups and $\mathcal E$ be a quasi-coherent $G$-equivariant sheaf on $X$.  We can create a new equivariant sheaf,
 \begin{displaymath}
  \mathcal E(\chi) := \pi^*V \otimes_{\mathcal O_X} \mathcal E.
 \end{displaymath}
 Here $\pi: X \to \op{Spec }k$ is the structure map and we view $V$ as a free $G$-equivariant sheaf on $\op{Spec }k$. We will employ this construction mainly when $\op{dim}V = 1$, in which case $\mathcal E(\chi)$ will be called the \textbf{twist of $\mathcal E$ by $\chi$}. 
\end{definition} 

\begin{proposition} \label{proposition: when is the GIT quotient is a scheme}
 If the stabilizer of any point in $X^{\op{ss}}(\mathcal L)$ is finite, or, equivalently, if $X^{\op{s}}(\mathcal L)= X^{\op{ss}}(\mathcal L)$, then the GIT quotient, $X \modmod{\mathcal L} G$, is a Deligne-Mumford stack. If the stabilizer of any point is trivial, then $X \modmod{\mathcal L} G$ is a scheme, isomorphic to Mumford's GIT quotient. 
\end{proposition}

\begin{proof}
 Recall that $k$ is algebraically-closed of characteristic zero. One can apply Theorem 4.21 of \cite{DM} to verify the first part of the proposition, see also Corollary 2.2 of \cite{Edidin}. Corollary 2.2 of \cite{Edidin} also shows that $X \modmod{\mathcal L} G$ is an algebraic space under the assumption that all stabilizers are trivial. Theorem 13.6 of \cite{Alper} shows that Mumford's GIT quotient, $Y$, for the $G$-line bundle, $\mathcal L$, is a good moduli space of $X \modmod{\mathcal L} G$ in general. Then, Theorem 6.6 of \cite{Alper} says that $Y$ must be isomorphic to $X \modmod{\mathcal L} G$.
\end{proof}

We recall a standard, but essential, definition.

\begin{definition} \label{definition: 1-ps}
 The multiplicative group over $k$, $\op{Spec}k[t,t^{-1}]$, is denoted by $\mathbb{G}_m$. An injective group homomorphism, $\lambda: \mathbb{G}_m \to G$, is called a \textbf{one-parameter subgroup} of $G$. We shall often use $\lambda$ to also denote its image as a subgroup of $G$. We will denote the inverted one-parameter subgroup, $\lambda(\alpha^{-1})$, by $-\lambda$. If $G$ acts on $X$, we denote the fixed locus of the induced $\mathbb{G}_m$ action on $X$ by $X^{\lambda}$.
\end{definition}

Associated to a one-parameter subgroup, we have some other subgroups of $G$.

\begin{definition} \label{definition: P(l) and L(l)}
 Let $\lambda: \mathbb{G}_m \to G$ be a one-parameter subgroup. We set
 \begin{displaymath}
  P(\lambda) := \{ g \in G \mid \lim_{\alpha \to 0} \lambda(\alpha) g \lambda(\alpha)^{-1} \text{ exists}\}.
 \end{displaymath}
 We also set
 \begin{displaymath}
 U(\lambda) := \{ g \in G \mid \lim_{\alpha \to 0} \lambda(\alpha) g \lambda(\alpha)^{-1} = e\}
\end{displaymath}
 and let $C(\lambda)$ be the centralizer in $G$ of $\lambda$.
\end{definition}

\begin{remark}
 A bit of clarification on the notation is perhaps in order. Whenever we have a $\mathbb{G}_m$ action, $\sigma$, on a separated scheme, $X$, and we write an expression such as
 \begin{displaymath}
  \lim_{\alpha \to 0} \sigma(\alpha,x) = x^*,
 \end{displaymath}
 we mean that there exists an extension of the morphism,
 \begin{align*}
  \mathbb{G}_m & \to X \\
  \alpha & \mapsto \sigma(\alpha,x),
 \end{align*}
 to a morphism, $\mathbb{A}^1 \to X$, that sends $0$ to $x^*$. Since we have assumed $X$ to be separated, $x^*$ is unique if it exists. 
\end{remark}

\begin{lemma} \label{lemma: rho: P to C}
 The function,
 \begin{align*}
  \rho: P(\lambda) & \to C(\lambda) \\
  p & \mapsto \lim_{\alpha \to 0} \lambda(\alpha) p \lambda(\alpha)^{-1},
 \end{align*}
 is a homomorphism.
\end{lemma}

\begin{proof}
 This is straightforward.
\end{proof}

\begin{lemma} \label{lemma: C part and U part of P}
 There is a short exact sequence of groups,
 \begin{displaymath}
  0 \to U(\lambda) \to P(\lambda) \overset{\rho}{\to} C(\lambda) \to 0.
 \end{displaymath}
 The group, $U(\lambda)$, is the unipotent radical of $P(\lambda)$ and $C(\lambda)$ is reductive.
\end{lemma}

\begin{proof}
 This is contained in the proof of Proposition 2.6 of \cite{MFK}.
\end{proof}

\begin{definition} \label{definition: weight}
 Let $V$ be a finite-dimensional vector space over $k$ and assume we have a representation of $\mathbb{G}_m$ on $V$. Decompose 
 \begin{displaymath}
  V = \bigoplus_{\chi \in \Z} V_{\chi}
 \end{displaymath}
 where $\chi: \mathbb{G}_m \to \mathbb{G}_m$ is the corresponding character of $\mathbb{G}_m$ and $V_{\chi}$ is the subspace of $V$ where $\mathbb{G}_m$ acts by $\chi$. For a vector, $v \in V_{\chi}$, we say the \textbf{weight} of $v$ is $\chi$. The set of $\chi$ such that $V_{\chi}$ is nonzero is called the set of weights of $V$.
\end{definition}

The definition of weights is geometric - we think of $V$ as a scheme not as a free sheaf on a point. Dual to the action of $\mathbb{G}_m$ on $V$ is the co-action, $\Delta: \op{Sym} V^{\vee} \to \op{Sym} V^{\vee}[t,t^{-1}]$.

\begin{definition} \label{definition: degree}
 Let $R$ be a commutative $k$-algebra. Let $\Delta: R \to R[t,t^{-1}]$ be a co-action of $\mathbb{G}_m$ on $R$. Let $M$ be an $R$-module equipped with a compatible co-action, $\Delta_M: M \to M[t,t^{-1}]$. We say that $m \in M$ is \textbf{homogeneous of degree} $l$ if $\Delta_M(m) = m \otimes t^l$. We let $M_l$ be the subspace of homogeneous elements of degree $l$.
\end{definition}

\begin{remark}
 In the case of $R = \op{Sym} V^{\vee}$, the weights of the action on $V$ are the {\em negatives} of the degrees of homogeneous elements.
\end{remark}

We next recall Mumford's numerical function, following the definition in \cite{MFK}, up to a sign.

\begin{definition} \label{definiton: weights of vector bundle and Mumford stability function}
 Let $\mathcal E$ be a locally-free quasi-coherent $G$-equivariant sheaf, $\lambda$ be a one-parameter subgroup, and $x \in X^{\lambda}$. The set of \textbf{$\lambda$-weights} of $\mathcal E$ at $x$ is the set of weights of the $\mathbb G_m$-action on the fiber of the geometric vector bundle associated to $\mathcal E$ at $x$, $\op{V}(\mathcal E)_x = \underline{\op{Spec}}(\op{Sym}\mathcal E)_x$. Denote this set by $\mu(\mathcal E,\lambda,x)$. We call the function, $\mu$,  \textbf{Mumford's numerical function}.  
\end{definition}

\begin{remark}
 One way to compute $\mu$ without forming the associated geometric vector bundle, is 
 \begin{displaymath}
  \mu(\mathcal E, \lambda, x) = - \text{degrees of $\lambda$ on $\op{H}^0(X,\mathcal E_x)$}.
 \end{displaymath}
 
 There are a few conventions which affect the sign of $\mu$ at play here. We follow \cite{MFK} with the exception of deleting the negative sign in Definition 2.2 of \cite{MFK}. Deleting this sign now seems standard, see e.g. \cite{DH98}, since it allows one to interpret $\mu$ as distance.
\end{remark}

The next lemma shows that the weights remain constant on connected components of the fixed loci.

\begin{lemma} \label{lemma: weights constant on connected components of fixed locus}
 Let $\lambda: \mathbb{G}_m \to G$ be a one-parameter subgroup, $x \in X^{\lambda}$ be a fixed point of $\lambda$, and $\mathcal E$ be an equivariant vector bundle. If $x$ and $x'$ lie in the same connected component of $X^{\lambda}$, then 
 \begin{displaymath}
  \mu(\mathcal E, \lambda , x) = \mu(\mathcal E, \lambda, x').
 \end{displaymath}
\end{lemma}

\begin{proof}
 We may forget about the $G$ action and only remember the induced $\mathbb G_m$-action on $X$. First, we note that the computation of $\lambda$-weights factors through restriction to the fixed locus of $\lambda$. Let $X^{\lambda}$ denote this fixed locus. It carries a trivial $\mathbb G_m$-action and $\mathcal E|_{X^{\lambda}}$ is an equivariant vector bundle with respect to this action. 
 
 Consider any point, $x \in X^{\lambda}$ and trivialize the geometric vector bundle, $\op{V}(\mathcal E|_{X^{\lambda}})$, in a neighborhood, $U$, of $x$ so that $\op{V}(\mathcal E|_{X^{\lambda}})|_U \cong U \times \mathbb{A}^n$.   Since $\mathbb G_m$ acts trivially on $Y$, there is a linear $\mathbb{G}_m$ action on $\mathbb{A}^n$ so that the projection onto $\op{V}(\mathcal E|_{X^{\lambda}})|_U \to \mathbb{A}^n$ is $\mathbb{G}_m$ equivariant.  The weights on any fiber in $U$ are therefore determined by the weights of this action on $\mathbb{A}^n$.  Hence, the weights are locally constant on $Y$, and therefore, by definition, $\mu$ is locally constant.
\end{proof} 

It is useful to explicitly state the following properties of $\mu$.

\begin{lemma} \label{lemma: weights under the orbit}
 Let $G$ act on $X$ and $Y$ and assume $f: X \to Y$ is a $G$-equivariant morphism. Let $\lambda: \mathbb{G}_m \to G$ be a one-parameter subgroup, $x \in X, y \in Y$ be a fixed points of $\lambda$, and $\mathcal E$ and $\mathcal E'$ be $G$-equivariant vector bundles on $Y$. Then,
 \begin{enumerate}
  \item $\mu(\mathcal E, g^{-1}\lambda g , y) = \mu(\mathcal E, \lambda, \sigma(g,y))$ for all $g \in G$.
  \item $\mu(f^*\mathcal E, \lambda, x) = \mu(\mathcal E, \lambda, f(x))$. 
  \item $\mu(\mathcal E \otimes \mathcal E', \lambda, y) = \mu(\mathcal E, \lambda, y) + \mu(\mathcal E', \lambda, y)$, where the later expression describes the pointwise sum of subsets of $\Z$.
 \end{enumerate}
\end{lemma}

\begin{proof}
 These are clear from the definitions.
\end{proof}

To explicitly determine the semi-stable, or equivalently, the unstable locus from the definition is difficult in general. However, under some mild assumptions, Mumford gave an alternative characterization of the unstable locus.

\begin{theorem} \label{theorem: Hilbert-Mumford numerical stability}
 Let $\mathcal L$ be a $G$-line bundle. Assume $X$ is proper over $k$ and $\mathcal L$ is ample.  For a point, $x \in X$, $x \in X^{\op{us}}(\mathcal L)$ if and only if there exists a one-parameter subgroup, $\lambda: \mathbb{G}_m \to G$, such that $\mu(\mathcal L,\lambda, x^*) > 0$ where $x^* := \lim_{\alpha \to 0} \sigma(\lambda(\alpha),x)$.  
\end{theorem}

\begin{proof}
 This is the Hilbert-Mumford numerical criterion, Theorem 2.1 of \cite{MFK}.
\end{proof}

We also will need a version for affine space.

\begin{proposition} \label{proposition: affine HM numerical stability}
 Let $X = \mathbb{A}^n$ and let $\mathcal L$ be a $G$-equivariant line bundle. For a point, $x \in X$, $x \in X^{\op{us}}(\mathcal L)$ if and only if there exists a one-parameter subgroup, $\lambda: \mathbb{G}_m \to G$, such that $\lim_{\alpha \to 0} \sigma(\lambda(\alpha),x)=:x^*$ exists and $\mu(\mathcal L,\lambda, x^*) > 0$. 
\end{proposition}

\begin{proof}
 This is Proposition 2.5 of \cite{King}.
\end{proof}

Next, we wish to define a type of stratification of a variety with a group action. Before giving the definition itself, we set some initial terminology.

\begin{definition} \label{definition: contracting loci}
 Let $\lambda: \mathbb{G}_m \to G$ be a one-parameter subgroup of $G$. We shall denote a connected component of $X^{\lambda}$ by $Z_{\lambda}^0$. Associated to $Z_{\lambda}^0$, we have two other subvarieties,
 \begin{align*}
  Z_{\lambda} & := \{ x \in X \mid \lim_{\alpha \to 0} \sigma(\lambda(\alpha),x) \in Z_{\lambda}^0 \}.
 \end{align*}
 We call $Z_{\lambda}$ the \textbf{contracting variety} associated to $Z^0_{\lambda}$.
 
 We will also close these varieties up under the action of $G$. We set
 \begin{align*}
  S_{\lambda}^0 & := G \cdot Z_{\lambda}^0 \\
  S_{\lambda} & := G \cdot Z_{\lambda}.
 \end{align*}
\end{definition}

\begin{definition}
 We say that a $\mathbb{G}_m$ equivariant morphism,
 \[
 f: X \to Y
 \]
 is an $\mathbb{A}$\textbf{-fibration} if for any point $y \in Y$ there is a Zariski neighborhood, $U$, and a linear action of $\mathbb{G}_m$ action on $\mathbb{A}^n$, for some $n \geq 0$, so that
 \begin{displaymath}
  f|_{f^{-1}(U)}: f^{-1}(U) \to U
 \end{displaymath}
 is $\mathbb{G}_m$ equivariantly isomorphic to the projection,
 \begin{displaymath}
  U \times \mathbb{A}^n \to U.
 \end{displaymath}
\end{definition}

\begin{proposition} \label{proposition: BB}
 Let $X$ be a smooth quasi-projective variety equipped with a $G$ action, $\sigma$. Let $\lambda: \mathbb{G}_m \to G$ be a one-parameter subgroup of $G$. Then, $X^{\lambda}$ is a smooth closed subvariety of $X$ and $Z_{\lambda}$ is a smooth locally-closed subvariety of $X$. The function,
\begin{align*}
 \pi: Z_{\lambda} & \to Z_{\lambda}^0 \\
 x & \mapsto \lim_{\alpha \to 0} \sigma(\lambda(\alpha),x),
\end{align*}
 is an $\mathbb{A}$-fibration.
\end{proposition}

\begin{proof}
 This is part of the Bia\l{}ynicki-Birula decomposition of $X$, see Theorem 4.1 in \cite{BB}.
\end{proof}

\begin{proposition} \label{proposition: P+L+BB}
 Let $X$ be a smooth quasi-projective variety equipped with a $G$ action. Let $\lambda: \mathbb{G}_m \to G$ be a one-parameter subgroup of $G$. There is a natural action of $P(\lambda)$ on $Z_{\lambda}$ and a natural action of $C(\lambda)$ on $Z_{\lambda}^0$. The morphism, $\pi$, of Proposition~\ref{proposition: BB} is equivariant with respect to these actions and the homomorphism, $\rho$, from Lemma \ref{lemma: rho: P to C}: for $x \in Z_{\lambda}$,
 \begin{equation} \label{equation: P+L+BB} 
  \pi(\sigma(p,x)) = \sigma(\rho(p),\pi(x)).
 \end{equation}
\end{proposition}
\begin{proof}
 The natural actions and Equation~\ref{equation: P+L+BB} are immediate consequences of the definitions. 
\end{proof}

Let $G \overset{P(\lambda)}{\times} Z_{\lambda}$ denote the quotient stack, $[(G \times Z_{\lambda})/P(\lambda)]$, where $p \in P(\lambda)$ acts as 
\begin{displaymath}
 (p,(g,z)) \mapsto (gp^{-1},pz).
\end{displaymath}
In general, $G \overset{P(\lambda)}{\times} Z_{\lambda}$ is an algebraic space, \cite{Tho87}, and there is a natural evaluation morphism,
\begin{align*}
 \tau_{\lambda}: G \overset{P(\lambda)}{\times} Z_{\lambda} & \to S_{\lambda} \\
 (g,z) & \mapsto \sigma(g,z).
\end{align*}

\begin{definition} \label{definition: HKKN strat}
 Assume $X$ is a smooth variety with a $G$ action. A \textbf{HKKN stratification}, $\mathfrak{K}$, of $X$ is a nested sequence of open subvarieties,
 \begin{displaymath}
  X = X_0^{\mathfrak{K}} \supset X_1^{\mathfrak{K}} \supset \cdots \supset X_n^{\mathfrak{K}},
 \end{displaymath}
 one parameter subgroups, $\lambda_j: \mathbb{G}_m \to G$, and choices of connected components, $Z_{\lambda_j}^0$, of the fixed locus of $\lambda_j$ on $X_{j-1}^{\mathfrak{K}}$, for $1 \leq j \leq n$, such that 
 \begin{itemize}
  \item $X_{j}^{\mathfrak{K}} = X_{j-1}^{\mathfrak{K}} \setminus S_{\lambda_{j}}$.
  \item For each $j$, the morphism,
  \begin{displaymath}
   \tau_{\lambda_j}: G \overset{P(\lambda_j)}{\times} Z_{\lambda_j} \to S_{\lambda_j},
  \end{displaymath}
   is an isomorphism.
  \item For each $j$, $S_{\lambda_j}$ is a a closed subvariety in $X_{j-1}^{\mathfrak{K}}$.
 \end{itemize}
 An \textbf{elementary HKKN stratification} is a HKKN stratification with $n=1$. We denote one as
 \begin{displaymath}
  X = X_{\lambda} \sqcup S_{\lambda}.
 \end{displaymath}
\end{definition}

The main source of HKKN stratifications is GIT.  

\begin{theorem} \label{theorem: GIT HKKN strat}
 Let $X$ be a smooth proper variety equipped with a $G$ action and let $\mathcal L$ be an ample $G$-line bundle. There is an HKKN stratification,
 \begin{displaymath}
  X = X_0^{\mathfrak{K}} \supset X_1^{\mathfrak{K}} \supset \cdots \supset X_n^{\mathfrak{K}} = X^{\op{ss}}(\mathcal L).
 \end{displaymath}
\end{theorem}

\begin{proof}
 This is due independently to F. Kirwan, \cite{Kir}, and L. Ness, \cite{Ness}. They built upon work of G. Kempf, \cite{Kempf}, and W. Hesselink, \cite{Hess}. Section 1 of \cite{DH98} provides a compact, but thorough, exposition of the construction of the stratification and its properties.
\end{proof}

Another class of examples of HKKN stratifications comes from GIT with $X$ an affine space.

\begin{corollary} \label{cor: affine GIT HKKN strat}
 Let $X = \mathbb{A}^m$ equipped with a linear $G$ action and let $\mathcal L$ be a $G$-line bundle. There is an HKKN stratification,
 \begin{displaymath}
  X = X_0^{\mathfrak{K}} \supset X_1^{\mathfrak{K}} \supset \cdots \supset X_n^{\mathfrak{K}} = X^{\op{ss}}(\mathcal L).
 \end{displaymath}
\end{corollary}

\begin{proof}
 One compactifies $\mathbb{A}^m$ to $\mathbb{P}^m$, embeds $G$ into $\op{PGL}_{m+1}$ via the inclusion,
 \begin{align*}
 \op{GL}_m(k) & \to \op{GL}_{m+1}(k) \\
 A & \mapsto \begin{pmatrix} 1 & 0 \\ 0 & A \end{pmatrix},
\end{align*}
 and applies Theorem~\ref{theorem: GIT HKKN strat}. This is carried out explicitly in \cite{Halic}.
\end{proof}

\subsection{Derived categories of coherent sheaves}

We first remind the reader about some basic facts involving triangulated categories, including admissible subcategories and semi-orthogonal decompositions. Standard references for the material are \cite{Bon,BK}.

We begin by recalling a pair of basic definitions.

\begin{definition} \label{definition: thick subcategory}
 Let $\mathcal T$ be a triangulated category. We say a full subcategory, $\mathcal S$, is a \textbf{thick subcategory} if $\mathcal S$ is a triangulated subcategory and $\mathcal S$ is closed under taking summands. 

 We say a full subcategory, $\mathcal S$, \textbf{generates} $\mathcal T$ if the smallest thick subcategory of $\mathcal T$ containing all objects of $\mathcal S$ is $\mathcal T$ itself. We say a set of subcategories, $\mathcal S_i, i \in I$, generates $\mathcal T$ if the union of the $\mathcal S_i$ generates $\mathcal T$.
\end{definition}

Let $\mathcal T$ be a triangulated category and $\mathcal S$ a full subcategory. Recall that the left orthogonal, $\leftexp{\perp}{\mathcal S}$, is the full subcategory $\mathcal T$ consisting of all objects, $T \in \mathcal T$, with $\op{Hom}_{\mathcal T}(T,S) = 0$ for any $S \in \mathcal S$. The right orthogonal, $\mathcal S^{\perp}$, is defined similarly.

\begin{definition}\label{def:SO}
 A \textbf{semi-orthogonal decomposition} of a triangulated category, $\mathcal T$, is a sequence of full triangulated subcategories, $\mathcal A_1, \dots ,\mathcal A_m$, in $\mathcal T$ such that $\mathcal A_i \subset \mathcal A_j^{\perp}$ for $i<j$ and, for every object $T \in \mathcal T$, there exists a diagram:
 \begin{center}
 \begin{tikzpicture}[description/.style={fill=white,inner sep=2pt}]
 \matrix (m) [matrix of math nodes, row sep=1em, column sep=1.5em, text height=1.5ex, text depth=0.25ex]
 {  0 & & T_{m-1} & \cdots & T_2 & & T_1 & & T   \\
   & & & & & & & &  \\
   & A_m & & & & A_2 & & A_1 & \\ };
 \path[->,font=\scriptsize]
  (m-1-1) edge (m-1-3) 
  (m-1-3) edge (m-1-4)
  (m-1-4) edge (m-1-5)
  (m-1-5) edge (m-1-7)
  (m-1-7) edge (m-1-9)

  (m-1-9) edge (m-3-8)
  (m-1-7) edge (m-3-6)
  (m-1-3) edge (m-3-2)

  (m-3-8) edge node[sloped] {$ | $} (m-1-7)
  (m-3-6) edge node[sloped] {$ | $} (m-1-5) 
  (m-3-2) edge node[sloped] {$ | $} (m-1-1)
 ;
 \end{tikzpicture}
 \end{center}
 where all triangles are distinguished and $A_k \in \mathcal A_k$. We shall denote a semi-orthogonal decomposition by $\langle \mathcal A_1, \ldots, \mathcal A_m \rangle$. If the subcategories in a semi-orthogonal decomposition are the essential images of fully-faithful functors, $\Upsilon_i: \mathcal A_i \to \mathcal T$, we shall also denote a semi-orthogonal decomposition by
 \begin{displaymath}
  \langle \Upsilon_1, \ldots, \Upsilon_m \rangle.
 \end{displaymath}

 Let $E_1,\ldots,E_n$ be objects of $\mathcal T$. We say that $E_1,\ldots,E_n$ is an \textbf{exceptional collection} if 
 \begin{displaymath}
  \op{Hom}_{\mathcal T}(E_i,E_i[l]) = \begin{cases} k & \text{if } l = 0 \\ 0 & \text{otherwise} \end{cases}
 \end{displaymath}
 for all $i$ and
 \begin{displaymath}
  \op{Hom}_{\mathcal T}(E_j,E_i[l]) = 0
 \end{displaymath}
 for all $j > i$ and all $l$. We say that $E_1,\ldots,E_n$ is \textbf{full} if $E_1,\ldots,E_n$ generates $\mathcal T$. 
\end{definition}

Closely related to the notion of a semi-orthogonal decomposition is the notion of a left/right admissible subcategory of a triangulated category. 

\begin{definition}
 Let $f: \mathcal A \to \mathcal T$ be the inclusion of a full triangulated subcategory of $\mathcal T$. The subcategory, $\mathcal A$, is called \textbf{right admissible} if the inclusion functor, $f$, has a right adjoint and \textbf{left admissible} if it has a left adjoint. A full triangulated subcategory is called \textbf{admissible} if it is both right and left admissible.
\end{definition}

\begin{proposition} \label{proposition: characterizations of SODs}
 Let $\mathcal T$ be a triangulated category and let $\mathcal S$ be a thick subcategory. The following are equivalent:
\begin{enumerate}
 \item The subcategories, $\mathcal S$ and $\mathcal S^{\perp}$, generate $\mathcal T$.
 \item The category, $\mathcal T$, admits a semi-orthogonal decomposition $\langle \mathcal S^{\perp}, \mathcal S \rangle$.
 \item The inclusion, $\mathcal S \to \mathcal T$, is right admissible. 
 \item The inclusion, $\mathcal S^{\perp} \to \mathcal T$, is left admissible.
\end{enumerate}
\end{proposition}

\begin{proof}
 This is Lemma 3.1 of \cite{Bon}.
\end{proof}

\begin{remark}
 What is called a semi-orthogonal decomposition here is sometimes called a weak semi-orthogonal decomposition, \cite{Orl09}. A strong semi-orthogonal decomposition is a semi-orthogonal decomposition where all subcategories are admissible. 
\end{remark}

Next, we turn to computing morphisms in the derived category of $[X/G]$. First, we recall an important definition. 

\begin{definition}
 A linear algebraic group, $G$, is called \textbf{linearly reductive} if the functor of $G$-invariants for linear representations of $G$, i.e.
 \begin{align*}
  \bullet^G: \op{Qcoh}([\op{pt}/G]) & \to \op{Qcoh}(\op{pt}) \\
  V & \mapsto V^G,
 \end{align*}
 is exact.
\end{definition}

\begin{proposition} \label{proposition: reductive = linearly reductive}
 Over a field of characteristic zero, which is any field appearing in this paper, the following are equivalent:
 \begin{enumerate}
  \item $G$ is reductive.
  \item $G$ is linearly reductive.
 \end{enumerate}
\end{proposition}

\begin{proof}
 See Appendix A of \cite{MFK}.
\end{proof}

As $G$ is reductive, we can simplify the computation of morphism spaces in the derived category of $G$-equivariant sheaves with the following lemma.

\begin{lemma} \label{lemma: taking G invariants of Ext is G-Ext}
 Assume that $G$ is reductive and let $\mathcal E$ and $\mathcal F$ be two $G$-equivariant quasi-coherent sheaves on $X$. Then, there are natural isomorphisms,
\begin{displaymath}
 \op{Ext}^i_{[X/G]}(\mathcal E, \mathcal F) \cong \op{Ext}^i_X(\mathcal E, \mathcal F)^G.
\end{displaymath}
 Moreover, if $\mathcal E$ and $\mathcal F$ are complexes of $G$-equivariant quasi-coherent sheaves on $X$, then,
\begin{displaymath}
 \op{Hom}_{[X/G]}(\mathcal E, \mathcal F[i]) \cong \op{Hom}_X(\mathcal E, \mathcal F[i])^G, 
\end{displaymath}
 where the morphism spaces above are taken in the derived category. 
\end{lemma}

\begin{proof}
 The first statement is a special case of the second. There is an isomorphism of functors, \[
\op{Hom}_{[X/G]}(\bullet,\bullet) = \op{Hom}_X(\bullet,\bullet)^G.
\]
 So, in general, we have a spectral sequence whose $E_2$-page is the composition of the components of the derived functors from $\op{Hom}_X$ and the functor of $G$-invariants. However, if we assume that $G$ is reductive, the functor of $G$-invariants is exact. Therefore, the spectral sequence degenerates at the $E_2$-page to yield the isomorphism.
\end{proof}

We will often assume, in addition, that $X$ admits a $G$-invariant open affine covering, $\mathfrak U$. We can compute cohomology on $[X/G]$ using a \v{C}ech complex associated to such a cover. 

If $G$ is an algebraic torus, then a result of Sumihiro states such a cover exists. 

\begin{theorem} \label{theorem: Sumihiro}
 If $G$ is an algebraic torus acting on a normal scheme, $X$, then $X$ admits a $G$-invariant open affine covering. 
\end{theorem}

\begin{proof}
 This is Corollary 2 of \cite{Sumihiro}.
\end{proof}

\begin{proposition} \label{proposition: generators for equivariant derived category}
 Let $X$ be a smooth variety equipped with a $G$ action. If $G$ is reductive, then any object of $\dbcoh{[X/G]}$ is quasi-isomorphic to a bounded complex of $G$-equivariant vector bundles.
\end{proposition}

\begin{proof}
 By Theorem 2.18 of \cite{Tho87}, for any coherent $G$-equivariant sheaf, $\mathcal F$, on $X$, there exists a surjection, 
\begin{displaymath}
 \mathcal E \twoheadrightarrow \mathcal F,
\end{displaymath}
 from a $G$-equivariant vector bundle, $\mathcal E$.

 It suffices to show that any coherent $G$-equivariant sheaf on $X$ lies in the thick category generated by locally-free $G$-equivariant sheaves. Let $\mathcal C$ be a coherent $G$-equivariant sheaf on $X$ and construct an exact sequence using Theorem 2.18 of \cite{Tho87},
\begin{displaymath}
 \mathcal E_{-n} \to \mathcal E_{-n+1} \to \cdots \to \mathcal E_{-1} \to \mathcal E_0 \to \mathcal C \to 0,
\end{displaymath}
 with $\mathcal E_{-i}$ a $G$-equivariant vector bundle. If $n > \op{dim }X$, then the kernel, $\mathcal K$, of the map $\mathcal E_{-n} \to \mathcal E_{-n+1}$ is locally-free and $G$-equivariant.
\end{proof}

\subsection{Factorization categories}

As usual, $X$ is a smooth and quasi-projective variety with the action of a linear algebraic group, $G$, over an algebraically closed field of characteristic zero, $k$. Let $\mathcal L$ be a $G$-equivariant line bundle on $X$ and let $w$ be a $G$-invariant section of $\mathcal L$. In this section, we recall the appropriate analog of the bounded derived category of coherent sheaves for a quadruple, $(X,G,\mathcal L,w)$. Most of the results presented are due to L. Positselski \cite{Pos1,Pos2}.

\begin{definition}
 A {\bf gauged Landau-Ginzburg model}, or \textbf{gauged LG-model}, is the quadruple, $(X,G,\mathcal L,w)$, with $X$, $G$, $\mathcal L$, and $w$ as above. We shall commonly denote a gauged LG-model by the pair $([X/G],w)$. 
\end{definition}

A case of importance in this paper: $w$ is a regular function on $X$ such that there exists a character, $\chi: G \to \mathbb G_m$, with the property that $w(\sigma(g,x)) = \chi(g) w(x)$ for all $g \in G$ and $x \in X$.  In this case, we say that $w$ is \textbf{$G$-semi-invariant}. 

In this section, to unclutter the notation, given a quasi-coherent $G$-equivariant sheaf, $\mathcal E$, we denote $\mathcal E \otimes \mathcal L^n$ by $\mathcal E(n)$. Given a morphism, $f: \mathcal E \to \mathcal F$, we denote $f \otimes \op{Id}_{\mathcal L^n}$ by $f(n)$. Following Positselski, \cite{Pos1,Pos2}, one gives the following definition.

\begin{definition}
 A \textbf{factorization} of a gauged LG-model, $([X/G],w)$, consists of a pair of quasi-coherent $G$-equivariant sheaves, $\mathcal E^{-1}$ and $\mathcal E^0$, and a pair of $G$-equivariant $\mathcal O_X$-module homomorphisms,
\begin{align*}
 \phi^{-1}_{\mathcal E} &: \mathcal E^{0}(-1) \to \mathcal E^{-1} \\
 \phi^0_{\mathcal E} &: \mathcal E^{-1} \to \mathcal E^0
\end{align*}
 such that the compositions, $\phi^0_{\mathcal E} \circ \phi^{-1}_{\mathcal E} : \mathcal E^0(-1) \to \mathcal E^0$ and  $\phi^{-1}_{\mathcal E}(1) \circ \phi_{\mathcal E}^0: \mathcal E^{-1} \to \mathcal E^{-1}(1)$, are isomorphic to multiplication by $w$.  We shall often simply denote the factorization $(\mathcal E^{-1}, \mathcal E^0, \phi_{\mathcal E}^{-1}, \phi_{\mathcal E}^0)$ by $\mathcal E$. The quasi-coherent $G$-equivariant sheaves, $\mathcal E^0$ and $\mathcal E^{-1}$, are called the \textbf{components of the factorization}, $\mathcal E$. We also set
 \begin{displaymath}
  \mathcal E^i := \begin{cases} \mathcal E^0(j) & \text{ if } i=2j \\ \mathcal E^{-1}(j) & \text{ if } i = 2j-1. \end{cases}
 \end{displaymath}

 A \textbf{morphism of factorizations}, $g: \mathcal E \to \mathcal F$, is a pair of morphisms of quasi-coherent $G$-equivariant sheaves,
 \begin{align*}
  g^{-1} & : \mathcal E^{-1} \to \mathcal F^{-1} \\
  g^0 & : \mathcal E^0 \to \mathcal F^0,
 \end{align*}
 making the diagram,
 \begin{center}
 \begin{tikzpicture}[description/.style={fill=white,inner sep=2pt}]
  \matrix (m) [matrix of math nodes, row sep=3em, column sep=3em, text height=1.5ex, text depth=0.25ex]
  {  \mathcal E^{0}(-1) & \mathcal E^{-1} & \mathcal E^{0} \\
   \mathcal F^{0}(-1) & \mathcal F^{-1} & \mathcal F^{0} \\ };
  \path[->,font=\scriptsize]
  (m-1-1) edge node[above]{$\phi^{-1}_{\mathcal E}$} (m-1-2) 
  (m-1-1) edge node[left]{$g^{0}(-1)$} (m-2-1)
  (m-2-1) edge node[above]{$\phi^{-1}_{\mathcal F}$} (m-2-2)
  (m-1-2) edge node[above]{$\phi^{0}_{\mathcal E}$} (m-1-3) 
  (m-1-2) edge node[left]{$g^{-1}$} (m-2-2)
  (m-2-2) edge node[above]{$\phi^{-1}_{\mathcal F}$} (m-2-3) 
  (m-1-3) edge node[left]{$g^{0}$} (m-2-3)
  ;
 \end{tikzpicture} 
 \end{center}
 commute.
 
 We let $\op{Qcoh}([X/G],w)$ be the Abelian category of factorizations. Similarly, we denote by $\op{Inj}([X/G],w)$ be the subcategory of $\op{Qcoh}([X/G],w)$ consisting of factorizations whose components are injective quasi-coherent $G$-equivariant sheaves, and we denote by $\op{Proj}([X/G],w)$ be the subcategory of $\op{Qcoh}([X/G],w)$ consisting of factorizations whose components are projective quasi-coherent $G$-equivariant sheaves. We also let $\op{coh}([X/G],w)$ be the Abelian subcategory of factorizations with coherent components.
 
 There is an obvious notion of a chain homotopy between morphisms in $\op{Qcoh}([X/G],w)$. We let $K(\op{Qcoh}[X/G],w)$ be the corresponding homotopy category.
\end{definition}

\begin{remark}
 Factorizations are a generalization of matrix factorizations introduced by D. Eisenbud, \cite{EisMF}.
\end{remark}

There is a natural notion of translation, $[1]$. The functor, $[1]$, sends the factorization, $\mathcal E$, to the factorization, $\mathcal E[1] := (\mathcal E^0,\mathcal E^{-1}(1), -\phi^0_{\mathcal E}, -\phi^{-1}_{\mathcal E}(1))$ and $[n]$ is the $n$-fold iterated composition of $[1]$.
 
For any morphism, $g : \mathcal E \to \mathcal F$, there is a natural cone construction. We write, $C(g)$, for the resulting factorization. It is defined as 
\begin{displaymath}
 C(g):= \left( \mathcal E^{0} \oplus \mathcal F^{-1}, \mathcal E^{-1}(1) \oplus \mathcal F^0, \begin{pmatrix} -\phi_{\mathcal E}^0 & 0 \\ g^{-1} & \phi_{\mathcal F}^{-1} \end{pmatrix}, \begin{pmatrix} -\phi_{\mathcal E}^{-1}(1) & 0 \\ g^0 & \phi_{\mathcal F}^0 \end{pmatrix} \right).
\end{displaymath}
It is a standard exercise to show that the autoequivalence, $[1]$, and the cone construction induce the structure of a triangulated category on the homotopy category, $K(\op{Qcoh}[X/G],w)$.

We wish to derive $\op{Qcoh}([X/G],w)$, however, we lack a notion of quasi-isomorphism because our ``complexes'' lack cohomology. For the usual derived categories of sheaves, one can view localization by the class of quasi-isomorphisms as the Verdier quotient by acyclic objects. In \cite{Pos1}, Positselski defined the correct substitute in $\op{Qcoh}([X/G],w)$  for acyclic complexes.  

\begin{definition}
 Let 
 \begin{equation} \label{equation: chain complex of factorizations}
  \mathcal E_t \overset{g_{t+1}}{\to} \mathcal E_{t+1} \overset{g_{t+2}}{\to} \cdots \overset{g_{0}}{\to} \mathcal E_0
 \end{equation}
 be a complex of factorizations, i.e. a sequence of morphisms in $\op{Qcoh}([X/G],w)$ satisfying $g_{i+1} \circ g_i = 0$ for all $t \leq i \leq -1$. We define a sequence of new factorizations inductively. We set
 \begin{displaymath}
  T_1 := C(g_{0}).
 \end{displaymath}
 There is natural morphism of factorizations,
 \begin{displaymath}
  \tilde{g}_i: \mathcal E_{i}[-1-i] \overset{g_{i+1}[-1-i]}{\to} \mathcal E_{i+1}[-1-i] \to T_{i+1}.
 \end{displaymath}
 We then set
 \begin{displaymath}
  T_i := C(\tilde{g}_i).
 \end{displaymath}
 Finally, the \textbf{totalization} of the complex in Equation \ref{equation: chain complex of factorizations} is defined to be the factorization, $T_{t+1}$.
\end{definition}

The following definition gives the correct analog of derived category of quasi-coherent sheaves for factorizations of LG-models. These definitions are due to Positselski, \cite{Pos1,Pos2}:
\begin{definition}
 A factorization, $\mathcal A$, is called \textbf{acyclic} if it lies in the smallest thick subcategory of $K( \op{Qcoh}[X/G],w)$ containing the totalizations of all exact complexes from $\op{Qcoh}([X/G],w)$. We let $\op{Acycl}([X/G],w)$ denote the thick subcategory of $K( \op{Qcoh}[X/G],w)$ consisting of acyclic factorizations.  The \textbf{derived category of quasi-coherent factorizations} of the LG-model, $([X/G],w)$, is the Verdier quotient,
\begin{displaymath}
 \dqcoh{[X/G],w} := K( \op{Qcoh}[X/G],w)/\op{Acycl}([X/G],w).
\end{displaymath}
 Similarly, we let $\op{acycl}([X/G],w)$ be the think subcategory of $K(\op{coh}[X/G],w)$ consisting of acyclic coherent factorizations. The \textbf{derived category of coherent factorizations}, or, simply, the \textbf{derived category}, of the LG-model $([X/G],w)$ is the Verdier quotient,
\begin{displaymath}
 \dcoh{[X/G],w} := K( \op{coh}[X/G],w)/\op{acycl}([X/G],w).
\end{displaymath}

 A morphism in $\op{Qcoh}([X/G],w)$ which becomes an isomorphism in $\dqcoh{[X/G],w}$ will be called a \textbf{quasi-isomorphism}, in analogy with derived category of sheaves. Similarly, two factorizations which are isomorphic in $\dqcoh{[X/G],w}$ are called \textbf{quasi-isomorphic}.
\end{definition}

There is a description of these Verdier localizations as homotopy categories of chain complexes of appropriate exact categories, similar to derived categories of Abelian categories. To provide such a description, one must construct resolutions of factorizations. The result below allows one to construct a resolution of factorization from resolutions of its components. It is a special case of a more general result from \cite{BDFIK-res}. 

\begin{lemma} \label{lemma: strictification}
 Let $\mathcal F$ be a factorization of an affine gauged LG model $([X/G],w)$ with $G$ reductive. Choose finite locally-free resolutions of its components,
\begin{center}
\begin{tikzpicture}[description/.style={fill=white,inner sep=2pt}]
\matrix (m) [matrix of math nodes, row sep=1em, column sep=1.5em, text height=1.5ex, text depth=0.25ex]
{  \cdots & \mathcal P_{-1}^0 & \mathcal P_0^0 & \mathcal F^0 & 0 \\
   \cdots & \mathcal P_{-1}^{-1} & \mathcal P_0^{-1} & \mathcal F^{-1} & 0 \\ };
\path[->,font=\scriptsize]
 (m-1-1) edge (m-1-2) 
 (m-1-2) edge (m-1-3)
 (m-1-3) edge node[above] {$\gamma_0$} (m-1-4)
 (m-1-4) edge (m-1-5)

 (m-2-1) edge (m-2-2) 
 (m-2-2) edge (m-2-3)
 (m-2-3) edge node[above] {$\gamma_{-1}$} (m-2-4)
 (m-2-4) edge (m-2-5)
;
\end{tikzpicture} 
\end{center}
 and form the following locally-free sheaves by combining even and odd parts of the resolutions:
\begin{align*}
 \mathcal Q^0 & = \bigoplus_{i = 2l} \mathcal P^0_i (l) \oplus \bigoplus_{i=2l+1} \mathcal P^{-1}_i(l) \\
 \mathcal Q^{-1} & = \bigoplus_{i = 2l+1} \mathcal P^0_i (l+1) \oplus \bigoplus_{i=2l} \mathcal P^{-1}_i(l). 
\end{align*}
 There exists a factorization, $\mathcal Q$, with components $\mathcal Q^0$ and $\mathcal Q^{-1}$, and a quasi-isomorphism, $\gamma: \mathcal Q \to \mathcal F$ in $\op{Qcoh}([X/G],w)$, whose components are the maps $\gamma_i$.
\end{lemma}

\begin{proof}
 This is part of \cite[Theorem 3.9]{BDFIK-res}.
\end{proof}

\begin{proposition} \label{prop: injective fact}
 Let $([X/G],w)$ be a gauged LG model. The composition,
\begin{displaymath}
 K(\op{Inj}[X/G],w) \hookrightarrow K(\op{Qcoh}[X/G],w) \to \dqcoh{[X/G],w},
\end{displaymath}
 is an exact equivalence of triangulated categories.
\end{proposition}

\begin{proof}
 Recall that we have assumed that $X$ is smooth. The arguments are repetitions of the proofs of Theorem 3.5.(a) and Theorem 3.7 of \cite{Pos1}.
\end{proof}

\begin{proposition} \label{prop: projective fact}
 Let $([X/G],w)$ be a gauged LG model with $X$ affine and $G$ reductive. The composition,
\begin{displaymath}
 K(\op{Proj}[X/G],w) \hookrightarrow K(\op{Qcoh}[X/G],w) \to \dqcoh{[X/G],w},
\end{displaymath}
 is an exact equivalence of triangulated categories.
\end{proposition}

\begin{proof}
 Recall that we have assumed that $X$ is smooth. The arguments are repetitions of the proofs of Theorem 3.5.(a) and Theorem 3.8 of \cite{Pos1}, see also Lemma 1.7 of \cite{Pos2}. Lemma \ref{lemma: strictification} can be used to give another proof of essential surjectivity.
\end{proof}

\sidenote{{\color{red} Pawel suggested this was too little to treat the idea. Maybe move it later to the proof and give better references for the non-factorization case.}}
Let us describe local cohomology for $G$-equivariant factorizations. Assume $Z$ is a closed $G$-invariant subset. Let $\mathcal E$ be a quasi-coherent $G$-equivariant factorization. Resolve $\mathcal E$ be an injective factorization, $\mathcal I$. We extend sheafy local (hyper)cohomology to factorizations by
\begin{displaymath}
 H_Z \mathcal E := H_Z^0(X,\mathcal I).
\end{displaymath}
Here, we apply the functor of sheafy local cohomology to each component and to each morphism in the factorization. Similarly, we set
\begin{displaymath}
 Q_Z \mathcal E := j_*j^* \mathcal I
\end{displaymath}
where $j: U = X \setminus Z \to X$ is the inclusion. In general, we say that a factorization, $\mathcal E$, is $Z$-\textbf{torsion} if the natural map $H_Z \mathcal E \to \mathcal E$ is a quasi-isomorphism. Let $\op{D}_Z(\op{coh} [X/G],w)$ be the thick subcategory of coherent $Z$-torsion factorizations.

\begin{proposition} \label{proposition: kernel of restriction for equivariant factorizations}
 For any factorization, $\mathcal E$, there is an exact triangle,
 \begin{displaymath}
  H_Z \mathcal E \to \mathcal E \to Q_Z \mathcal E \to H_Z \mathcal E [1],
\end{displaymath}
 in $\dqcoh{[X/G],w}$. The kernel of the functor,
 \[
 j^*: \dqcoh{[X/G],w} \to \dqcoh{[U/G],w|_U},
 \]
 is $\op{D}_Z(\op{Qcoh} [X/G],w)$.
\end{proposition}

\begin{proof}
 The first statement without $G$ is Theorem 1.10 of \cite{Pos2}. The maps involved are naturally $G$-equivariant. From the exact triangle, we see that $\mathcal E|_U$ is acyclic if and only if $H_Z \mathcal E \cong \mathcal E$, i.e. $\mathcal E$ is $Z$-torsion. 
\end{proof}

We will need to bootstrap fully-faithfulness from derived categories of coherent sheaves to factorizations. Assume we have two gauged LG-models, $(X_1,G_1,\mathcal L_1,w_1)$ and $(X_2,G_2,\mathcal L_2,w_2)$, a morphism,
\begin{displaymath}
 f: X_1 \to X_2,
\end{displaymath}
and a homomorphism,
\begin{displaymath}
 \rho: G_1 \to G_2,
\end{displaymath}
such that
\begin{itemize}
 \item $f(\sigma_1(g,x)) = \sigma_2(\rho(g),f(x))$,
 \item $f^* \mathcal L_2 \cong \mathcal L_1$,
 \item and, via the projection formula and the isomorphism above, the map,
 \begin{displaymath}
  f_* w_1: f_*\mathcal O_{X_1} \to f_* \mathcal L_1,
 \end{displaymath}
 equals 
 \begin{displaymath}
  \op{Id}_{f_* \mathcal O_{X_1}} \otimes w_2 : f_*\mathcal O_{X_1} \to f_*\mathcal O_{X_1} \otimes \mathcal L_2 \cong f_* \mathcal L_1.
 \end{displaymath}
\end{itemize}

One gets functors,
\begin{displaymath}
 f_*: \op{Qcoh}([X_1/G_1],w_1) \to \op{Qcoh}([X_2/G_2],w_2),
\end{displaymath}
and
\begin{displaymath}
 f^*: \op{Qcoh}([X_2/G_2],w_2) \to \op{Qcoh}([X_1/G_1],w_1).
\end{displaymath}
The functor, $f_*$, can be derived by replacing the argument by an injective resolution, as from Proposition \ref{prop: injective fact}, and applying $f_*$. Similarly, if $G$ is reductive, $f^*$ can be derived by taking locally-free resolutions. We get a pair of derived functors,
\begin{align*}
 \mathbf{R}f_* & : \dqcoh{[X_1/G_1],w_1} \to \dqcoh{[X_2/G_2],w_2} \\
 \mathbf{L}f^* & : \dqcoh{[X_2/G_2],w_2} \to \dqcoh{[X_1/G_1],w_1}.
\end{align*}

\begin{lemma} \label{lemma: ff on dbcoh -> ff on dcoh}
 For two factorizations, $\mathcal E_1$ and $\mathcal F_1$, in $\dcoh{[X_1/G_1],w_1}$ the maps,
 \begin{displaymath}
  \mathbf{R}f_* : \op{Hom}_{[X_1/G_1],w_1}(\mathcal E_1,\mathcal F_1[t]) \to \op{Hom}_{[X_2/G_2],w_2}(\mathbf{R}f_*\mathcal E_1,\mathbf{R}f_*\mathcal F_1[t]),
 \end{displaymath}
 are isomorphisms for all $t \in \Z$ if the maps,
 \begin{displaymath}
  \mathbf{R}f_* : \op{Hom}_{[X_1/G_1]}(\mathcal E_1^r,\mathcal F_1^s[t]) \to \op{Hom}_{[X_2/G_2]}(\mathbf{R}f_*\mathcal E_1^r,\mathbf{R}f_*\mathcal F_1^s[t]),
 \end{displaymath}
 are isomorphisms for all $r,s,t \in \Z$.
 
 Assume that $X_2$ is affine and $G$ is reductive. For two factorizations, $\mathcal E_2$ and $\mathcal F_2$, in $\dcoh{[X_2/G_2],w_2}$ the maps,
 \begin{displaymath}
  \mathbf{L}f^* : \op{Hom}_{[X_2/G_2],w_2}(\mathcal E_2,\mathcal F_2[t]) \to \op{Hom}_{[X_1/G_1],w_1}(f^*\mathcal E_2,f^*\mathcal F_2[t]),
 \end{displaymath}
 are isomorphisms for all $t \in \Z$ if the maps,
 \begin{displaymath}
  \mathbf{L}f^* : \op{Hom}_{[X_2/G_2],w_2}(\mathcal E_2^r,\mathcal F_2^s[t]) \to \op{Hom}_{[X_1/G_1],w_1}(f^*\mathcal E_2^r,f^*\mathcal F_2^s[t]),
 \end{displaymath}
 are isomorphisms for all $r,s,t \in \Z$.
\end{lemma}

\begin{proof}
 This is an application of \cite[Lemmas 4.12 and 4.13]{BDFIK-res}.
\end{proof}

Finally, we recall a result which allows one to compare derived categories of factorizations and derived categories of sheaves.

Let $X$ be a smooth variety equipped with an action of $G$. Let $\mathcal E$ be a locally-free $G$-equivariant sheaf on $X$ and let $s \in \Gamma(X,\mathcal E)^G$ be a $G$-invariant regular section. Finally let $Y$ be the zero locus of $s$. On the geometric vector bundle, $\op{V}(\mathcal E)$, $s$ induces a $G$-invariant regular function,
\begin{displaymath}
 w: \op{V}(\mathcal E) \to k. 
\end{displaymath}
We let $\mathbb{G}_m$ act on $\op{V}(\mathcal E)$ by scaling the fibers. Let $\op{V}(\mathcal E)|_Y$ be the restriction of the vector bundle to $Y$ and let $\pi: \op{V}(\mathcal E)|_Y \to Y$ be the projection. Let $i: \op{V}(\mathcal E)|_Y \to \op{V}(\mathcal E)$ be the inclusion.

Define a functor,
\begin{displaymath}
 \mathfrak I: \op{coh}[Y/G] \to \op{coh}([\op{V}(\mathcal E)/(G \times \mathbb{G}_m)],w),
\end{displaymath}
by
\begin{center}
  \begin{tikzpicture}[description/.style={fill=white,inner sep=2pt}]
   \matrix (m) [matrix of math nodes, row sep=3em, column sep=3em, text height=1.5ex, text depth=0.25ex]
   {  \mathcal F & 0 & i_*\pi^*\mathcal F, \\ };
   \path[->,font=\scriptsize]
   (m-1-1) edge[|->] (m-1-2)
   (m-1-2) edge[out=30,in=150] node[above] {$0$} (m-1-3)
   (m-1-3) edge[out=210, in=330] node[below] {$0$} (m-1-2);
  \end{tikzpicture}
\end{center}
Extend $\mathfrak I$ to a functor,
\begin{displaymath}
 \mathfrak I: \op{Ch}(\op{coh}[Y/G]) \to \op{coh}([\op{V}(\mathcal E)/(G \times \mathbb{G}_m)],w),
\end{displaymath}
by totalizing each chain complex of coherent $G$-equivariant sheaves. This functor descends to the derived categories, as it takes acyclics to acyclics,
\begin{displaymath}
 \mathfrak I: \dbcoh{[Y/G]} \to \dcoh{[\op{V}(\mathcal E)/(G \times \mathbb{G}_m)],w}.
\end{displaymath}

\begin{theorem} \label{theorem: Isik}
 The functor on derived categories, $\mathfrak I$, is an equivalence. Moreover, this equivalence takes perfect complexes in $Y$ to factorizations which are torsion along the zero section of $\op{V}(\mathcal E)$.
\end{theorem}

\sidenote{{\color{red} So I lie. In this form, it is not where I cite. However, we already know how to prove this. Question is where to put it.}{\color{blue}  Can we see that this is the functor by just using the fact that it agrees with Umut's description on the spanning class consisting of all line bundles on the base?}{\color{red} If we can compute Umut's functor, sure. But, either way we cannot have a proof here.}}
\begin{proof}
 The statement without $G$ is due to M.U. Isik and I. Shipman \cite{Isik, Shipman}. The statement about perfect complexes is clear from the construction of the equivalence in \cite{Isik}. Descent to $G$-equivariant categories follows from the fact that the main result of \cite{MR11} descends, see Section 4.3 of \cite{MR11}.
\end{proof}

We record a simple corollary of Theorem \ref{theorem: Isik}.

\begin{corollary} \label{corollary: Isik}
 We have an equivalence, $\mathfrak I: \dbcoh{[X/G]} \cong \dcoh{[X/(G \times \mathbb{G}_m)],0}$.
\end{corollary}

\begin{proof}
 We take $\mathcal E = 0$ in Theorem \ref{theorem: Isik}. 
\end{proof}

\section{Main result} \label{section: main result}

In this section, we prove the main result relating the derived categories of quotient stacks differing by an elementary wall crossing. Before stating the main result, we need to set some terminology.

\subsection{Preliminaries} \label{section: prelims for main result}

Given an elementary HKKN stratification, $\mathfrak{K}$, we let \sidenote{\color{blue} the notation looks like restricting to X as opposed to relative bundle if I'm not mistaken.  Isn't there supposed to be a backslash?}
\begin{displaymath}
 t(\mathfrak{K}) := \mu(\omega_{S_{\lambda}|X}, \lambda, x)
\end{displaymath}
for $x \in Z_{\lambda}^0$. Here,
\begin{displaymath}
 \omega_{S_{\lambda}|X} = \bigwedge\nolimits^{\op{codim} S_{\lambda}} \mathcal N^{\vee}_{S_{\lambda}|X}
\end{displaymath}
is the relative canonical sheaf of the embedding, $S_{\lambda} \to X$. By Lemma~\ref{lemma: weights constant on connected components of fixed locus}, $t(\mathfrak{K})$ does not depend on the choice of $x$. Let us note that $t(\mathfrak{K}) < 0$ as the normal vectors to $S_{\lambda}$ must have negative $\lambda$-weight along $Z_{\lambda}^0$. 

We now introduce the central idea of grade restriction windows, or, at least, our variation on the theme.

Let $\lambda: \mathbb G_m \to G$ be a one-parameter subgroup and let $Z_{\lambda}^0$ be a connected component of $X^{\lambda}$. Let $N_{S^0_{\lambda}|X} := \op{V}(\mathcal N^{\vee}_{S^0_{\lambda}|X})$ be the geometric normal bundle of $S^0_{\lambda}$ in $X$. We restrict $N_{S^0_{\lambda}|X}$ to $Z_{\lambda}^0$ and complete it along the zero section. To simplify notation, let $\widehat{N}^0 := \widehat{(N_{S^0_{\lambda}|X})|_{Z_\lambda^0}}$. For an open subset, $V$ of $Z_{\lambda}^0$, denote by $\widehat{N}^0_V$ the corresponding open subscheme of $\widehat{N}^0$.  \sidenote{{\color{blue} V is a little overused here since it's also used for a geometric vector bundle -- maybe that was your self side note?} {\color{red} Yeah, I agree. We need U and V for open sets so we should change the notation for a geometric vector bundle. How about T(E)?}}
  
\begin{definition} \label{definition: weights along completion}
 Let $\mathcal E$ be a coherent $G$-equivariant factorization of $G$-invariant section, $w$, of a $G$-line bundle, $\mathcal L$, or let $\mathcal E$ be a bounded complex of coherent $G$-equivariant sheaves. Let $I$ be a subset of $\Z$.  We say that $\mathcal E$ has \textbf{weights along $\widehat{N}^0$ in $I$} if there is an open affine cover, $\{V_j\}_{j \in J}$, of ${Z}^0_{\lambda}$ such that $\mathcal E|_{\widehat{N}^0_{V_j}}$ is $\lambda$-equivariantly quasi-isomorphic to a $\lambda$-equivariant factorization, or bounded complex, $\mathcal F$, with finite-rank locally components, $\mathcal F^l$, satisfying
 \begin{displaymath}
  \mu(\mathcal F^l,\lambda,x) \subseteq I
 \end{displaymath}
 for all $j,l \in \Z$ and any $x \in Z_{\lambda}^0$.
\end{definition}
 
\begin{definition}
 The \textbf{$I$-window}, or \textbf{$I$-grade-restricted window}, is the full triangulated subcategory of $\dcoh{[X/G],w}$ consisting of factorizations which have weights along $\widehat{N}^0$ in $I$. We denote the $I$-window by $\weezer_{\lambda,I}([X/G],w)$, or simply by $\weezer_I$ when the context allows. 
 
 Similarly, we let $\weezer_{\lambda,I}([X/G])$ denote the full triangulated subcategory of $\dbcoh{[X/G]}$ consisting of complexes which have weights along $\widehat{N}^0$ in $I$.
\end{definition}

Let us take a moment to comment on the definition, specifically how one restricts $\mathcal E$ to an open neighborhood of $\widehat{N}^0$. First, one restricts $\mathcal E$ to the completion, $\widehat{S}_{\lambda}^0$. Take a $\lambda$-invariant affine open cover, $\{U_j\}_{j \in J}$, of $S_{\lambda}^0$. Over each $U_j$, the completion of $X$ along $S_{\lambda}^0$ is isomorphic to the completion of $N_{S_{\lambda}^0|X}$ along the zero section. This allows us to, locally, view $\mathcal E$ as a sheaf on $\widehat{N}^0$. Restricting further gives $\mathcal E|_{\widehat{N}^0_{V_j}}$ where $V_j = U_j \cap Z_{\lambda}^0$. 

\begin{lemma} \label{lemma: weights of tensor and dual}
 Let $\mathcal E \in \weezer_{\lambda,I}([X/G])$ and $\mathcal F \in \weezer_{\lambda,I'}([X/G],w)$. Then, 
 \[
 \mathcal E \overset{\mathbf{L}}{\otimes} \mathcal F \in \weezer_{\lambda,I+I'}([X/G],w), \mathcal E^{\vee} \in \weezer_{\lambda,-I}([X/G]), \tand \mathcal F^{\vee} \in \weezer_{\lambda,-I'}([X/G],w).
 \]
\end{lemma}

\begin{proof}
 This is clear.
\end{proof}

\begin{remark}
 The definition of a window goes back to Kawamata \cite{KawFF} and Van den Bergh \cite{VdB} in the $\mathbb{G}_m$-case. Kawamata extended this to Abelian $G$ in \cite{Kaw05}. Windows are a central topic in \cite{HHP} and \cite{Seg2}. The definition of a window given by Halpern-Leistner \cite{HL12} is slightly different from, but equivalent to, the one given above.
\end{remark}

\subsection{Fully-faithfulness} \label{section: fully-faithfulness}

Let $X$ be a smooth and quasi-projective variety. Assume that $X$ possesses a elementary HKKN-stratification, $\mathfrak{K}$,
\begin{displaymath}
 X = X_{\lambda} \sqcup S_{\lambda}.
\end{displaymath}
Let $w$ be a $G$-invariant section of a $G$-line bundle, $\mathcal L$. Let $i: X_{\lambda} \to X$ denote the inclusion of $X_{\lambda}$ into $X$, and let 
\begin{displaymath}
 i^*: \dcoh{[X/G],w} \to \dcoh{[X_{\lambda}/G],w|_{X_{\lambda}}}
\end{displaymath}
be the functor given by pulling back a $G$-equivariant factorization from $X$ to $X_{\lambda}$. 

\begin{lemma} \label{lemma: fully-faithfulness}
 Assume that $S^0_{\lambda}$ admits a $G$-invariant open affine covering and that $\mu(\mathcal L,\lambda,x) = 0$ for $x \in Z_{\lambda}^0$. If $\mathcal E \in \weezer_I$, $\mathcal F \in \weezer_{I'}$, and $I'-I \subseteq [t(\mathfrak{K})+1,\infty)$, then the restriction map,
 \begin{displaymath}
  i^*: \op{Hom}_{[X/G],w}(\mathcal E,\mathcal F) \to \op{Hom}_{[X_{\lambda}/G],w|_{X_{\lambda}}}(\mathcal E|_{X_{\lambda}},\mathcal F|_{X_{\lambda}}),
 \end{displaymath}
 is an isomorphism.
\end{lemma}

\begin{proof}
 By Proposition \ref{proposition: kernel of restriction for equivariant factorizations}, for each factorization, $\mathcal F$, on $X$, there is an exact triangle in $\dqcoh{[X/G],w}$,
 \begin{displaymath}
  H_{S_{\lambda}} \mathcal F \to \mathcal F \to \mathbf{R}i_*i^* \mathcal F \to H_{S_{\lambda}} \mathcal F[1],
 \end{displaymath}
 where $H_{S_{\lambda}} \mathcal F$ is the sheafy local (hyper)cohomology along $S_{\lambda}$. Let $\mathcal E$ be another factorization. Then, asking for
 \begin{displaymath}
  i^*: \op{Hom}_{[X/G],w}(\mathcal E,\mathcal F[t]) \to \op{Hom}_{[X_{\lambda}/G],w|_{X_{\lambda}}}(\mathcal E|_{X_{\lambda}},\mathcal F|_{X_{\lambda}}[t])
 \end{displaymath}
 to be an isomorphism for all $t \in \Z$ is equivalent to requiring that
 \begin{displaymath}
  \op{Hom}_{[X/G],w}(\mathcal E, H_{S_{\lambda}} \mathcal F[t])=0
 \end{displaymath}
 for all $t$. Lemma \ref{lemma: vanishing local cohom} below finishes the proof.
\end{proof}

\begin{corollary} \label{corollary: fully-faithfulness}
 Let $I \subset \Z$ with $\sup \{n-m \mid n,m \in I\} > t(\mathfrak{K})$. Assume that $S^0_{\lambda}$ admits a $G$-invariant open affine covering and that $\mu(\mathcal L,\lambda,x) = 0$ for $x \in Z_{\lambda}^0$. The functor,
 \begin{displaymath}
  i^*: \weezer_I \to \dcoh{[X_{\lambda}/G],w|_{X_{\lambda}}},
 \end{displaymath}
 given by restricting $i^*$ to the subcategory, $\weezer_I$, is fully-faithful. 
\end{corollary}

\begin{proof}
 This is straightforward.
\end{proof}

To prove Lemma \ref{lemma: vanishing local cohom} below, we need to use a result of R. Thomason.

\begin{theorem} \label{theorem: Thomason}
 Let $Y$ be a Noetherian scheme with an action of a linear algebraic group, $H$, which is a closed subgroup of another linear algebraic subgroup, $G$.  Let $e: Y \to G \overset{H}{\times} Y$ be the inclusion via the identity, $\op{Spec }k \to G$. The functor,
 \begin{displaymath}
  e^*: \op{Qcoh}[G \overset{H}{\times} Y/G] \to \op{Qcoh}[Y/H],
 \end{displaymath}
 is an equivalence of Abelian categories. Moreover, for any $\mathcal F \in \op{Qcoh}[G \overset{H}{\times} Y/G]$, $e^*\mathcal F$ is locally-free, respectively coherent, if and only if $\mathcal F$ is locally-free, respectively coherent.
\end{theorem}

\begin{proof}
 This is Lemma 1.3 of \cite{Tho87}.
\end{proof}

\begin{corollary} \label{corollary: Thomason}
 Let $Y$ be a Noetherian scheme with an action of a linear algebraic group, $H$, which is a closed subgroup of another linear algebraic subgroup, $G$. Let $\mathcal L$ be a $G$-line bundle on $G \overset{H}{\times} Y$ and $w$ be a $G$-invariant section of $\mathcal L$. We have equivalences,
 \begin{align*}
  e^* & : \dcoh{[G \overset{H}{\times} Y/G],w} \to \dcoh{[Y/H],w|_Y} \\
  e^* & : \dbcoh{[G \overset{H}{\times} Y/G]} \to \dcoh{[Y/H]},
 \end{align*}
 given by the (underived) restriction along the inclusion, $e: Y \to G \overset{H}{\times} Y$.
\end{corollary}

\begin{remark}
 It is clear from the arguments of \cite{Tho87} that $Y$ may be replaced with a Noetherian formal scheme.
\end{remark}

We will also need a slight extension of the $\mathbb{A}$-fibration of the contracting locus found in Proposition \ref{proposition: BB}. Assume we have an elementary HKKN stratification, 
\begin{displaymath}
 X = X_{\lambda} \sqcup S_{\lambda}.
\end{displaymath}
From Proposition \ref{proposition: BB}, we know the morphism,
\begin{align*}
 \pi: Z_{\lambda} & \to Z_{\lambda}^0 \\\
  x & \mapsto \lim_{\alpha \to 0} \sigma(\lambda(\alpha),x)
\end{align*}
is an $\mathbb{A}$-fibration. As a condition of a HKKN stratification, we know that the map,
\begin{displaymath}
 \tau: G \overset{P(\lambda)}{\times} Z_{\lambda} \to S_{\lambda},
\end{displaymath}
is an isomorphism. Let us cross $\pi$ with $G$ to get a morphism,
\begin{displaymath}
 G \overset{P(\lambda)}{\times} \pi: G \overset{P(\lambda)}{\times} Z_{\lambda} \to  G \overset{P(\lambda)}{\times} Z_{\lambda}^0.
\end{displaymath}
We can apply $\tau$ and its inverse to get a morphism,
\begin{displaymath}
  \tilde{\pi} := \tau|_{G \overset{P(\lambda)}{\times} Z_{\lambda}^0} \circ (G \overset{P(\lambda)}{\times} \pi) \circ \tau^{-1} : S_{\lambda} \to S_{\lambda}^0.
\end{displaymath}

\begin{lemma} \label{lemma: affine fibration for S^0}
 The morphism, $\tilde{\pi}: S_{\lambda} \to S_{\lambda}^0$, is an $\mathbb{A}$-fibration.
\end{lemma}

\begin{proof}
 From the assumption that 
 \begin{displaymath}
  \tau: G \overset{P(\lambda)}{\times} Z_{\lambda} \to S_{\lambda},
 \end{displaymath} 
 is an isomorphism, it follows that 
 \begin{displaymath}
  \tau|_{G \overset{P(\lambda)}{\times} Z_{\lambda}^0}: Z_{\lambda}^0 \to S_{\lambda}^0
 \end{displaymath}
 is also an isomorphism. 
 
 Thus, we reduce to checking that 
 \begin{displaymath}
  G \overset{P(\lambda)}{\times} \pi: G \overset{P(\lambda)}{\times} Z_{\lambda} \to  G \overset{P(\lambda)}{\times} Z_{\lambda}^0
 \end{displaymath}
 is an $\mathbb{A}$-fibration. This follows immediately from Proposition~\ref{proposition: BB}. 
\end{proof}

\begin{lemma} \label{lemma: vanishing local cohom}
 Let $X$ possess an elementary HKKN stratification,
 \begin{displaymath}
  X = X_{\lambda} \sqcup S_{\lambda}.
 \end{displaymath}
 Assume that $S_{\lambda}^0$ admits a $G$-invariant affine open cover and $\mu(\mathcal L,\lambda,x) = 0$ for $x \in Z_{\lambda}^0$.
 
 If $\mathcal E \in \weezer_I$ and $\mathcal F \in \weezer_{I'}$ and $I'-I \subseteq [t(\mathfrak{K})+1,\infty)$, then
 \begin{displaymath}
  \op{Hom}_{[X/G],w}(\mathcal E, H_{S_{\lambda}} \mathcal F) = 0.
 \end{displaymath}
\end{lemma}

\begin{proof} 
 Let $\widehat{S}_{\lambda}$ denote the completion of $X$ along $S_{\lambda}$. Since $H_{S_{\lambda}} \mathcal F$ is supported on $S_{\lambda}$, completion gives a natural isomorphism, 
 \begin{displaymath}
  \op{Hom}_{[X/G],w}(\mathcal E, H_{S_{\lambda}} \mathcal F) \cong \op{Hom}_{[\widehat{S}_{\lambda}/G],w}(\widehat{\mathcal E}, \widehat{H_{S_{\lambda}} \mathcal F}).
 \end{displaymath}
 As we proceed, we will also let $H_{S_{\lambda}} \mathcal F$ and $\mathcal E$ denote their corresponding completions.
 
 To compute the morphism spaces on $\widehat{S}_{\lambda}$, we may use the \v{C}ech complex associated to a $G$-invariant affine open cover of $\widehat{S}_{\lambda}$. One hypothesis of Lemma \ref{lemma: vanishing local cohom} is the existence of such a $G$-invariant affine open cover of $S_{\lambda}^0$, $\{U_j\}_{j \in J}$. The collection, $\{\tilde{\pi}^{-1}(U_j)\}_{j \in J}$, is an affine open cover of $S_{\lambda}$ by Lemma \ref{lemma: affine fibration for S^0}. Let $\{\widehat{\tilde{\pi}^{-1}(U_j)}\}_{j \in J}$ denote the corresponding affine cover of $\widehat{S}_{\lambda}$. 
 
 To compute the morphism spaces using the \v{C}ech complex, we use a complex whose terms are 
 \begin{displaymath}
  \op{Hom}_{[\widehat{\tilde{\pi}^{-1}(U_{j_1})} \cap \cdots \cap \widehat{\tilde{\pi}^{-1}(U_{j_s})}/G],w|_{\widehat{\tilde{\pi}^{-1}(U_{j_1})} \cap \cdots \cap \widehat{\tilde{\pi}^{-1}(U_{j_s})}}}(\mathcal E, H_{S_{\lambda}} \mathcal F).
 \end{displaymath}
 
 It suffices to prove this claim: let $U$ be a $G$-invariant affine open subset of $S_{\lambda}^0$. Then,
 \begin{displaymath}
  \op{Hom}_{[\widehat{\tilde{\pi}^{-1}(U)}/G], w|_{\widehat{\tilde{\pi}^{-1}(U)}}}( \mathcal E, H_{S_{\lambda}} \mathcal F) = 0.
 \end{displaymath}

 As $\tilde{\pi}^{-1}(U)$ is affine, the completion of $X$ along $S_{\lambda}$ restricted to $\tilde{\pi}^{-1}(U)$ is isomorphic to the completion of the geometric vector bundle, $(N_{S_{\lambda}|X})|_{\tilde{\pi}^{-1}(U)}$, along the zero section of the bundle, $\widehat{(N_{S_{\lambda}|X})|_{\tilde{\pi}^{-1}(U)}}$. To unclutter the notation, let us set $N := N_{S_{\lambda}|X}$. For a subset, $Y$, of $X$, we denote the restriction of the geometric vector bundle, $N$, to $Y$ as $N_Y$. Let $V = U \cap Z_{\lambda}^0$. Let $\widehat{N}_{\pi^{-1}(V)}$ denote the completion of $N_{\pi^{-1}(V)}$ along the zero section.
 
 Since we assume that $X_{\lambda} \sqcup S_{\lambda}$ is a HKKN stratification, we have an isomorphism, $\tau: G \overset{P(\lambda)}{\times} Z_{\lambda} \cong S_{\lambda}$. This immediately extends to an isomorphism,
 \begin{displaymath}
  \tau_N: G \overset{P(\lambda)}{\times} N_{Z_{\lambda}} \cong N,
 \end{displaymath}
 which preserves the zero section. Completing along the zero sections gives an isomorphism,
 \begin{displaymath}
  \widehat{\tau}_N: G \overset{P(\lambda)}{\times} \widehat{N}_{Z_{\lambda}} \cong \widehat{N}.
 \end{displaymath}
 
 We may apply Corollary \ref{corollary: Thomason} to get an equivalence,
 \begin{equation} \label{equation: Thomason equiv}
  \dqcoh{[\widehat{N}_{\tilde{\pi}^{-1}(U)}/G],w|_{\widehat{N}_{\tilde{\pi}^{-1}(U)}}} \cong \dqcoh{[\widehat{N}_{\pi^{-1}(V)}/P(\lambda)],w|_{\widehat{N}_{\pi^{-1}(V)}}}.
 \end{equation}
 As $\widehat{\tilde{\pi}^{-1}(U)}$ is affine and $G$ is reductive, coherent $G$-equivariant sheaves on $\widehat{\tilde{\pi}^{-1}(U)}$ have no higher cohomology. Appealing to Theorem \ref{theorem: Thomason}, the higher derived functors of $P(\lambda)$-invariants must vanish on the representations furnished by global sections of coherent sheaves on $\widehat{(N_{S_{\lambda}|X})|_{\pi^{-1}(V)}}$. While local cohomology is not coherent, it admits a filtration by coherent $P(\lambda)$-equivariant subsheaves so exactness of $P(\lambda)$-invariants extends to morphism spaces between coherent sheaves and local-cohomology sheaves of coherent sheaves. Thus, we are able to pass the functor of $P(\lambda)$-invariants outside,
 \begin{gather*} \label{equation: nasty expression}
  \op{Hom}_{[\widehat{N}_{\pi^{-1}(V)} /P(\lambda)], w|_{\widehat{N}_{\pi^{-1}(V)} }}( \mathcal E|_{\widehat{N}_{\pi^{-1}(V)}}, H_{S_{\lambda}} \mathcal F|_{\widehat{N}_{\pi^{-1}(V)}}) \cong \\ \op{Hom}_{\widehat{N}_{\pi^{-1}(V)}, w|_{\widehat{N}_{\pi^{-1}(V)} }}( \mathcal E|_{\widehat{N}_{\pi^{-1}(V)}}, H_{S_{\lambda}} \mathcal F|_{\widehat{N}_{\pi^{-1}(V)}})^{P(\lambda)}. 
 \end{gather*}
 
 The $P(\lambda)$ invariants will be a smaller vector space than the $\lambda$-invariants. We may reduce to showing the vanishing of 
 \begin{equation} \label{equation: lambda nasty} 
  \op{Hom}_{\widehat{N}_{\pi^{-1}(V)}, w|_{\widehat{N}_{\pi^{-1}(V)} }}( \mathcal E|_{\widehat{N}_{\pi^{-1}(V)}}, H_{S_{\lambda}} \mathcal F|_{\widehat{N}_{\pi^{-1}(V)}})^{\lambda}.
 \end{equation}
 
 We can complete further, from $\widehat{N}_{\pi^{-1}(V)} = \widehat{(N_{S_{\lambda}|X})|_{\pi^{-1}(V)}}$ to $\widehat{N}^0_V = \widehat{(N_{S_{\lambda}^0|X})|_V}$. Lemma \ref{lemma: completion reduction, ff} below states that the map on the morphism space of Equation~\eqref{equation: lambda nasty} is an isomorphism. Thus, we may replace $\widehat{N}_{\pi^{-1}(V)}$ by $\widehat{N}^0_V$ and continue the argument.
 
 It is straightforward to see that there is an isomorphism of factorizations,
 \begin{equation*} \label{equation: another eqn in ff}
  H_{S_{\lambda}} \mathcal F|_{\widehat{N}_{\pi^{-1}(V)}} \cong H_{\pi^{-1}(V)} (\mathcal F|_{\widehat{N}_{\pi^{-1}(V)}}).
 \end{equation*}
 Passing to $\widehat{N}^0_V$, we have
 \begin{displaymath}
  (H_{\pi^{-1}(V)} (\mathcal F)|_{\widehat{N}_{\pi^{-1}(V)}})|_{\widehat{N}^0_V} \cong H_{\widehat{\pi^{-1}(V)}} (\mathcal F|_{\widehat{N}^0_V})
 \end{displaymath}
 where $\widehat{\pi^{-1}(V)}$ is the completion of $\pi^{-1}(V)$ along $V$.
 
 To check the vanishing of
 \begin{displaymath}
  \op{Hom}_{\widehat{N}^0_V, w|_{\widehat{N}^0_V }}( \mathcal E|_{\widehat{N}^0_V}, H_{\widehat{\pi^{-1}(V)}} (\mathcal F|_{\widehat{N}^0_V}))^{\lambda}
 \end{displaymath}
 it suffices to show the vanishing for the components of the factorizations,
 \begin{displaymath}
  \op{Hom}_{\widehat{N}^0_V}( \mathcal E^r|_{\widehat{N}^0_V}, H_{\widehat{\pi^{-1}(V)}} (\mathcal F^s|_{\widehat{N}^0_V}))^{\lambda} =0
 \end{displaymath}
 for all $r,s \in \Z$.
 
 From the hypothesis of the lemma, possibly after shrinking $V$, we can replace $\mathcal E|_{\widehat{N}^0_V}$ by a $\lambda$-equivariantly quasi-isomorphic factorization whose components are of the form,
 \begin{displaymath}
  \bigoplus_{n \in I} \mathcal O(n)^{\oplus m_n}.
 \end{displaymath}
 Similarly, we can replace $\mathcal F|_{\widehat{N}^0_V}$ by a $\lambda$-equivariantly quasi-isomorphic factorization whose components are of the form,
 \begin{displaymath}
  \bigoplus_{n \in I'} \mathcal O(n)^{\oplus m'_n}.
 \end{displaymath}
 
 Thus,
 \begin{displaymath}
  \op{Hom}_{\widehat{N}^0_V}( \mathcal E^r|_{\widehat{N}^0_V}, H_{\widehat{\pi^{-1}(V)}} (\mathcal F^r|_{\widehat{N}^0_V}))^{\lambda}
 \end{displaymath}
 is a sum of terms of the form
 \begin{displaymath}
  \op{H}^j(\widehat{N}^0_V,H_{\widehat{\pi^{-1}(V)}} \mathcal O(n))^{\lambda}
 \end{displaymath}
 with $n \in [t(\mathfrak{K})+1,\infty)$.
 
 Write $V = \op{Spec} R$. Then,
 \begin{displaymath}
  \widehat{N}^0_V = \op{Spf }R[[z_1,\ldots,z_d,u_1,\ldots,u_c]] =: \op{Spf} R[[\bm{z},\bm{u}]]
 \end{displaymath}
 where $\widehat{\pi^{-1}(V)}$ corresponds to the ideal $(\bm{u})$. The action of $\lambda$ on $z_j$ has negative degrees while the action on $u_j$ has positive degrees. A classical computation of Serre \cite{SerreFAC} shows that
 \begin{displaymath}
  \op{H}^j(\widehat{N}^0_V,H_{\widehat{\pi^{-1}(V)}} \mathcal O) = \begin{cases} \bigoplus_{l_1,\ldots,l_c >0} R[[\bm{z}]]\left(\frac{1}{\bm{u}^{\bm{l}}}\right) & j = c \\ 0 & \text{otherwise}. \end{cases}
 \end{displaymath}
 All elements of $\op{H}^j(\widehat{N}^0_V,H_{\widehat{\pi^{-1}(V)}} \mathcal O)$ have $\lambda$-degrees $ \leq t(\mathfrak{K})$. Thus, 
 \begin{displaymath}
  \op{H}^j(\widehat{N}^0_V,H_{\widehat{\pi^{-1}(V)}} \mathcal O(n))^{\lambda} = 0
 \end{displaymath}
 for $n \in [t(\mathfrak{K})+1,\infty)$. 
\end{proof}

Let $R$ be Noetherian. Assume we have an action of $\mathbb{G}_m$ on 
\begin{displaymath}
 T := \op{Spf }R[z_1,\ldots,z_d][[u_1,\ldots,u_c]] = \op{Spf }R[\bm{z}][[\bm{u}]]
\end{displaymath}
such that degrees of $z_i$ are negative, the degrees of the $u_i$ are positive, and $\op{Spec }R$ is pointwise fixed. Let $Z = \op{Spec }R[\bm{z}]$ and let $\widehat{T}$ be the completion of $R[\bm{z}][[\bm{u}]]$ along the ideal $(\bm{z})$.

\begin{lemma} \label{lemma: completion reduction, loc. cohom ok}
 We have
 \begin{displaymath}
  \op{H}^j(\widehat{T},\widehat{H_Z \mathcal O}(n))^{\mathbb{G}_m} = \op{H}^j(\widehat{T},H_{\widehat{Z}} \mathcal O(n))^{\mathbb{G}_m} = \op{H}^j(T,H_Z \mathcal O(n))^{\mathbb{G}_m}
 \end{displaymath}
 for all $n,j \in \Z$.
\end{lemma}

\begin{proof}
 Again using \cite{SerreFAC}, we have
 \begin{displaymath}
  \op{H}^j(\widehat{T},H_{\widehat{Z}} \mathcal O) = \begin{cases} \bigoplus_{l_1,\ldots,l_c >0} R[[\bm{z}]]\left(\frac{1}{\bm{u}^{\bm{l}}}\right) & j = c \\ 0 & \text{otherwise} \end{cases}
 \end{displaymath}
 and
 \begin{displaymath}
  \op{H}^j(T,H_{Z} \mathcal O) = \begin{cases} \bigoplus_{l_1,\ldots,l_c >0} R[\bm{z}]\left(\frac{1}{\bm{u}^{\bm{l}}}\right) & j = c \\ 0 & \text{otherwise}. \end{cases}
 \end{displaymath}

 Let $Y$ be the subscheme corresponding to the ideal $(\bm{z})$. For any $n$, since the degrees of $\bm{z}$ and $1/\bm{u}$ are strictly negative, there exists an $m(n)$ such that the maps,
 \begin{displaymath}
  \left(\bigoplus_{l_1,\ldots,l_c >0} R[[\bm{z}]]\left(\frac{1}{\bm{u}^{\bm{l}}}\right)\right)_n \to \left(\bigoplus_{l_1,\ldots,l_c >0} \frac{R[\bm{z}]}{(\bm{z})^{m+1}}\left(\frac{1}{\bm{u}^{\bm{l}}}\right)\right)_n \to \left(\bigoplus_{l_1,\ldots,l_c >0} \frac{R[\bm{z}]}{(\bm{z})^{m}}\left(\frac{1}{\bm{u}^{\bm{l}}}\right)\right)_n,
 \end{displaymath}
 are isomorphisms for $m \geq m(n)$. Thus,
 \begin{align*}
  \op{H}^j(\widehat{T},\widehat{H_Z \mathcal O}(n))^{\mathbb{G}_m} & = \lim_m \op{H}^j(T,(H_Z \mathcal O/\mathcal I_Y^m) (n))^{\mathbb{G}_m} \\
  & = \op{H}^j(T,(H_Z \mathcal O/\mathcal I_Y^{m(n)}) (n))^{\mathbb{G}_m} \\
  & = \op{H}^j(T,H_Z \mathcal O(n))^{\mathbb{G}_m} \\
  & = \op{H}^j(\widehat{T},H_{\widehat{Z}} \mathcal O(n))^{\mathbb{G}_m}.
 \end{align*}
\end{proof}

Now, assume that, in addition, we have a $\mathbb{G}_m$ invariant regular function, $w \in \op{H}^0(T,\mathcal O(d))^{\mathbb{G}_m}$. Let $\mathcal E$ and $\mathcal F$ be two free finite-rank factorizations in $\dcoh{[T/\mathbb{G}_m],w}$. Let $w$ also denote the corresponding regular function on $\widehat{T}$.

\begin{lemma} \label{lemma: completion reduction, ff}
 The map,
 \begin{displaymath}
  \op{Hom}_{[T/\mathbb{G}_m],w}(\mathcal E, H_Z \mathcal F) \to \op{Hom}_{[\widehat{T}/\mathbb{G}_m],w}(\widehat{\mathcal E}, \widehat{H_Z \mathcal F}),
 \end{displaymath}
 is an isomorphism.
\end{lemma}

\begin{proof}
 This is true if we show that the maps, 
 \begin{displaymath}
  \op{Hom}_{[T/\mathbb{G}_m]}(\mathcal E^r, H_Z \mathcal F^s[l]) \to \op{Hom}_{[\widehat{T}/\mathbb{G}_m]}(\widehat{\mathcal E^r}, \widehat{H_Z \mathcal F^s}[l]),
 \end{displaymath}
 are isomorphisms for all $r,s,l \in \Z$. However, the latter computation reduces to Lemma \ref{lemma: completion reduction, loc. cohom ok}. 
\end{proof}

\begin{remark}
 After completion of the original version of this article, we were informed by Halpern-Leistner that Lemma \ref{lemma: vanishing local cohom} was essentially already proven by C. Teleman in \cite{Tel}. In fact, Teleman's argument does not require the existence of a $G$-invariant affine open cover of $S_{\lambda}^0$. Using Teleman's argument, one can reinterpret the window, $\weezer_{\lambda,I}$, as those factorizations in $\dbcoh{[X/G],w}$ whose (derived) restriction to $Z_{\lambda}^0$ have $\lambda$-weights in $I$. Using this definition and Teleman's result/arguments allows one to avoid the requirement that $S_{\lambda}^0$ admits a $G$-invariant affine open cover, see \cite{HL12} for the details of such an argument. 
\end{remark}

\subsection{Essential surjectivity} \label{section: essential surjectivity}

In the previous section, we studied the fullness and faithfulness of $i^* : \dcoh{[X/G],w} \to \dcoh{[X_{\lambda}/G],w|_{X_{\lambda}}}$ on windows. In this section, we address essential surjectivity. First, we describe how to decompose $P(\lambda)$-equivariant sheaves on $Z_{\lambda}$.

\begin{lemma} \label{lemma: filtration on weights on Z lambda}
 Assume that $\mu(\mathcal L, \lambda,x) = 0$ for any $x \in Z_{\lambda}^0$. Let $\mathcal F$ be a locally-free factorization in $\dcoh{[Z_{\lambda}/P(\lambda)],w|_{Z_{\lambda}}}$. For each, $l \in \Z$, there exists a short exact sequence of locally-free $P(\lambda)$-equivariant factorizations,
 \begin{displaymath}
  0 \to \mathcal F_{<l} \to \mathcal F \to \mathcal F/\mathcal F_{<l} \to 0,
 \end{displaymath}
 where $\mathcal F_{<l}$ is, locally on $Z_{\lambda}$, $\lambda$-equivariantly quasi-isomorphic to a factorization whose components are of the form, $\bigoplus_{t < l} \mathcal O(t)^{m_t}$, and $\mathcal F/\mathcal F_{<l}$ is locally $\lambda$-equivariantly quasi-isomorphic to a factorization whose components are of the form, $\bigoplus_{t \geq l} \mathcal O(t)^{m_t}$.
\end{lemma}

\begin{proof}
 As a consequence of Proposition \ref{proposition: BB}, there exists an open affine cover, $\{V_j\}_{j \in J}$, and $\lambda$-equivariant isomorphisms,
 \begin{displaymath}
  \phi_j: \pi^{-1}(V_j) \to V_j \times \mathbb{A}^d 
 \end{displaymath}
 where $\lambda$ acts linearly on the $\mathbb{A}^d$ and trivially on the $V_j$ factor. Consider the overlap map,
\begin{displaymath}
 \phi_{j_2} \circ \phi_{j_1}^{-1} : (V_{j_1} \cap V_{j_2}) \times \mathbb{A}^d. 
\end{displaymath}
 Let $\op{Spec} R = V_{j_1} \cap V_{j_2}$ and consider the induced ring endomorphism,
\begin{displaymath}
 \psi : R[z_1,\ldots,z_d] \to R[z_1,\ldots,z_d]. 
\end{displaymath}
 As each $\phi_j$ takes $Z_{\lambda}^0 \cap V_j$ to $V_j \times \{0\}$, we can write
\begin{displaymath}
 \psi(z_j) = \sum_l r_{jl} z_l + \sum_{u,v} r_{juv} z_r z_s + \cdots.
\end{displaymath}
 Since each $\phi_j$ is $\lambda$-equivariant, $\psi$ commutes with the $\lambda$ action on $\mathbb{A}^d$. Thus, $\psi$ must take $z_j$ to a polynomial with total degree equal to that of $z_j$. 
 
 Shrinking the cover if necessary, we can assume that, for each component $\mathcal F^r$, there is a $\lambda$-equivariant isomorphism,
 \begin{displaymath}
  \mathcal F^r|_{\pi^{-1}(V_j)} \cong \bigoplus_{t \in \Z} \mathcal O(t)^{\oplus m_t}.
 \end{displaymath}
 For each $t$, we get a subsheaf of $\mathcal F^r|_{\pi^{-1}(V_j)}$ corresponding to $\mathcal O(t)^{\oplus m_t}$. These glue together because the overlap maps, as noted above, preserve the $\lambda$-degree. Denote the resulting sheaf as $(\mathcal F^r)_t$. Thus, $\lambda$-equivariantly, we have a splitting,
 \begin{displaymath}
  \mathcal F^r = \bigoplus_{t \in \Z} (\mathcal F^r)_t.
 \end{displaymath}
 We set 
 \begin{displaymath}
  \mathcal F^r_{<l} := \bigoplus_{t < l} (\mathcal F^r)_t.
 \end{displaymath}
 Restricted to $\pi^{-1}(V_j)$, each morphism in the factorization, $\phi^r: \mathcal F^{r-1} \to \mathcal F^r$, is a matrix whose components have non-positive degrees. As $\phi^r$ must be $\lambda$-equivariant, $\phi^r(\mathcal F^{r-1}_{<l}) \subset \mathcal F^r_{<l}$. Thus, we get a locally-free subfactorization,
 \begin{displaymath}
  \mathcal F_{< l} : = (\mathcal F^r_{<l}, \phi^r|_{\mathcal F^{r-1}_{<l}}).
 \end{displaymath}

 Next, we claim that $\mathcal F_{< l}$ inherits the $P(\lambda)$-equivariant structure of $\mathcal F$. This follows from the following claim. Fix $p \in P(\lambda)$ and $x \in Z_{\lambda}$. We have a map,
 \begin{displaymath}
  \sigma_p: \op{V}(\mathcal F^r)_{x} \to \op{V}(\mathcal F^r)_{\sigma(p,x)},
 \end{displaymath}
 which, we claim, preserves the filtration induced by $\mathcal F^r_{<l}$. Recall that passing to the geometric vector bundle is a contravariant functor. Thus, the quotient, $\mathcal F^r \to \mathcal F^r/\mathcal F^r_{<l}$, becomes a sub-bundle, $\op{V}(\mathcal F^r/\mathcal F^r_{<l}) \subset \op{V}(\mathcal F^r)$. The sub-bundle, $\op{V}(\mathcal F^r/\mathcal F^r_{<l})$, consists of the elements of the fibers that have $\lambda$-weight $\geq l$.
 
 Choose $\lambda$-eigenbases of $\op{V}(\mathcal F^r_x) \cong \mathbb{A}^f$ and $\op{V}(\mathcal F^r_{\sigma(p,x)}) \cong \mathbb{A}^f$. We get a linear map,
 \begin{displaymath}
  A_p: \mathbb{A}^f \to \mathbb{A}^f,
 \end{displaymath}
 corresponding to the action of $p$ on $\mathcal F^r$. We need to check this map preserves the sub-bundle of weights $\geq l$, or, in other words, does not decrease the $\lambda$-weights. The condition that the limit, $\lim_{\alpha \to 0} \lambda(\alpha) p \lambda(\alpha)^{-1}$, exists easily translates into exactly this fact. Thus, the filtration is preserved and we may conclude that $\mathcal F_{<l}$ is $P(\lambda)$-equivariant. 
\end{proof}

\begin{proposition} \label{proposition: essential surjectivity}
 Fix $d \in \Z$. Assume that $S_{\lambda}^0$ admits a $G$-invariant affine open cover and $\mu(\mathcal L,\lambda,x) = 0$ for $x \in Z_{\lambda}^0$. The functor,
 \begin{displaymath}
  i^*: \weezer_{[d+t(\mathfrak{K})+1,d]} \to \dcoh{[X_{\lambda}/G],w|_{X_{\lambda}}},
 \end{displaymath}
 given by restricting $i^*$ to the subcategory, $\weezer_{[d+t(\mathfrak{K})+1,d]}$, is essentially surjective. 
\end{proposition}
\begin{proof}
 We will argue the case $d=-t(\mathfrak{K})-1$. The arguments for other $d$ are identical.
 
 The functor,
 \begin{displaymath}
  i^*: \dcoh{[X/G],w} \to \dcoh{[X_{\lambda}/G],w|_{X_{\lambda}}},
 \end{displaymath}
 is essentially surjective and has kernel exactly those factorizations whose components are supported on $S_{\lambda}$. To demonstrate that 
 \begin{displaymath}
  i^*: \weezer_{[d+t(\mathfrak{K})+1,d]} \to \dcoh{[X_{\lambda}/G],w|_{X_{\lambda}}},
 \end{displaymath}
 is essentially surjective we will show that for any factorization, $\mathcal F \in \dcoh{[X/G],w}$, there exists a sequence of factorizations, $\mathcal F = \mathcal F_0,\mathcal F_1,\ldots,\mathcal F_r$, and exact triangles,
 \begin{displaymath}
  \mathcal T_i[-1] \to \mathcal F_{i+1} \to \mathcal F_i \to \mathcal T_i,
 \end{displaymath}
 for $0 \leq i \leq r-1$ with $\mathcal T_i$ supported on $S_{\lambda}$ and $\mathcal F_r \in \weezer_{[d+t(\mathfrak{K})+1,d]}$. When arguing this, as a piece of terminology, we say that $\mathcal F_{i+1}$ replaces $\mathcal F_i$ if we have a triangle as above.
 
 We have the exact sequence, 
 \begin{displaymath}
  0 \to \mathcal I_{S_{\lambda}} \to \mathcal O_X \overset{\varepsilon_1}{\to} \mathcal O_{S_{\lambda}} \to 0,
 \end{displaymath}
 where $\mathcal I_{S_{\lambda}}$ is the ideal sheaf of $S_{\lambda}$ with its induced equivariant structure. Taking (derived) $\mathcal O_X$-duals, we get an exact triangle
 \begin{displaymath}
  \mathcal I_{S_{\lambda}}^{\vee} \leftarrow \mathcal O_X \leftarrow \mathcal O_{S_{\lambda}}^{\vee} \leftarrow \mathcal I_{S_{\lambda}}^{\vee}[-1].
 \end{displaymath}
 Tensoring either with $\mathcal F$, we see that we may replace $\mathcal F$ by $\mathcal F \overset{\mathbf{L}}{\otimes} \mathcal I_{S_{\lambda}}$ or $\mathcal F \overset{\mathbf{L}}{\otimes} \mathcal I_{S_{\lambda}}^{\vee}$. 
 
 Let us first show that $\mathcal I_{S_{\lambda}} \in \weezer_{[t(\mathfrak{K}),-1]}$. Choose an affine open subset, $V \cong \op{Spec }R$, of $Z_{\lambda}^0$ and let $\widehat{N}^0_{V_j}$ be the corresponding affine subset of $\widehat{N}^0 = \widehat{(N_{S_{\lambda}^0|X})|_{Z_{\lambda}^0}}$. We can choose an isomorphism,
 \begin{displaymath}
  \widehat{N}^0_{V_j} = \op{Spf }R[[z_1,\ldots,z_d,u_1,\ldots,u_c]] =: \op{Spf }R[[\bm{z},\bm{u}]]
 \end{displaymath}
 with the $\lambda$-degrees of $z_i$ negative, the $\lambda$-degrees of $u_i$ positive, the $\lambda$-degrees of $R$ zero. The ideal sheaf, $\mathcal I_{S_{\lambda}}$, restricted to $\widehat{N}^0_{V_j}$, corresponds to the $(\bm{u}) := (u_1,\ldots,u_c)$.
 
 To replace $(\bm{u})$ by a quasi-isomorphic bounded complex of locally-free $\lambda$-equivariant sheaves we may take the Koszul complex on the ideal $(\bm{u}) $. It is immediate that components of this Koszul complex have weights lying in the interval, $[t(\mathfrak{K}),-1]$.
 
 Using Lemma \ref{lemma: weights of tensor and dual}, there exists an $m$ such that $\mathcal F \overset{\mathbf{L}}{\otimes} (\mathcal I_{S_{\lambda}}^{\vee})^{\overset{\mathbf{L}}{\otimes} m}$ lies in $\weezer_{[0,\infty)}$. We may replace $\mathcal F$ by $\mathcal F \overset{\mathbf{L}}{\otimes} (\mathcal I_{S_{\lambda}}^{\vee})^{\overset{\mathbf{L}}{\otimes} m}$ and assume that $\mathcal F \in \weezer_{[0,\infty)}$ for the remainder of the argument. There is a minimal $l$ such that $\mathcal F \in \weezer_{[0,l]}$. If $l < -t(\mathfrak{K})$, then we are done. Assume otherwise. 
 
 The short exact sequence,
 \begin{equation} \label{equation: weight splitting}
  0 \to \mathcal F|_{Z_{\lambda}}^{<l} \to\mathcal  F|_{Z_{\lambda}} \to \mathcal F|_{Z_{\lambda}}/(\mathcal F|_{Z_{\lambda}}^{<l}) \to 0,
 \end{equation}
 of Lemma \ref{lemma: filtration on weights on Z lambda} corresponds to a short exact sequence of $G$-equivariant locally-free sheaves on $S_{\lambda}$ by Theorem \ref{theorem: Thomason}. As stated in Theorem \ref{theorem: Thomason}, the equivalence, $\op{coh}([S_{\lambda}/G]) \cong \op{coh}([Z_{\lambda}/P(\lambda)])$, is the restriction along the inclusion, $Z_{\lambda} \to S_{\lambda}$. To set some notation, let us denote the inverse functor to the restriction along the inclusion, $Z_{\lambda} \to S_{\lambda}$, by $T$. The short exact sequence of Equation \ref{equation: weight splitting} corresponds to 
 \begin{displaymath}
  0 \to T((\mathcal F|_{Z_{\lambda}})_{>l}) \to \mathcal F|_{S_{\lambda}} \overset{\varepsilon_2}{\to} \mathcal F|_{S_{\lambda}}/T((\mathcal F|_{Z_{\lambda}})_{<l}) \to 0.
 \end{displaymath}
 
 Let, $\mathcal K_l(\mathcal F)$, be the kernel of $\varepsilon_2 \circ (\varepsilon_1 \otimes \mathcal F): \mathcal F \to \mathcal F|_{S_{\lambda}}/T(\mathcal F|_{Z_{\lambda}}^{<l})$ so that we have a short exact sequence,
 \begin{displaymath}
  0 \to \mathcal K_l(\mathcal F) \to \mathcal F \to \mathcal F|_{S_{\lambda}}/T((\mathcal F|_{Z_{\lambda}})_{<l}) \to 0.
 \end{displaymath}
 
 We claim that $\mathcal K_l(\mathcal F) \in \weezer_{[0,l-1]}$. Before proceeding with the proof of this claim, note that the validity of the claim allows us to produce our desired sequence of factorizations, $\mathcal \mathcal F_{i+1} = \mathcal K_{l-i}(\mathcal F_i)$, since $\mathcal F|_{S_{\lambda}}/T((\mathcal F|_{Z_{\lambda}})_{<s}))$ is supported on $S_{\lambda}$ for any $s \in \Z$. 
 
 Next, we proceed with the proof that $\mathcal K_l(\mathcal F) \in \weezer_{[0,l-1]}$. Since $\mathcal F \in \weezer_{[0,l]}$, by definition, we may find an open affine cover, $\{V_j\}_{j \in J}$, of $Z_{\lambda}^0$ and get an associated affine open cover, $\{\widehat{N}^0_{V_j}\}_{j \in J}$, of $\widehat{N}^0 = \widehat{(N_{S_{\lambda}^0|X})|_{Z_{\lambda}^0}}$ such that $\mathcal F|_{\widehat{N}^0_{V_j}}$ is $\lambda$-equivariantly quasi-isomorphic to a factorization each of whose components are of the form, $\bigoplus_{j \in [0,l]} \mathcal O(j)^{\oplus m_j}$. 
 
 It is straightforward to check that, if $\mathcal F'$ is a free factorization on $\widehat{N}^0_{V_j}$ that is $\lambda$-equivariantly quasi-isomorphic to $\mathcal F|_{\widehat{V}_j}$, then there is a $\lambda$-equivariant quasi-isomorphism,
 \begin{displaymath}
  \mathcal K_l(\mathcal F)|_{\widehat{N}^0_{V_j}} \cong \mathcal K_l'(\mathcal F'),
 \end{displaymath}
 where $\mathcal K_l'(\mathcal F')$ fits in the short exact sequence,
 \begin{displaymath}
  0 \to \mathcal K_l'(\mathcal F') \to \mathcal F' \to \mathcal F'|_{\widehat{\pi^{-1}(V_j)}}/(\mathcal F'|_{\widehat{\pi^{-1}(V_j)}})_{<l} \to 0.
 \end{displaymath}
 Thus, we may assume that the components of the $\mathcal F|_{\widehat{N}^0_{V_j}}$ are actually isomorphic to $\lambda$-equivariant modules of the form, $\bigoplus_{j \in [0,l]} \mathcal O(j)^{\oplus m_j}$. 
 
 Let us first treat the case where $\mathcal F$ is a single vector bundle, $\mathcal F'=\bigoplus_{j \in [0,l]} \mathcal O(j)^{m_j}$. Then,
 \begin{displaymath}
  \mathcal K_l'(\mathcal F') \cong \bigoplus_{j \in [0,l-1]} \mathcal O(j)^{\oplus m_j} \oplus \mathcal I_{Z_{\lambda} \cap \widehat{V}_j} \mathcal O(l)^{\oplus m_l}.
 \end{displaymath}
 It is immediate that $\bigoplus_{j \in [0,l-1]} \mathcal O(j)^{m_j} \in \weezer_{[0,l-1]}$. We can use the Koszul resolution of $\mathcal I_{Z_{\lambda}}$, which we have already observed lies in $\weezer_{[t(\mathfrak{K}),-1]}$, and tensor with $\mathcal O(l)^{m_l}$ to get a complex lying in $\weezer_{[l+t(\mathfrak{K}),l-1]}$. We have assumed that $l+t(\mathfrak{K}) > 0$ so we are done.
 
 Now, for a factorization, we may apply Proposition \ref{lemma: strictification} to convert a resolution of the components into a resolution of the factorization.
\end{proof}

\subsection{Wall contributions} \label{section: wall contributions}

In this section, we provide a comparison between $\weezer_I$ and $\weezer_{I'}$ for $I \subseteq I' \subset \Z$. 

\begin{definition}
 Let $G_{\lambda}$ be the quotient group, $C(\lambda)/\lambda$, let $Y_{\lambda}$ be the quotient stack, $[Z_{\lambda}^0/G_{\lambda}]$, and let $w_{\lambda} = w|_{Z_{\lambda}^0}$.
\end{definition}

Let us first start by studying $\dcoh{[Z^0_{\lambda}/C(\lambda)],w_{\lambda}}$. 

\begin{lemma} \label{lemma: lambda-splitting on Z lambda0}
 Let $\mathcal E$ be an object of $\dqcoh{[Z^0_{\lambda}/C(\lambda)],w_{\lambda}}$. There exists a functorial $\lambda$-equivariant splitting,
 \begin{displaymath}
  \mathcal E \cong \bigoplus_{l \in \Z} \mathcal E_l,
 \end{displaymath}
 where each $\mathcal E_l$ is a factorization with components that are locally isomorphic to $\mathcal O(l)^{m_l}$. 
\end{lemma}

\begin{proof}
 As any object of $\dqcoh{[Z^0_{\lambda}/C(\lambda)],w_{\lambda}}$ is isomorphic to a locally-free factorizations, it suffices to exhibit the splitting for locally-free factorizations. We first assume that $\mathcal E$ is locally-free. We may view $Z_{\lambda}^0$ as glued from affine open subsets, $\{V_j\}_{j\in J}$, for which $\mathcal E^r|_{V_j}$ is free as a $\lambda$-equivariant sheaf for all $r$. We define $(\mathcal E^r|_{V_j})_l$ to be the weight $l$ summand of $\mathcal E^r|_{V_j}$ and we claim that the collection, $(\mathcal E^r|_{V_j})_l$, glues to a locally-free sheaf on $Z_{\lambda}^0$. The gluing maps for $V_{j_1} \cap V_{j_2}$ are all $\lambda$-weight zero since $\lambda$ acts trivially on $Z_{\lambda}^0$. Thus, the gluing maps preserve the splitting by $\lambda$-weight and the $(\mathcal E^r|_{V_j})_l$ glue to a $\lambda$-equivariant sheaf which we denote by $(\mathcal E^r)_l$. 
 
 Next, we need to check that $(\mathcal E^r)_l$ is $C(\lambda)$-equivariant. We claim that the $C(\lambda)$ equivariant structure on $\mathcal E^r$ descends $\mathcal E^r_l$ for any $l$. To verify this claim, it suffices to look at fibers over points. Fix $x \in Z_{\lambda}^0$ ang $g \in G$. From the $C(\lambda)$-equivarint structure of $\mathcal E^r$, we get a linear map of vector spaces,
 \begin{displaymath}
  A_g: \op{V}(\mathcal E^r)_x \to \op{V}(\mathcal E^r)_{\sigma(g,x)}.
 \end{displaymath}
 As $g \lambda = \lambda g$, $A_g$ must preserve $\lambda$-weight. This proves the claim and allows us to split any locally-free coherent sheaf on $[Z^0_{\lambda}/C(\lambda)]$. 
 
 Each morphism in the factorization, $\phi^r: \mathcal E^{r-1} \to \mathcal E^r$, is $\lambda$-equivariant so must also commute with the $\lambda$-splitting. Thus, we get a splitting,
 \begin{displaymath}
  \mathcal E = \bigoplus_{l \in \Z} \mathcal E_l
 \end{displaymath}
 of the factorization itself. The functoriality of the splitting also follows from the $\lambda$-equivariance of any morphism.
\end{proof}

\begin{definition}
 Fix $l \in \Z$, we set
 \begin{displaymath}
  \dcoh{[Z^0_{\lambda}/C(\lambda)],w_{\lambda}}_l = \{ \mathcal E \in \dcoh{[Z^0_{\lambda}/C(\lambda)],w_{\lambda}} \mid \mathcal E = \mathcal E_l \}.
 \end{displaymath}
 We call $\dcoh{[Z^0_{\lambda}/C(\lambda)],w_{\lambda}}_l$ the \textbf{category of factorizations of weight $l$ on $Y_{\lambda}$}.
\end{definition}

\begin{lemma} \label{lemma: wall compositions are same}
 We have an equivalence,
 \begin{displaymath}
  \dcoh{Y_{\lambda},w_{\lambda}} \cong \dcoh{[Z_{\lambda}^0/C(\lambda)],w_{\lambda}}_0.
 \end{displaymath}

 Assume that there there is a character, $\chi: C(\lambda) \to \mathbb{G}_m$, such that
 \begin{displaymath}
  \chi \circ \lambda(\alpha) = \alpha^l.
 \end{displaymath}
 Then, twisting by $\chi$ provides an equivalence,
 \begin{displaymath}
  \dcoh{[Z^0_{\lambda}/C(\lambda)],w_{\lambda}}_r  \cong \dcoh{[Z^0_{\lambda}/C(\lambda)],w_{\lambda}}_{r+l},
 \end{displaymath}
 for any $r \in \Z$.
\end{lemma}

\begin{proof}
 A quasi-coherent sheaf on $Y_{\lambda}=[Z_{\lambda}^0/G_{\lambda}]$ is a quasi-coherent $C(\lambda)$-equivariant sheaf on $Z_{\lambda}^0$ for which $\lambda$ acts trivially, i.e. of $\lambda$-weight zero. The last statement is clear.
\end{proof}

\sidenote{{\color{red} Too many i's for inclusions...}}
Let $i: S_{\lambda} \to X$ denote the inclusion. Recall that $T$ stands for the Thomason equivalence of Corollary \ref{corollary: Thomason}. Now consider the composition of the following functors,
\begin{align*}
 \nu_l & : \dcoh{[Z^0_{\lambda}/C(\lambda)],w_{\lambda}}_l \to \dcoh{[Z^0_{\lambda}/C(\lambda)],w_{\lambda}} \\
 \pi^* & : \dcoh{[Z^0_{\lambda}/C(\lambda)],w_{\lambda}} \to \dcoh{[Z_{\lambda}/P(\lambda)],w|_{Z_{\lambda}}} \\
 T & : \dcoh{[Z_{\lambda}/P(\lambda)],w|_{Z_{\lambda}}} \to \dcoh{[S_{\lambda}/G],w|_{S_{\lambda}}} \\
 i_* & : \dcoh{[S_{\lambda}/G],w|_{S_{\lambda}}} \to \dcoh{[X/G],w}.
\end{align*}
Call the composition,
\begin{displaymath}
 \Upsilon_l :=  i_* \circ T \circ \pi^* \circ \nu_l : \dcoh{[Z^0_{\lambda}/C(\lambda)],w_{\lambda}}_l \to \dcoh{[X/G],w}.
\end{displaymath}

\begin{lemma} \label{lemma: wall-contribution fully-faithful}
 Assume that $S^0_{\lambda}$ admits a $G$-invariant affine open cover. For any $l \in \Z$, $\Upsilon_l$ is fully-faithful.
\end{lemma}

\begin{proof}
 By Lemma \ref{lemma: ff on dbcoh -> ff on dcoh}, we may instead prove that the corresponding functor on bounded derived categories is fully-faithful.

 We show that each functor involved in defining $\Upsilon_l$ is fully-faithful on the image of the previous composition of functors. This is clear for $\nu_l$ and $T$, so it remains to check $\pi^*$ and $i_*$. 
 
 We start with $\pi^*$. We have an adjunction,
 \begin{displaymath}
  \op{Hom}_{[Z_{\lambda}/P(\lambda)]}(\pi^* \mathcal E, \pi^* \mathcal F) \cong \op{Hom}_{[Z^0_{\lambda}/C(\lambda)]}(\mathcal E, \pi_*\pi^* \mathcal F),
 \end{displaymath}
 and a unit morphism,
 \begin{displaymath}
  \mathcal F \to \pi_*\pi^* \mathcal F.
 \end{displaymath}
 We reduce to checking that application of the unit morphism yields an isomorphism,
 \begin{displaymath}
  \op{Hom}_{[Z^0_{\lambda}/C(\lambda)]}(\mathcal E, \mathcal F) \to \op{Hom}_{[Z^0_{\lambda}/C(\lambda)]}(\mathcal E, \pi_*\pi^* \mathcal F).
 \end{displaymath}
 We may compute $\op{Hom}$ for $[Z^0_{\lambda}/C(\lambda)]$ by computing $\op{Hom}$ for $Z_{\lambda}^0$ and applying the functor of $C(\lambda)$-invariants, as $C(\lambda)$ is reductive by Lemma \ref{lemma: C part and U part of P}. Thus, the following claim finishes the proof of fully-faithfulness for $\pi^* \circ \nu_l$. We claim the application of the unit morphism yields an isomorphism,
 \begin{displaymath}
  \op{Hom}_{Z^0_{\lambda}}(\mathcal E, \mathcal F )^{\lambda} \to \op{Hom}_{Z^0_{\lambda}}(\mathcal E, \pi_*\pi^* \mathcal F)^{\lambda},
 \end{displaymath}
 of $\lambda$-invariant vector spaces.
 
 We may reduce further to checking this claim for the components of $\mathcal E$ and $\mathcal F$ so we may assume that $\mathcal E$ and $\mathcal F$ are vector bundles on $Z_{\lambda}^0$. As an additional reduction, we work locally on $Z_{\lambda}^0$. So, take an open affine subset, $V$, of $Z_{\lambda}^0$ such that $\pi^{-1}(V) \cong V \times \mathbb{A}^d$.  Then, 
 \begin{displaymath}
  \pi_* \pi^* \mathcal F \cong \mathcal F \otimes_k k[z_1,\ldots,z_d]
 \end{displaymath}
 where each of the $z_j$ has negative $\lambda$-degrees. Consequently, 
 \begin{align*}
  \op{Hom}_V(\mathcal E, \pi_*\pi^* \mathcal F) & \cong \op{Hom}_V(\mathcal E, \mathcal F[z_1,\ldots,z_d]) .
 \end{align*}
 Since the weight of $\mathcal E$ is concentrated at $l$, only maps to the weight $l$ component of $\mathcal F[z_1,\ldots,z_d]$ will survive application of $\lambda$-invariants. But, 
 \begin{displaymath}
  (\mathcal F[z_1,\ldots,z_d])_{l} = \mathcal F,
 \end{displaymath}
 as $\mathcal F$ is in the image of $\nu_l$ and each of the $z_j$ have negative $\lambda$-degrees. The claim is proven.
 
 We next turn to $i_*$. The computation of 
 \begin{displaymath}
  \op{Hom}_{[X/G]}(i_* \mathcal E, i_* \mathcal F)
 \end{displaymath}
 only depends on the formal neighborhood of $S_{\lambda}$ in $X$. Thus, we may reduce to the completion of $X$ along $S_{\lambda}$. Denote the completion by $\widehat{S}_{\lambda}$. We may work locally on $\widehat{S}_{\lambda}$. Take an open cover, $\{U_j\}_{j\in J}$, of $S^0_{\lambda}$ and consider the open cover, $\{\tilde{\pi}^{-1}(U_j)\}_{j\in J}$, of $S_{\lambda}$. Let $\{\widehat{\tilde{\pi}^{-1}(U_j)}\}_{j\in J}$ be the corresponding cover of $\widehat{S}_{\lambda}$. There is an isomorphism, $\widehat{\tilde{\pi}^{-1}(U_j)} \cong \widehat{N}_{\tilde{\pi}^{-1}(U_j)}$ where $\widehat{N}_{\tilde{\pi}^{-1}(U_j)}$ is the completion of the normal bundle of $S_{\lambda}$ in $X$ restricted to $\tilde{\pi}^{-1}(U_j)$ along the zero section. Applying the Thomason equivalence, Corollary \ref{corollary: Thomason}, we have an isomorphism,
 \begin{gather*}
  \op{Hom}_{[\widehat{N}_{\tilde{\pi}^{-1}(U_j)}/G]}(i_* \mathcal E, i_* \mathcal F) \cong \op{Hom}_{[\widehat{N}_{\pi^{-1}(V_j)}/P(\lambda)]}(i'_* \mathcal E|_{\pi^{-1}(V_j)}, i'_* \mathcal F|_{\pi^{-1}(V_j)})
 \end{gather*}
 where $V_j = U_j \cap Z_{\lambda}^0$, $\widehat{N}_{\pi^{-1}(V_j)}$ is the completion of the restriction of $N_{S_{\lambda}|X}$ to $\pi^{-1}(V_j)$ along the zero section, and
 \begin{displaymath}
  i': \pi^{-1}(V_j) \to \widehat{N}_{\pi^{-1}(V_j)}
 \end{displaymath}
 is the inclusion. 
 
 We then have an adjunction,
 \begin{gather*}
  \op{Hom}_{[\widehat{N}_{\pi^{-1}(V_j)}/P(\lambda)]}(i'_* \mathcal E|_{\pi^{-1}(V_j)}, i'_* \mathcal F|_{\pi^{-1}(V_j)}) \cong \op{Hom}_{[\pi^{-1}(V_j)/P(\lambda)]}(\mathbf{L}i'^{*}i'_* \mathcal E|_{\pi^{-1}(V_j)}, \mathcal F|_{\pi^{-1}(V_j)}).
 \end{gather*}
 
 As $\widehat{\tilde{\pi}^{-1}(U)}$ is affine and $G$ is reductive, coherent $G$-equivariant sheaves on $\widehat{\tilde{\pi}^{-1}(U)}$ have no higher cohomology. Appealing to Theorem \ref{theorem: Thomason}, the higher derived functors of $P(\lambda)$-invariants must vanish on the representations furnished by global sections of coherent sheaves on $\widehat{(N_{S_{\lambda}|X})|_{\pi^{-1}(V)}}$.
 
 To compute
 \begin{displaymath}
  \op{Hom}_{[\pi^{-1}(V_j)/P(\lambda)]}(\mathbf{L}i'^{*}i'_* \mathcal E|_{\pi^{-1}(V_j)}, \mathcal F|_{\pi^{-1}(V_j)})
 \end{displaymath}
 we may compute
 \begin{displaymath}
  \op{Hom}_{\pi^{-1}(V_j)}(\mathbf{L}i'^{*}i'_* \mathcal E|_{\pi^{-1}(V_j)}, \mathcal F|_{\pi^{-1}(V_j)})
 \end{displaymath}
 and then apply the functor of $P(\lambda)$-invariants. It is therefore sufficient that 
 \begin{displaymath}
  \op{Hom}_{\pi^{-1}(V_j)}(\mathcal E|_{\pi^{-1}(V_j)}, \mathcal F|_{\pi^{-1}(V_j)})^{\lambda} \to \op{Hom}_{\pi^{-1}(V_j)}(\mathbf{L}i'^{*}i'_* \mathcal E|_{\pi^{-1}(V_j)}, \mathcal F|_{\pi^{-1}(V_j)})^{\lambda} 
 \end{displaymath}
 is an isomorphism. Shrinking $V_j$ further, we may assume that $\mathcal E|_{\pi^{-1}(V_j)}$ and $\mathcal F|_{\pi^{-1}(V_j)}$ are free. 
 
 By Lemma \ref{lemma: ff on dbcoh -> ff on dcoh} it suffices to prove the statement for components, so we again reduce, now by assuming that $\mathcal E = \mathcal F = \pi^*(\mathcal O_{\pi^{-1}(V_j)}(l))$. We have a $\lambda$-equivariant quasi-isomorphism,
 \begin{displaymath}
  \mathbf{L}i'^{*}i'_* \mathcal O_{\pi^{-1}(V_j)}(l) \cong \bigoplus_p \bigwedge\nolimits^{p} (\mathcal N^{\vee}_{S_{\lambda}|X})|_{\pi^{-1}(V_j)}(l)[-p].
 \end{displaymath}
 Then,
 \begin{displaymath}
  \op{Hom}^p_{\pi^{-1}(V_j)}(\mathbf{L}i'^{*}i'_* \mathcal E|_{\pi^{-1}(V_j)}, \mathcal F|_{\pi^{-1}(V_j)}) \cong  \op{H}^0(\pi^{-1}(V_j),\bigwedge\nolimits^{-p} (\mathcal N_{S_{\lambda}|X})|_{\pi^{-1}(V_j)}).
 \end{displaymath}
 As all degrees of $(\mathcal N_{S_{\lambda}|X})|_{\pi^{-1}(V_j)}$ are negative and the degrees of all coordinates on $\pi^{-1}(V_j)$ are non-positive, we have an isomorphism,
 \begin{displaymath}
  \op{H}^0(\pi^{-1}(V_j),\bigoplus_p \bigwedge\nolimits^{-p} (\mathcal N_{S_{\lambda}|X})|_{\pi^{-1}(V_j)})^{\lambda} \cong \op{H}^0(V_j,\mathcal O_{V_j}),
 \end{displaymath}
 which is compatible with the isomorphism,
 \begin{displaymath}
  \op{Hom}_{\pi^{-1}(V_j)}(\mathcal E|_{\pi^{-1}(V_j)}, \mathcal F|_{\pi^{-1}(V_j)})^{\lambda} \cong \op{H}^0(V_j,\mathcal O_{V_j})[0],
 \end{displaymath}
 under application of the counit. Thus, $i_*$ is fully-faithful, and, as a consequence, so is $\Upsilon_l$.
\end{proof}

\begin{lemma} \label{lemma: image window of wall contribution}
 The essential image of
 \begin{displaymath}
  \Upsilon_l: \dcoh{[Z_{\lambda}^0/C(\lambda)],w_{\lambda}}_l \to \dcoh{[X/G],w}
 \end{displaymath}
 lies in $\weezer_{[l+t(\mathfrak{K}),l]}$.
\end{lemma}

\begin{proof}
 Let $\mathcal E \in \dcoh{Y_{\lambda},w_{\lambda}}$. Fix an affine open cover, $\{V_j\}_{j \in J}$, of $Z_{\lambda}^0$ such that $\mathcal E|_{V_j}$ is $\lambda$-equivariantly quasi-isomorphic to a free factorization. Let $\{\widehat{N}^0_{V_j}\}_{j \in J}$ be the corresponding open cover of $\widehat{N}^0 = \widehat{(N|_{S^0_{\lambda}|X})|_{Z_{\lambda}}}$. Let $i': \widehat{\pi^{-1}(V_j)} \to \widehat{N}^0_{V_j}$ be the inclusion. To $\lambda$-equivariantly resolve the components of
 \begin{displaymath}
  \Upsilon_l(\mathcal E) = i_*(T(\pi^* \nu_l \mathcal E))|_{\widehat{N}^0_{V_j}} \cong i'_* (\pi^* \nu_l \mathcal E)|_{\widehat{N}^0_{V_j}}
 \end{displaymath}
 by free factorizations on $\widehat{N}^0_{V_j}$ we may tensor with the Koszul resolution of $\widehat{\pi^{-1}(V_j)}$ in $\widehat{N}^0_{V_j}$. As noted previously, the terms of the Koszul resolution have weights between $[t(\mathfrak{K}),0]$. Thus, our resolved components have weights lying in $[l+t(\mathfrak{K}),l]$. We may use Proposition \ref{lemma: strictification} to replace $\Upsilon_l(\mathcal E)|_{\widehat{N}^0_{V_j}}$ with $\lambda$-equivariantly quasi-isomorphic factorization whose weights lie in $[l+t(\mathfrak{K}),l]$.
\end{proof}

Finally, we have the following result.

\begin{proposition} \label{proposition: wall contribution}
 Assume that $S^0_{\lambda}$ admits a $G$-invariant affine cover. Assume that $v-u > -t(\mathfrak{K})$. There is a semi-orthogonal decomposition,
 \begin{displaymath}
  \weezer_{[u,v]} = \langle \Upsilon_{v}(\dcoh{[Z_{\lambda}^0/C(\lambda)],w_{\lambda}}_v), \weezer_{[u,v-1]} \rangle.
 \end{displaymath}
\end{proposition}

\begin{proof}
 We may reduce to the case that $u=0$. We first check that the image of $\Upsilon_v$ is right orthogonal to $\weezer_{[0,v-1]}$. Lemma \ref{lemma: image window of wall contribution} and Corollary \ref{corollary: fully-faithfulness} immediately imply this. 
 
 Next, recall from Lemma \ref{lemma: filtration on weights on Z lambda} that, for any $\mathcal E$, we have an exact sequence of factorizations,
 \begin{displaymath}
  0 \to \mathcal K_v(\mathcal E) \to \mathcal E \to \mathcal E|_{S_{\lambda}}/T(\mathcal E|_{Z_{\lambda}}^{<v}) \to 0,
 \end{displaymath}
 with $\mathcal K_v(\mathcal E) \in \weezer_{[0,v-1]}$ and $(\mathcal E|_{S_{\lambda}}/T(\mathcal E|_{Z_{\lambda}}^{<v}))|_{Z_{\lambda}}$ is $\lambda$-equivariantly quasi-isomorphic, locally on $Z_{\lambda}$, to a factorization whose weights are concentrated at $v$. Once we show that $E|_{S_{\lambda}}/T(\mathcal E|_{Z_{\lambda}}^{<v})$ lies in the essential image of $\Upsilon_v$, the proof is finished. 
 
 Take a factorization $\mathcal F \in \dcoh{[Z_{\lambda}/P(\lambda)],w|_{Z_{\lambda}}}$ whose weights are concentrated at $v$ and consider $\pi_*\mathcal F$. We have an inclusion,
 \begin{displaymath}
  (\pi_*\mathcal F)_v \to \pi_*\mathcal F,
 \end{displaymath}
 which, by adjunction, gives a map
 \begin{displaymath}
  \pi^*(\pi_*\mathcal F)_v \to \mathcal F. 
 \end{displaymath}
 We claim this is a quasi-isomorphism. It suffices to work locally and check the maps on the components of $\mathcal F$ are isomorphisms. We take an an open cover, $\{V_j\}_{j \in J}$, of $Z_{\lambda}^0$, restrict to $\pi^{-1}(V_j)$, and assume that $\mathcal F$ is $\mathcal O_{\pi^{-1}(V_j)}(v)^{m_v}$. In this case, a simple computation in coordinates shows that 
 \begin{displaymath}
  \pi^*(\pi_*\mathcal F)_v \to \mathcal F. 
 \end{displaymath}
 is an isomorphism. Taking $\mathcal F = (\mathcal E|_{S_{\lambda}}/T(\mathcal E|_{Z_{\lambda}}^{<v}))|_{Z_{\lambda}}$ finishes the proof that the image of $\Upsilon_v$ and $W_{[0,v-1]}$ generate $W_{[0,v]}$. By Proposition \ref{proposition: characterizations of SODs}, we get the desired semi-orthogonal decomposition.
\end{proof}

\begin{corollary} \label{corollary: wall contribution}
  Assume that $S^0_{\lambda}$ admits a $G$-invariant affine cover. Assume that $v-u > -t(\mathfrak{K}) + s$. There is a semi-orthogonal decomposition,
 \begin{displaymath}
  \weezer_{[u,v]} = \langle \Upsilon_{v-s}(\dcoh{[Z_{\lambda}^0/C(\lambda)],w_{\lambda}}_{v-s}),\ldots,\Upsilon_{v}(\dcoh{[Z_{\lambda}^0/C(\lambda)],w_{\lambda}}_v), \weezer_{[u,v-s-1]} \rangle.
 \end{displaymath}
\end{corollary}

\begin{proof}
 This is an immediate consequence of iterated applications of Proposition \ref{proposition: wall contribution}.
\end{proof}

\subsection{Varying stratifications and derived categories} \label{section: final statement}

\begin{definition} \label{definition: matched wall crossing}
 Let $X$ be a smooth, quasi-projective variety equipped with a $G$ action. An \textbf{elementary wall crossing} is a pair of elementary HKKN stratifications, $\mathfrak{K}^+$ and $\mathfrak{K}^-$,
 \begin{align*}
  X & = X_+ \sqcup S_{\lambda} \\
  X & = X_- \sqcup S_{-\lambda},
 \end{align*}
 corresponding to a single one-parameter subgroup, $\lambda: \mathbb{G}_m \to G$, and the same connected component of the fixed locus, $Z_{\lambda}^0 = Z_{-\lambda}^0$. We let
 \begin{displaymath}
  \mu := -t(\mathfrak{K}^+) + t(\mathfrak{K}^-).
 \end{displaymath}
\end{definition}

% \begin{definition} \label{definition: matched wall crossing}
%  Let $X$ be a smooth, quasi-projective variety equipped with a $G$ action. A \textbf{matched wall crossing} is a pair of HKKN stratafications, $\mathfrak{K}^-$ and $\mathfrak{K}^+$,
%  \begin{align*}
%   X & \supset X_1^+ \supset \cdots \supset X_p^+ =: X_+ \\
%   X & \supset X_1^- \supset \cdots \supset X_p^- =: X_-,
%  \end{align*}
%  such that $-\lambda^-_j = \lambda^+_{\rho(j)}$ and $Z^0_{\lambda_j^-} = Z^0_{\lambda^+_{\rho(j)}}$ for all $j$ and some $p$-permutation, $\rho$. We will often state that $X_-$ and $X_+$ are related by a matched wall crossing and we will often set
%  \begin{align*}
%   S^+_j & := S_{\lambda^+_j} \\
%   S^-_j & := S_{\lambda^-_{\rho(j)}}.
%  \end{align*} 
%  We say that a matched wall crossing is an \textbf{elementary wall crossing} if $p=1$. We will say that a matched wall crossing is \textbf{ordered} if $\rho = 1$.
% \end{definition}
\sidenote{{\color{blue} It looks like you commented out the general thing, meaning that we only use elementary HKKN stratifications at the moment, despite the fact that we defined the general thing.} {\color{red} Yup. No good applications for the general one.}}
\begin{theorem} \label{theorem: elementary wall crossing}
 Let $X$ be a smooth, quasi-projective variety equipped with the action of a reductive linear algebraic group, $G$. Let $w \in \op{H}^0(X,\mathcal L)^G$ be a $G$-invariant section of a $G$-line bundle, $\mathcal L$, and assume that $\mu(\mathcal L,\lambda,x) = 0$ for $x \in Z_{\lambda}^0$. 
 
 Assume we have an elementary wall crossing, $(\mathfrak{K}^+,\mathfrak{K}^-)$, 
 \begin{align*}
  X & = X_+ \sqcup S_{\lambda} \\
  X & = X_- \sqcup S_{-\lambda}.
 \end{align*}
 Assume that $S^0_{\lambda}$ admits a $G$ invariant affine open cover. Fix $d \in \Z$. 
 
 \begin{enumerate}
  \item If $t(\mathfrak{K}^+) < t(\mathfrak{K}^-)$, then there are fully-faithful functors,
  \begin{displaymath}
   \Phi^+_d: \dcoh{[X_-/G],w|_{X_-}} \to \dcoh{[X_+/G],w|_{X_+}},
  \end{displaymath}
  and, for $-t(\mathfrak{K}^-) + d \leq j \leq -t(\mathfrak{K}^+) + d -1$,
  \begin{displaymath}
   \Upsilon_j^+: \dcoh{[Z_{\lambda}^0/C(\lambda)],w_{\lambda}}_j \to \dcoh{[X_+/G],w|_{X_+}},
  \end{displaymath}
  and a semi-orthogonal decomposition,
  \begin{displaymath}
   \dcoh{[X_+/G],w|_{X_+}} = \langle \Upsilon^+_{-t(\mathfrak{K}^-)+d}, \ldots, \Upsilon^+_{-t(\mathfrak{K}^+)+d-1}, \Phi^+_d \rangle.
  \end{displaymath}
  \item If $t(\mathfrak{K}^+) = t(\mathfrak{K}^-)$, then there is an exact equivalence,
  \begin{displaymath}
   \Phi^+_d: \dcoh{[X_-/G],w|_{X_-}} \to \dcoh{[X_+/G],w|_{X_+}}.
  \end{displaymath}
  \item If $t(\mathfrak{K}^+) > t(\mathfrak{K}^-)$, then there are fully-faithful functors,
  \begin{displaymath}
   \Phi^-_d: \dcoh{[X_+/G],w|_{X_+}} \to \dcoh{[X_-/G],w|_{X_-}},
  \end{displaymath}
  and, for $-t(\mathfrak{K}^+) + d \leq j \leq -t(\mathfrak{K}^-) + d -1$,
  \begin{displaymath}
   \Upsilon_j^-: \dcoh{[Z_{\lambda}^0/C(\lambda)],w_{\lambda}}_j \to \dcoh{[X_-/G],w|_{X_-}},
  \end{displaymath}
  and a semi-orthogonal decomposition,
  \begin{displaymath}
   \dcoh{[X_-/G],w|_{X_-}} = \langle \Upsilon^-_{-t(\mathfrak{K}^+)+d}, \ldots, \Upsilon^-_{-t(\mathfrak{K}^-)+d-1}, \Phi^-_d \rangle.
  \end{displaymath}
 \end{enumerate}
\end{theorem}

\begin{proof}
 The third case follows from the first by switching, $+ \leftrightarrow -$. We will assume that $t(\mathfrak{K}^+) \leq t(\mathfrak{K}^-)$ and prove the first and second cases at the same time.

 From Corollary \ref{corollary: fully-faithfulness} and Proposition \ref{proposition: essential surjectivity}, we know that the inclusions,
 \begin{align*}
  i_+ & : X_+ \to X \\
  i_- & : X_- \to X,
 \end{align*}
 induce equivalences, via restriction,
 \begin{align*}
  i_+^* & : \weezer_{\lambda,[d,-t(\mathfrak{K}^+)+d-1]} \to \dcoh{[X_+/G],w|_{X_+}} \\
  i_-^* & : \weezer_{-\lambda,[t(\mathfrak{K}^-)-d+1,-d]} \to \dcoh{[X_-/G],w|_{X_-}}.
 \end{align*}
 Note that $\weezer_{\lambda,I} = \weezer_{-\lambda,-I}$ so we may directly compare windows for $\lambda$ and $-\lambda$. In particular, we have an inclusion,
 \begin{displaymath}
  i_d: \weezer_{-\lambda,[t(\mathfrak{K}^-)-d+1,-d]} \to \weezer_{\lambda,[d,-t(\mathfrak{K}^+)+d-1]},
 \end{displaymath}
 as $t(\mathfrak{K}^+) \leq t(\mathfrak{K}^-)$. We define
 \begin{displaymath}
  \Phi_d^+ : = i_+^* \circ i_d \circ (i_-^*)^{-1} : \dcoh{[X_-/G],w|_{X_-}} \to \dcoh{[X_+/G],w|_{X_+}}.
 \end{displaymath}
 From Corollary \ref{corollary: wall contribution}, we have a semi-orthogonal decomposition,
 \begin{displaymath}
  \weezer_{\lambda,[d,-t(\mathfrak{K}^+)+d-1]} = \langle \Upsilon_{-t(\mathfrak{K}^-)+d}, \ldots, \Upsilon_{-t(\mathfrak{K}^+)+d-1}, \weezer_{-\lambda,[t(\mathfrak{K}^-)-d+1,-d]} \rangle
 \end{displaymath}
 from the fully-faithful functors,
 \begin{displaymath}
  \Upsilon_l: \dcoh{[Z_{\lambda}^0/C(\lambda)],w_{\lambda}}_l \to \weezer_{\lambda,[t(\mathfrak{K}^+)+l,l]} \subset \dcoh{[X/G],w}.
 \end{displaymath}
 We set 
 \begin{displaymath}
  \Upsilon_l^+ := i_+^* \circ \Upsilon_l : \dcoh{[Z_{\lambda}^0/C(\lambda)],w_{\lambda}}_l \to \dcoh{[X_+/G],w|_{X_+}}
 \end{displaymath}
 to finish the proof.
\end{proof}

Finally, we give a lemma which allows one to more easily compute the quantity, $\mu = -t(\mathfrak{K}^+) + t(\mathfrak{K}^-)$.

\begin{lemma} \label{lemma: formula for mu}
 Let
 \begin{align*}
  X & = X_+ \sqcup S_{\lambda} \\
  X & = X_- \sqcup S_{-\lambda}
 \end{align*}
 be an elementary wall crossing. Assume that $G_x \subset C(\lambda)$ for $x \in Z_{\lambda}^0$. Then,
 \begin{displaymath}
  \mu := -t(\mathfrak{K}^+) + t(\mathfrak{K}^-) = \mu(\omega_X^{-1},\lambda,x).
 \end{displaymath}
%  
%  In particular, 
%  \begin{itemize}
%   \item $t(\mathfrak{K}^+) < t(\mathfrak{K}^-)$ if and only if $\mu(\omega_X^{-1},\lambda,x) > 0$,
%   \item $t(\mathfrak{K}^+) = t(\mathfrak{K}^-)$ if and only if $\mu(\omega_X^{-1},\lambda,x) = 0$,
%   \item $t(\mathfrak{K}^+) > t(\mathfrak{K}^-)$ if and only if $\mu(\omega_X^{-1},\lambda,x) < 0$.
%  \end{itemize}
\end{lemma}

\begin{proof}
 Recall that 
 \begin{displaymath}
  t(\mathfrak{K}) = \mu(\bigwedge\nolimits^{\op{codim} S_{\lambda}} \mathcal N_{S_{\lambda}|X}^{\vee},\lambda,x)
 \end{displaymath}
 for some $x \in Z_{\lambda}^0$. Thus, 
 \begin{displaymath}
  -\mu = t(\mathfrak{K}^+) - t(\mathfrak{K}^-) = \mu(\bigwedge\nolimits^{\op{codim} S_{\lambda}} \mathcal N_{S_{\lambda}|X}^{\vee} \otimes \bigwedge\nolimits^{\op{codim} S_{-\lambda}} \mathcal N_{S_{-\lambda}|X}^{\vee}, \lambda, x).
 \end{displaymath}
 We may decompose the tangent space at $x$ as
 \begin{displaymath}
  \mathcal  T_{X,x} = \mathcal N_{S_{-\lambda}|X,x} \oplus \mathcal N_{S_{\lambda}|X,x} \oplus \mathcal T_{S_{\lambda}^0,x}.
 \end{displaymath}
 Taking duals and top exterior powers we get
 \begin{displaymath}
  \mu = \mu(\omega_X^{-1},\lambda,x) - \mu(\bigwedge\nolimits^{\op{dim} S_{\lambda}^0}\mathcal T_{S_{\lambda}^0},\lambda,x).
 \end{displaymath}
 It suffices to show that
 \begin{displaymath}
  \mu(\bigwedge\nolimits^{\op{dim} S_{\lambda}^0}\mathcal T_{S_{\lambda}^0},\lambda,x) = 0.
 \end{displaymath}
 Since $S_{\lambda}^0$ is the $G$ orbit of $Z_{\lambda}^0$, we have a surjection
 \begin{displaymath}
  d\sigma:  \mathcal T_{G,e} \oplus \mathcal T_{Z_{\lambda}^0,x} \to  \mathcal T_{S_{\lambda}^0,x}
 \end{displaymath}
 given by linearizing the action, $\sigma: G \times X \to X$, at $(e,x)$. The map, $d\sigma$ is $\lambda$-equivariant if we equip $\mathcal T_{G,e}$ with the adjoint action. Let $K$ be the kernel of $d\sigma$ so we have a short exact sequence,
 \begin{displaymath}
  0 \to K \to \mathcal T_{G,e} \oplus \mathcal T_{Z_{\lambda}^0,x} \overset{d\sigma}{\to} \mathcal T_{S_{\lambda}^0,x} \to 0.
 \end{displaymath}
 The Lie algebra, $\mathcal T_{G,e}$, is reductive.  Therefore it splits into its Abelian and semi-simple components.  The Abelian component has $\lambda$-weight zero. If $l$ is a weight from the semi-simple part, so is $-l$, see for example Lecture $14$ of \cite{FH}. Therefore, $\mu(\bigwedge\nolimits^{\op{dim} G} \mathcal T_{G,e},\lambda,e)$ is zero and $\lambda$ acts trivially on $Z_{\lambda}^0$. So,
 \begin{displaymath}
  \mu(\bigwedge\nolimits^{\op{dim} S_{\lambda}^0} \mathcal T_{S_{\lambda}^0}, \lambda, x) = - \mu( \bigwedge\nolimits^{\op{dim} K} K, \lambda, 0).
 \end{displaymath}
 We have an isomorphism,
 \begin{displaymath}
  K \cong \mathcal T_{G \cdot x \cap Z_{\lambda}^0,x} \oplus \mathcal T_{G_x,e}.
 \end{displaymath}
 All $\lambda$-weights of $\mathcal T_{G  \cdot x \cap Z_{\lambda}^0,x}$ are zero as, again, $\lambda$ acts trivially on $Z_{\lambda}^0$. All the $\lambda$-weights of $\mathcal T_{G_x,e}$ are zero as we have assumed that $\lambda$ commutes with $G_x$. % The final statement is evident.
\end{proof}

\section{Varying the GIT quotient and comparing derived categories} \label{sec: VGIT}

Our main application of Theorem \ref{theorem: elementary wall crossing} is in comparing the derived categories of GIT quotients obtained from different choices of equivariant line bundle, commonly known as VGIT.

\subsection{Background} \label{subsection: background VGIT and DC}

We let $X$ be a smooth, quasi-projective variety equipped an action of a reductive linear algebraic group, $G$. We recall some important notions surrounding equivariant line bundles.

\begin{definition}
 The group of $G$-equivariant lines bundles up to isomorphism is called the  \textbf{$G$-Picard group} and is denoted by $\op{Pic}^G(X)$. 
 
 Two equivariant line bundles $\mathcal L_1, \mathcal L_2 \in \op{Pic}^G(X)$ are said to be \textbf{$G$-algebraically equivalent} if there is a connected variety, $T$, an action of $G$ on $T \times X$ equivariant with respect to the projection, $T \times X \to X$, points $t_1,t_2 \in T$, and a $G$-line bundle, $\mathcal L$, on $T \times X$, such that $\mathcal L|_{t_1 \times X } \cong \mathcal L_1$ and $\mathcal L|_{t_2 \times X} \cong \mathcal L_2$. The quotient of $\op{Pic}^G(X)$ by $G$-algebraic equivalence is called the \textbf{$G$-Neron Severi group} and is denoted by $\op{NS}^G(X)$.
 
 We say that a $G$-equivariant line bundle, $\mathcal L$, is \textbf{amply $G$-effective} if $\mathcal L$ is ample on $X$ and $X^{\op{ss}}(\mathcal L) \not = \emptyset$. We denote the subset of $\op{NS}^G(X)$ consisting of amply $G$-effective line bundles by $C^G(X)$. We say two equivariant line bundles, $\mathcal L_1$ and $\mathcal L_2$, are \textbf{GIT-equivalent} if
 \begin{displaymath}
  X^{\op{ss}}(\mathcal L_1) = X^{\op{ss}}(\mathcal L_2).
 \end{displaymath}
 
 If the closures of the GIT-equivalence classes in $C^G(X)_{\R}$ form a fan, we call this fan the \textbf{GIT fan} for the action of $G$ on $X$. The top dimensional cones of the GIT fan are called \textbf{chambers} and the codimension one cones are called \textbf{walls}.
\end{definition}

The following is the main result of \cite{Res}, see also Section 3 of \cite{DH98} and Section 2 of \cite{Tha96}.

\begin{theorem} \label{theorem: GIT fan}
 Let $X$ be a smooth, projective variety equipped an action of a reductive linear algebraic group, $G$. The closures of the GIT-equivalence classes in $C^G(X)_{\R}$ form a fan.
\end{theorem}

\begin{proof}
 This is Theorem 5.2 of \cite{Res} building off Theorem 2.3 of \cite{Tha96} and Section 3 of \cite{DH98}.
\end{proof}

There is an analogous statement for affine space.

\begin{proposition} \label{proposition: affine GIT fan}
 Let $G$ be a reductive linear algebraic group acting linearly on $X:=\mathbb{A}^n$. The closures of the GIT-equivalence classes in $C^G(X)_{\R}$ form a fan.
\end{proposition}

\begin{proof}
 One compactifies $\mathbb{A}^n$ to $\mathbb{P}^n$, embeds $G$ into $\op{PGL}_{n+1}$ via the inclusion,
 \begin{align*}
 \op{GL}_n & \to \op{GL}_{n+1} \\
 A & \mapsto \begin{pmatrix} 1 & 0 \\ 0 & A \end{pmatrix},
\end{align*}
 and applies Theorem \ref{theorem: GIT fan}. This is carried out in \cite{Halic}.
\end{proof}

In our later applications, we will be interested in the behavior of the derived category under crossing a wall between adjacent chambers. We will consider a restricted class of variations described below.

\begin{definition}
 Let $X$ be a smooth and quasi-projective variety acted on by a reductive linear algebraic group, $G$. Let $\mathcal L_-$ and $\mathcal L_+$ be two $G$-equivariant line bundles which are both ample as line bundles. We will call the pair $(\mathcal L_-,\mathcal L_+)$ a \textbf{variation}. 
 
 For $t \in [-1,1]$, let $\mathcal L_t$ be the line, $\mathcal L_-^{\frac{1-t}{2}} \otimes \mathcal L_+^{\frac{1+t}{2}}$, in $\op{NS}^G(X)_{\R}$. We say that the variation satisfies the \textbf{DHT condition} if 
 \begin{itemize}
  \item The set, $X^{\op{ss}}(\mathcal L_t)$, is constant for $-1 \leq t < 0$ and for $0 < t \leq 1$. Denote the two subsets by $X^{\op{ss}}(-)$ and $X^{\op{ss}}(+)$. Set $X^{\op{ss}}(0) := X^{\op{ss}}(\mathcal L_0)$.
  \item For each $x \in X^{\op{ss}}(0) \setminus \left(X^{\op{ss}}(-) \cup X^{\op{ss}}(+) \right)$, the stabilizer of $x$ is isomorphic to $\mathbb{G}_m$.
  \item The set, $X^{\op{ss}}(0) \setminus \left(X^{\op{ss}}(-) \cup X^{\op{ss}}(+) \right)$, is connected.
 \end{itemize}
\end{definition}

We recall the following result from \cite{Tha96} and \cite{DH98} which is our main source of elementary wall crossings.

\begin{theorem} \label{theorem: VGIT give elementary wall crossing}
 Let $X$ be a smooth and projective variety acted on by a reductive linear algebraic group, $G$. Let $\mathcal L_-$ and $\mathcal L_+$ be two $G$-equivariant line bundles that are both ample as line bundles. Assume the variation satisfies the DHT condition. Then, there is a one-parameter subgroup, $\lambda: \mathbb{G}_m \to G$, a connected component, $Z_{\lambda}^0$, of the fixed locus, $(X^{\op{ss}}(0))^{\lambda}$, and an elementary wall crossing,
 \begin{align*}
  X^{\op{ss}}(0) & = X^{\op{ss}}(+) \sqcup S_{\lambda} \\
  X^{\op{ss}}(0) & = X^{\op{ss}}(-) \sqcup S_{-\lambda}.
 \end{align*}
\end{theorem}

\begin{proof}
 This is implicit in Proposition 4.6 of \cite{Tha96}. This is also Lemma 4.2.4 of \cite{DH98}. 
\end{proof}

\subsection{Comparing derived categories} \label{subsection: comparing derived categories}

We present some simple applications of Theorem \ref{theorem: elementary wall crossing} in light of the facts discussed in the previous section.

\begin{theorem} \label{theorem: VGIT and derived categories}
 Let $X$, $G$, $\mathcal L_+$, $\mathcal L_-$ satisfy the hypotheses of Theorem \ref{theorem: VGIT give elementary wall crossing} and let $w \in \op{H}^0(X,\mathcal L)^G$ for a $G$-line bundle, $\mathcal L$.  Let $\lambda: \mathbb{G}_m \to G$ be the one-parameter subgroup and $(\mathfrak{K}^+,\mathfrak{K}^-)$ be the elementary wall crossing guaranteed in the conclusions of Theorem \ref{theorem: VGIT give elementary wall crossing}. 
 
 Set $X \modmod{} - := [X^{\op{ss}}(-)/G]$ and $X \modmod{} + := [X^{\op{ss}}(+)/G]$. Assume $\mu(\mathcal L,\lambda,x) = 0$ for $x \in Z_{\lambda}^0$. Fix $d \in \Z$. 
 
 \begin{enumerate}
  \item If $t(\mathfrak{K}^+) < t(\mathfrak{K}^-)$, then there are fully-faithful functors,
  \begin{displaymath}
   \Phi^+_d: \dbcoh{X \modmod{} -, w|_{X_-}} \to \dbcoh{X \modmod{} +,w|_{X_+}},
  \end{displaymath}
  and, for $-t(\mathfrak{K}^-) + d \leq j \leq -t(\mathfrak{K}^+) + d -1$,
  \begin{displaymath}
   \Upsilon_j^+:  \dbcoh{[Z_{\lambda}^0/C(\lambda)],w_{\lambda}}_j \to \dcoh{X \modmod{} +,w|_{X_+}},
  \end{displaymath}
  and a semi-orthogonal decomposition,
  \begin{displaymath}
   \dcoh{X \modmod{} +,w|_{X_+}} = \langle \Upsilon^+_{-t(\mathfrak{K}^-)+d}, \ldots, \Upsilon^+_{-t(\mathfrak{K}^+)+d-1}, \Phi^+_d \rangle.
  \end{displaymath}
  \item If $t(\mathfrak{K}^+) = t(\mathfrak{K}^-)$, then there is an exact equivalence,
  \begin{displaymath}
   \Phi^+_d: \dcoh{X \modmod{} -,w|_{X_-}} \to \dcoh{X \modmod{} +,w|_{X_+}}.
  \end{displaymath}
  \item If $t(\mathfrak{K}^+) > t(\mathfrak{K}^-)$, then there are fully-faithful functors,
  \begin{displaymath}
   \Phi^-_d: \dcoh{X \modmod{} +,w|_{X_+}} \to \dcoh{X \modmod{} -,w|_{X_-}},
  \end{displaymath}
  and, for $-t(\mathfrak{K}^+) + d \leq j \leq -t(\mathfrak{K}^-) + d -1$,
  \begin{displaymath}
   \Upsilon_j^-:  \dcoh{[Z_{\lambda}^0/C(\lambda)],w_{\lambda}}_j \to \dcoh{X \modmod{} -,w|_{X_-}},
  \end{displaymath}
  and a semi-orthogonal decomposition,
  \begin{displaymath}
   \dcoh{X \modmod{} -,w|_{X_-}} = \langle \Upsilon^-_{-t(\mathfrak{K}^+)+d}, \ldots, \Upsilon^-_{-t(\mathfrak{K}^-)+d-1}, \Phi^-_d \rangle.
  \end{displaymath}
 \end{enumerate}
\end{theorem}

\begin{proof}
 Theorem \ref{theorem: VGIT give elementary wall crossing} guarantees that all of the hypotheses of Theorem \ref{theorem: elementary wall crossing}, except the $G$-invariant cover of $S^0_{\lambda}$, hold. Since semi-stable loci have $G$ invariant open affine covers by definition, the condition on $S_{\lambda}^0$ also holds.
\end{proof}

\begin{remark}
 We can treat the case of the derived categories of coherent sheaves on $X \modmod{} +$ and $X \modmod{} -$ by appealing to Corollary \ref{corollary: Isik}.
\end{remark}

\begin{remark} \label{remark: walls as GIT quotients}
Recall from Lemma~\ref{lemma: wall compositions are same} that there is a smooth stack, $Y_{\lambda} = [Z_{\lambda}^0/G_{\lambda}]$, and an equivalence,
\[
  \dcoh{Y_{\lambda},w_{\lambda}} \cong \dcoh{[Z_{\lambda}^0/C(\lambda)],w_{\lambda}}_0.
\]
Furthermore, if there is a character, $\chi: C(\lambda) \to \mathbb{G}_m$, such that
 \begin{displaymath}
  \chi \circ \lambda(\alpha) = \alpha,
 \end{displaymath}
Then by Lemma~\ref{lemma: wall compositions are same}
\[
  \dcoh{Y_{\lambda},w_{\lambda}} \cong \dcoh{[Z_{\lambda}^0/C(\lambda)],w_{\lambda}}_l.
\]
for all $l$.

The $Y_{\lambda}$ are themselves GIT quotients. Indeed, if we let $X^{\lambda}_0$ be the appropriate connected component of the fixed locus in $X$ of our one-parameter subgroup, $\lambda$, then, $Y_{\lambda} = X^{\lambda}_0 \modmod{\mathcal L_0} C(\lambda)$. See Remark 9.1 of \cite{Ness}. This provides a pleasant inductive structure to which we will later appeal.
\end{remark}

Another, slightly different, situation of import in this paper is the following.

\begin{proposition} \label{proposition: affine GIT and der cat}
 Let $G$ be a reductive linear algebraic group acting linearly on affine space, $X:=\mathbb{A}^n$. Let $\chi_-$ and $\chi_+$ be two characters of $G$. Let $\mathcal L_- := \mathcal O(\chi_-)$ and $\mathcal L_+ := \mathcal O(\chi_+)$ be the corresponding variation. Assume the variation satisfies the DHT condition. 
 
 Then, there exists a one-parameter subgroup, $\lambda: \mathbb{G}_m \to G$, and an elementary wall crossing $(\mathfrak{K}^+,\mathfrak{K}^-)$,
 \begin{align*}
  X^{\op{ss}}(0) & = X^{\op{ss}}(+) \sqcup S_{\lambda} \\
  X^{\op{ss}}(0) & = X^{\op{ss}}(-) \sqcup S_{-\lambda}.
 \end{align*}
 
 Let $w$ be a $G$ semi-invariant regular function on $X$ transforming under the character, $\chi$. Assume that $\chi \circ \lambda$ is constant. Set $X \modmod{} - := [X^{\op{ss}}(-)/G]$ and $X \modmod{} + := [X^{\op{ss}}(+)/G]$. Fix $d \in \Z$. 
 
 \begin{enumerate}
  \item If $t(\mathfrak{K}^+) < t(\mathfrak{K}^-)$, then there are fully-faithful functors,
  \begin{displaymath}
   \Phi^+_d: \dcoh{X \modmod{} - ,w_{-}} \to \dcoh{X \modmod{} +,w_{+}},
  \end{displaymath}
  and, for $-t(\mathfrak{K}^-) + d \leq j \leq -t(\mathfrak{K}^+) + d -1$,
  \begin{displaymath}
   \Upsilon_j^+: \dcoh{[Z_{\lambda}^0/C(\lambda)],w_{\lambda}}_j \to \dcoh{X \modmod{} +,w_{+}},
  \end{displaymath}
  and a semi-orthogonal decomposition,
  \begin{displaymath}
   \dcoh{X \modmod{} +,w_{+}} = \langle \Upsilon^+_{-t(\mathfrak{K}^-)+d}, \ldots, \Upsilon^+_{-t(\mathfrak{K}^+)+d-1}, \Phi^+_d \rangle.
  \end{displaymath}
  \item If $t(\mathfrak{K}^+) = t(\mathfrak{K}^-)$, then there is an exact equivalence,
  \begin{displaymath}
   \Phi^+_d: \dcoh{X \modmod{} -,w_{-}} \to \dcoh{X \modmod{} +,w_{+}}.
  \end{displaymath}
  \item If $t(\mathfrak{K}^+) > t(\mathfrak{K}^-)$, then there are fully-faithful functors,
  \begin{displaymath}
   \Phi^-_d: \dcoh{X \modmod{} +,w_{+}} \to \dcoh{X \modmod{} -,w_{-}},
  \end{displaymath}
  and, for $-t(\mathfrak{K}^+) + d \leq j \leq -t(\mathfrak{K}^-) + d -1$,
  \begin{displaymath}
   \Upsilon_j^-: \dcoh{[Z_{\lambda}^0/C(\lambda)],w_{\lambda}}_j \to \dcoh{X \modmod{} -,w_{-}},
  \end{displaymath}
  and a semi-orthogonal decomposition,
  \begin{displaymath}
   \dcoh{X \modmod{} -,w_{-}} = \langle \Upsilon^-_{-t(\mathfrak{K}^+)+d}, \ldots, \Upsilon^-_{-t(\mathfrak{K}^-)+d-1}, \Phi^-_d \rangle.
  \end{displaymath}
 \end{enumerate}
\end{proposition}

\begin{proof}
 One compactifies $\mathbb{A}^n$ to $\mathbb{P}^n$, embeds $G$ into $\op{PGL}_{n+1}$ via the inclusion,
 \begin{align*}
 \op{GL}_n & \to \op{GL}_{n+1} \\
 A & \mapsto \begin{pmatrix} 1 & 0 \\ 0 & A \end{pmatrix},
\end{align*}
 and applies Theorems \ref{theorem: VGIT give elementary wall crossing} and \ref{theorem: VGIT and derived categories}.
\end{proof}

\subsection{D-equivalence versus K-equivalence under elementary wall crossings}

\begin{definition}
 Let $Y_1$ and $Y_2$ be smooth projective varieties. We say that $Y_1$ \textbf{$D$-dominates} $Y_2$ if $Y_1$ and $Y_2$ are birational and $\dbcoh{Y_2}$ is equivalent to an admissible subcategory of $\dbcoh{Y_1}$. We write $Y_1 \geq_D Y_2$. We say that $Y_1$ and $Y_2$ are \textbf{$D$-equivalent} if they are birational and there is an equivalence, $\dbcoh{Y_1} \cong \dbcoh{Y_2}$. We write $Y_1 =_D Y_2$.
 
 We say that $Y_1$ \textbf{$K$-dominates} $Y_2$ if there exists a smooth projective variety, $Z$, and birational morphisms, $f_1: Z \to Y_1$ and $f_2: Z \to Y_2$, such that the difference of pullbacks of the canonical divisors is effective, $f_1^*K_{Y_1} - f_2^*K_{Y_2} \geq 0$. If $Y_1$ $K$-dominates $Y_2$, then we write $Y_1 \geq_K Y_2$.  We say that $Y_1$ and $Y_2$ are \textbf{$K$-equivalent} if $f_1^*K_{X_1}$ and $f_2^*K_{Y_2}$ are linearly equivalent. We write $Y_1 =_K Y_2$ for $K$-equivalence.
\end{definition}

In \cite{KawD-K}, Kawamata formulates the following conjecture.
\begin{conjecture} \label{conjecture: D=K}
 Two smooth projective varieties are $D$-equivalent if and only if they are $K$-equivalent.
\end{conjecture}
The above conjecture was proven for toroidal flips and toroidal divisorial contractions in \cite{Kaw05}. In addition, derived equivalence of birational Calabi-Yau threefolds is proven in \cite{Bridgeland-flops}. However, we thank Kawamata for pointing out the counterexample in \cite{Ue04} illustrating that $D$-equivalence may not always imply $K$-equivalence.  On the other hand, the full statement of the conjecture is true in many cases, including our situation.

We can combine the results on the birational geometry of elementary GIT wall crossings due to Thaddeus \cite{Tha96} and Dolgachev-Hu \cite{DH98} with Theorem \ref{theorem: VGIT and derived categories} to prove that $D$-domination and $K$-domination are equivalent when $Y_1$ and $Y_2$ are related by an elementary GIT wall crossing. 

Let $X$ be a smooth and proper variety acted on by a reductive linear algebraic group, $G$. Let $\mathcal L_-$ and $\mathcal L_+$ be two $G$-equivariant line bundles. Denote the the GIT quotients by
\begin{align*}
 X \modmod{} - & := X \modmod{\mathcal L_-} G \\
 X \modmod{} + & := X \modmod{\mathcal L_+} G.
\end{align*}

\begin{proposition} \label{proposition: D=K for elementary wall crossings}
 Assume that $G$ acts freely on both $X^{\op{ss}}(+)$ and $X^{\op{ss}}(-)$ and that $(\mathcal L_-,\mathcal L_+)$ is a DHT variation. Then, 
 \begin{enumerate}
  \item $X \modmod{} - \geq_K X \modmod{} +$ if and only if $\mu \leq 0$.
  \item $X \modmod{} + \geq_K X \modmod{} -$ if and only if $\mu \geq 0$.
 \end{enumerate}
\end{proposition}

\begin{proof}
 Let $\lambda$ be the one-parameter subgroup controlling the wall crossing of Theorem \ref{theorem: VGIT give elementary wall crossing}. By Proposition 4.9 of \cite{Tha96}, the weights of $\lambda$ on the normal bundle to $S_{\lambda}$ are all $-1$ and the weights of $\lambda$ on the normal bundle to $S_{-\lambda}$ are all $+1$. 

 Let $\bar{Z}$ be the good moduli space of $X \modmod{\mathcal L_0} G$. This is Mumford's GIT quotient, \cite{Alper}. Set
 \begin{displaymath}
  Z := X \modmod{} - \times_{\bar{Z}} X \modmod{} +,
 \end{displaymath}
 and denote the two projections by
 \begin{align*}
  f_- & : Z \to X \modmod{} - \\
  f_+ & : Z \to X \modmod{} +. 
 \end{align*}
 By Theorem 4.8 of \cite{Tha96}, $f_-$ is isomorphic to the blow up of $X \modmod{} -$ along $S_{\lambda} \modmod{} - := (S_{\lambda} \setminus S_{\lambda}^0) \modmod{\mathcal L_-} G$ and $f_+$ is isomorphic to the blow up of $X \modmod{} +$ along $S_{-\lambda} \modmod{} + := (S_{-\lambda} \setminus S_{\lambda}^0) \modmod{\mathcal L_-} G $. Thus,
 \begin{align*}
  f_-^*K_{X \modmod{} -} & = K_Z + (\op{codim}_{X \modmod{} -} (S_{\lambda} \modmod{} -) -1)E \\
  f_+^*K_{X \modmod{} +} & = K_Z + (\op{codim}_{X \modmod{} +} (S_{-\lambda} \modmod{} +) -1)E
 \end{align*}
 where $E$ is the common exceptional divisor of $f_-$ and $f_+$. Since $G$ acts freely on both semi-stable loci, we have
 \begin{align*}
  c_+ & := \op{codim}_{X \modmod{} -} (S_{\lambda} \modmod{} -) = \op{codim}_X S_{\lambda} \\
  c_- & := \op{codim}_{X \modmod{} +} (S_{-\lambda} \modmod{} +) = \op{codim}_X S_{-\lambda}.
 \end{align*}
 As the weights of $\lambda$ are all $\pm 1$, we have $\mu = c_+ - c_-$. Thus,
 \begin{displaymath}
  f_+^*K_{X \modmod{} +} - f_-^*K_{X \modmod{} -} = \mu E. 
 \end{displaymath}
 Thus, $\mu \geq 0$ implies that $X \modmod{} + \geq_K X \modmod{} -$ and $\mu \leq 0$ implies that $X \modmod{} + \leq_K X \modmod{} -$. 
 
 As $K$-domination is independent of the choice of birational resolution, $X \modmod{} - \leftarrow Z' \to X \modmod{} +$, if we assume that $X \modmod{} + \geq_K X \modmod{} -$, then $f_+^*K_{X \modmod{} +} - f_-^*K_{X \modmod{} -} \geq 0$ and $\mu \geq 0$. We get a similar statement with the roles reversed.  
\end{proof}

\begin{corollary} \label{corollary: D=K for elementary wall crossing}
 Propagating the assumptions from Proposition \ref{proposition: D=K for elementary wall crossings}, $X \modmod{} +$ $D$-dominates $X \modmod{} -$ if and only if $X \modmod{} +$ $K$-dominates $X \modmod{} -$.
\end{corollary}

\begin{proof}
 From Proposition \ref{proposition: D=K for elementary wall crossings} we know that $X \modmod{} + \geq_K X \modmod{} -$ if and only if $\mu \geq 0$. By Theorem \ref{theorem: VGIT and derived categories}, we know that $X \modmod{} - \geq_D X \modmod{} +$ if $\mu \geq 0$. Thus, $X \modmod{} + \geq_K X \modmod{} -$ implies $X \modmod{} + \geq_D X \modmod{} -$.
 
 Let us now assume that $X \modmod{} + \not \geq_K X \modmod{} -$. Again by Proposition \ref{proposition: D=K for elementary wall crossings}, this is true if and only if $\mu < 0$. Since $\mu < 0$, there is a semi-orthogonal decomposition,
 \begin{displaymath}
  \dbcoh{X \modmod{} -} = \langle \dbcoh{[Z_{\lambda}/C(\lambda)]}_0, \ldots, \dbcoh{[Z_{\lambda}/C(\lambda)]}_{-\mu-1}, \dbcoh{X \modmod{} +}.
 \end{displaymath}
 by Theorem \ref{theorem: VGIT and derived categories}. 
 
 Additivity of Hochschild homology under semi-orthogonal decomposition, Theorem 7.3 of \cite{Kuz09a}, gives that
 \begin{displaymath}
  \op{HH}_*(X \modmod{} -) \cong \op{HH}_*(X \modmod{} +) \oplus \bigoplus_{l=0}^{-\mu-1} \op{HH}_*(\dbcoh{[Z_{\lambda}^0/C(\lambda)]}_l).
 \end{displaymath}
 Recall that $\dbcoh{[Z_{\lambda}^0/C(\lambda)]}_0 \cong \dbcoh{Y_{\lambda}}$. Since we have a DHT variation, $G_{\lambda}$ acts freely on $Z_{\lambda}^0$ so, by Proposition \ref{proposition: when is the GIT quotient is a scheme}, $Y_{\lambda}$ is a variety and cannot have trivial Hochschild homology. Thus,
 \begin{displaymath}
  \op{dim} \op{HH}_*(X \modmod{} -) \gneq \op{dim} \op{HH}_*(X \modmod{} +).
 \end{displaymath}
 If we assume that $\dbcoh{X \modmod{} +}$ is an admissible subcategory of $\dbcoh{X \modmod{} -}$, we have
 \begin{displaymath}
  \op{dim} \op{HH}_*(X \modmod{} -) \leq \op{dim} \op{HH}_*(X \modmod{} +)
 \end{displaymath}
 from Theorem 7.3 of \cite{Kuz09a}. This gives a contradiction so $X \modmod{} + \not \geq_D X \modmod{} -$. 
\end{proof}

\begin{remark}
Corollary \ref{corollary: D=K for elementary wall crossing} subsumes previous work of Kawamata on Abelian groups, \cite{Kaw05}.
\end{remark}

Given the focus on LG-models in this paper, it makes sense to extend these questions to LG-models. We first give some definitions.

\begin{definition}
 Let $(Y_1,L_1,w_1)$ and $(Y_2,L_2,w_2)$ be two LG-models with $Y_1$ and $Y_2$ smooth. We say that $(Y_1,w_1)$ \textbf{$K$-dominates} $(Y_2,w_2)$ if there exists a smooth variety, $Z$, and proper birational morphisms, $f_1: Z \to Y_1$ and $f_2: Z \to Y_2$, such that 
 \begin{itemize}
  \item $f_1^*L_1 \cong f_2^*L_2$,
  \item $f_1^*w_1 = f_2^*w_2$ under the previous isomorphism,
  \item and $f_1^*K_{Y_1} - f_2^*K_{Y_2} \geq 0$.
 \end{itemize}
 If $f_1^*K_{Y_1} = f_2^*K_{Y_2}$ in addition, we say that $(Y_1,w_1)$ is \textbf{$K$-equivalent} to $(Y_2,w_2)$. 
\end{definition}

As the following conjecture arises in a straightforward manner from the ideas in \cite{KawD-K}, we call it Kawamata's LG-model conjecture.

\begin{conjecture} \label{conjecture: LG D and K}
 If $(Y_1,w_1)$ $K$-dominates $(Y_2,w_2)$, then $\dcoh{Y_2,w_2}$ is equivalent to an admissible subcategory of $\dcoh{Y_1,w_1}$. 
\end{conjecture}

As evidence, we verify this conjecture for elementary wall crossings in GIT. Let $X$ be a smooth and proper variety, or affine space, acted by a reductive linear algebraic group, $G$. Let $\mathcal L_-$ and $\mathcal L_+$ be two $G$-equivariant line bundles. Denote the the GIT quotients by
\begin{align*}
 X \modmod{} - & := X \modmod{\mathcal L_-} G \\
 X \modmod{} + & := X \modmod{\mathcal L_+} G.
\end{align*}
Let $w$ be a $G$-invariant section of a $G$-line bundle, $\mathcal L$, on $X$. Set
\begin{align*}
 w_- & := w|_{X^{\op{ss}}(-)} \\
 w_+ & := w|_{X^{\op{ss}}(+)}.
\end{align*}

\begin{proposition}
 Assume that $G$ acts freely on both $X^{\op{ss}}(+)$ and $X^{\op{ss}}(-)$ and that $(\mathcal L_-,\mathcal L_+)$ is a DHT variation. Let $\lambda$ be the one-parameter subgroup controlling the variation. Assume that $\mu(\mathcal L,\lambda, x) = 0$ for any $x \in Z_{\lambda}^0$. If $(X\modmod{} +,w_+)$ $K$-dominates $(X \modmod{} -,w_-)$, then $\dcoh{X \modmod{} -,w_-}$ is equivalent to an admissible subcategory of $\dcoh{X \modmod{} +,w_+}$.
\end{proposition}

\begin{proof}
 Let $\bar{Z}$ be the good moduli space of $X \modmod{\mathcal L_0} G$. Set
 \begin{displaymath}
  Z := X \modmod{} - \times_{\bar{Z}} X \modmod{} +,
 \end{displaymath}
 and denote the two projections by
 \begin{align*}
  f_- & : Z \to X \modmod{} - \\
  f_+ & : Z \to X \modmod{} +. 
 \end{align*} 
 It is easy to check that $f_1^*w_+ = f_2^*w_-$ and we have already seen that
 \begin{displaymath}
  f_+^*K_{X \modmod{} +} - f_-^*K_{X \modmod{} -} = \mu E,
 \end{displaymath}
 when $X$ is smooth and proper. If $X$ is affine space, $\mathbb{A}^n$, one may compactify to $\mathbb{P}^n$ and reuse the argument. Thus, $(X\modmod{} +,w_+)$ $K$-dominates $(X \modmod{} -,w_-)$ if and only if $\mu \geq 0$. Theorem \ref{theorem: VGIT and derived categories}, in the smooth and projective case, or Proposition \ref{proposition: affine GIT and der cat}, in the affine case, implies that $\dcoh{X \modmod{} -,w_-}$ is equivalent to an admissible subcategory of $\dcoh{X \modmod{} +,w_+}$.
\end{proof}

\begin{remark}
 There is a simple counterexample to the converse of Conjecture~\ref{conjecture: LG D and K}, i.e., there exist $K$-inequivalent LG-models, $(Y_1,w_1)$ and $(Y_2,w_2)$, with $\dcoh{Y_1,w_1} \cong \dcoh{Y_2,w_2}$. Moreover, one may assume that $Y_1=Y_2$ in this counterexample.
 
 Take $Y_1=Y_2=Y:= \op{V}(\mathcal O_{\P^1}(4))$ and let $f_1$ and $f_2$ be two homogeneous quartic polynomials in $k[x,y]$. Assume that $f_1$ and $f_2$ define smooth hypersurfaces in $\P^1$ and let $w_1$ and $w_2$ denote the regular functions on $Y$ corresponding to $f_1$ and $f_2$, respectively. For both $w_1$ and $w_2$, the singular locus consists of $4$ ordinary double points. Therefore, the idempotent completions of $\dcoh{Y_1,w_1}$ and $\dcoh{Y_2,w_2}$ are equivalent by \cite{OrlFC}. In this case, one can quickly verify that both categories are already idempotent complete so $\dcoh{Y_1,w_1} \cong \dcoh{Y_2,w_2}$.
 
 If $(Y,w_1)$ and $(Y,w_2)$ are $K$-equivalent, $f_1$ and $f_2$ must lie in the same $\op{GL}_2$-orbit. Taking $f_1$ and $f_2$ from different $\op{GL}_2$-orbits gives the counterexample. We thank Kuznetsov for this observation.

 We thank Stellari for pointing out a counterexample to the converse of Conjecture~\ref{conjecture: LG D and K} when one allows smooth Artin stacks: one can find homogeneous cubic polynomials, $f_1$ and $f_2$, in $k[x_0,\ldots,x_5]$ with 
 \begin{displaymath}
  \dcoh{[\mathbb{A}^6/\mathbb{G}_m], f_1} \cong \dcoh{[\mathbb{A}^6/\mathbb{G}_m], f_2}
 \end{displaymath}
 and with $f_1$ and $f_2$ defining non-isomorphic projective hypersurfaces so lying in different $\op{GL}_6$-orbits, see Remark 6.4 of \cite{BMMS}. It is interesting to note that such a counterexample does not exist for cubics in $5$ variables by Theorem 1.1 of \cite{BMMS}.
 
 We also thank Kawamata for his enlightening comments on the material in this section.
\end{remark}

\section{Derived categories of coherent sheaves on toric varieties} \label{section: toric}

In this section, we use Theorem \ref{theorem: elementary wall crossing} to study the derived categories of toric varieties and simplify a result of Kawamata on the existence of full exceptional collections, \cite{Kaw06,Kaw12}.

\subsection{Background on toric varieties and GIT}

\begin{definition}
 Let $G$ be a reductive linear algebraic group with a homomorphism, $\phi: G \to \mathbb{G}_m^n$, with finite kernel. We shall call a GIT quotient of $X: = \mathbb{A}^n$, with respect to the $G$ action, a \textbf{toric Deligne-Mumford stack} or toric DM-stack. We will say the stack is projective if $k[x_1,\ldots,x_n]^G = k$. If the quotient is represented by a scheme, we shall call the quotient a \textbf{toric variety}.
\end{definition}

\begin{remark}
 Note that, in general, a GIT quotient of affine space is automatically quasi-projective. In particular, if the quotient is proper, then it is projective. However, the most common definition of a toric DM-stack proceeds through a stacky fan, as in \cite{BCS}. This definition also includes non-quasi-projective toric stacks/varieties. At least in the case of varieties, our, more restrictive, definition does appear elsewhere in the literature, see \cite{MS} as an example. 
\end{remark}

The semi-stable locus is unaffected by finite extensions so we focus our discussion on the case where $G$ is a subgroup of $\mathbb{G}_m^n$. Since $G$ is Abelian, for any one-parameter subgroup, $\lambda: \mathbb{G}_m \to G$, we have $S_{\lambda} = Z_{\lambda}$ and $P(\lambda) = C(\lambda) = G$, which simplifies the analysis. Since we are on affine space, $\op{Pic}(X) = \op{NS}(X) = 1$, and $\op{Pic}^G(X) = \op{NS}^G(X) = \widehat{G}$, the space of characters of $G$. We have
\begin{displaymath}
 C^G(X) = \{\chi \in \widehat{G} \mid \exists f \in k[x_1,\ldots,x_n] \text{ with } f \not = 0 \text{ and } f(\sigma(g,x)) = \chi(g)^nf(x) \text{ for some } n > 0 \}.
\end{displaymath}

\begin{definition}
 From Proposition \ref{proposition: affine GIT fan}, there is a fan structure, $\Sigma_{\op{GKZ}}$, on $C^G(X) \otimes_{\Z} \R \subseteq \widehat{G} \otimes_{\Z} \R=:\widehat{G}_{\R}$. The fan, $\Sigma_{\op{GKZ}}$, predates the general GIT fan construction. It existence is due to I.M. Gel'fand, Kapranov, and A. Zelevinsky \cite{GKZ}. We shall therefore call it the \textbf{GKZ fan}. It is also often called the \textbf{secondary fan}.
 
 If $C$ is a chamber in $\Sigma_{\op{GKZ}}$ which intersects the closure of the complement of $C^G(X)_{\R}$, we shall call $C$ a \textbf{boundary chamber}. We shall also refer to the closure of the complement of $C^G(X)_{\R}$ as a chamber and call it the \textbf{empty chamber}.
\end{definition}

For a character, $\chi$, we will denote the semi-stable locus of $\mathcal O(\chi)$ by $X^{\op{ss}}(\chi)$ and the GIT quotient of $X$ by $\mathcal O(\chi)$ as $X \modmod{\chi} G$.

Recall that one-parameter subgroups and characters of $G$ are dual. Let $\Lambda(G)$ be the group of one-parameter subgroups of $G$. There is a perfect pairing,
\begin{align*}
 \langle \bullet, \bullet \rangle : \widehat{G} \times \Lambda(G) & \to \widehat{\mathbb{G}_m} \cong \Z \\
 (\chi, \lambda) & \mapsto \chi \circ \lambda. 
\end{align*}

\begin{proposition} \label{proposition: variation for toric var gives elem wall crossing}
 Assume that either $C_-$ and $C_+$ are adjacent chambers in $\Sigma_{\op{GKZ}}$ or $C_-$ is the empty chamber and $C_+$ is a boundary chamber. Let $\lambda$ be a one-parameter subgroup defining the hyperplane separating $C_-$ and $C_+$. Assume that $\lambda$ pairs non-negatively with $C_+$. Let $M_0$ be the wall separating $C_-$ and $C_+$. Let $\chi_0$ be a character in the interior of $M_0$, $\chi_-$ be a character in the interior of $C_-$, and $\chi_+$ be a character in the interior of $C_+$.
 
 There is an elementary wall crossing
 \begin{align*}
  X^{\op{ss}}(\chi_0) & = X^{\op{ss}}(\chi_+) \sqcup S_{\lambda} \\
  X^{\op{ss}}(\chi_0) & = X^{\op{ss}}(\chi_-) \sqcup S_{-\lambda}.
 \end{align*}
\end{proposition}

\begin{proof}
 If $C_-$ is the empty chamber, then $S_{-\lambda} = X^{\op{ss}}(\chi_0)$ as the $\lambda$-degree of all monomials is non-positive. We are left to check the cases where $\chi_+$ and $\chi_-$ lie in $C^G(X)$.

 By the Hilbert-Mumford numerical criterion, Proposition \ref{proposition: affine HM numerical stability}, $x \in X$ is unstable for $\mathcal O(\chi_{+})$ if and only if there exists a one-parameter subgroup, $\lambda': \mathbb{G}_m \to G$, such that $\lim_{\alpha \to 0} \sigma(\lambda'(\alpha),x) =: x^*$ exists and $\mu(\mathcal O(\chi_{+}),\lambda',x^*) > 0$. 
 
 Assume $x \in X^{\op{ss}}(\chi_0) \setminus X^{\op{ss}}(\chi_+)$. Then, by continuity of $\mu(\bullet,\lambda',x^*)$, we must have the inequality, $\mu(\mathcal O(\chi_{0}),\lambda',x^*) \geq 0$.  As $x \in X^{\op{ss}}(\chi_0)$, we must have $\mu(\mathcal O(\chi_{0}),\lambda',x^*) = 0$. If $\lambda' \not = \lambda$, $\chi_0$ lies in the intersection of the hyperplanes defined by $\lambda$ and $\lambda'$ in the GKZ fan. By Proposition \ref{proposition: affine GIT fan}, we may replace $\chi_0$ by another $\chi_0'$ in the interior of $M_0$ without affecting the semi-stable locus. As $C_0$ has codimension one, we can choose $\chi_0'$ to lie outside the hyperplane defined by $\lambda'$ and get a contradiction. Thus, $\lambda' = \lambda$ and
 \begin{displaymath}
   X^{\op{ss}}(\chi_0) \setminus X^{\op{ss}}(\chi_+) \subseteq S_{\lambda}.
 \end{displaymath}
 Similarly,
 \begin{displaymath}
   X^{\op{ss}}(\chi_0) \setminus X^{\op{ss}}(\chi_-) \subseteq S_{-\lambda}.
 \end{displaymath}
 
 It is straightforward to check that $S_{\lambda}$ is semi-stable for $\chi_0$ but not for $\chi_+$ while $S_{-\lambda}$ is semi-stable for $\chi_0$ but not for $\chi_-$. Thus, 
 \begin{align*}
  X^{\op{ss}}(\chi_0) \setminus X^{\op{ss}}(\chi_+) & \supseteq S_{\lambda} \\
  X^{\op{ss}}(\chi_0) \setminus X^{\op{ss}}(\chi_-) & \supseteq S_{-\lambda}.
 \end{align*}

 The subvarieties $S_{-\lambda}$ and $S_{\lambda}$, when defined on $X$, are linear subspaces of $X$. Thus, they are both closed as are their restrictions to $X^{\op{ss}}(\chi_0)$.
\end{proof}

\subsection{Derived categories of toric varieties} \label{subsection: derived categories of toric varieties}

We restate Theorem \ref{theorem: elementary wall crossing} in the case of toric DM stacks. Let $C_+$ and $C_-$ be adjacent chambers in $\Sigma_{\op{GKZ}}$ with $M_0$ be their common face contained in the hyperplane defined by a one-parameter subgroup, $\lambda$. Assume that $\lambda$ is normalized so that it pairs non-negatively with $C_+$. Fix $\chi_- \in C_-, \chi_0 \in M_0, \chi_+ \in C_+$ in their respective interiors. We have $G_{\lambda} = G/\lambda$ and $Y_{\lambda} = X^{\lambda} \modmod{\mathcal O(\chi_0)} G_{\lambda}$. Set $X \modmod{} + : = [X^{\op{ss}}(\chi_+)/G]$ and $X \modmod{} - : = [X^{\op{ss}}(\chi_-)/G]$. Let $(\mathfrak{K}^+,\mathfrak{K}^-)$ be the elementary wall crossing promised by Proposition \ref{proposition: variation for toric var gives elem wall crossing}. 

\begin{theorem} \label{theorem: SOD of derived categories for toric variations}
 Fix $d \in \Z$. Let $x \in Z_{\lambda}^0$.
 \begin{enumerate}
  \item If $\mu(\omega^{-1}_X,\lambda,x) > 0$, there exists fully-faithful functors,
  \begin{displaymath}
   \Phi^+_d : \dbcoh{ X \modmod{} - } \to \dbcoh{X \modmod{} +}
  \end{displaymath}
  and
  \begin{displaymath}
   \Upsilon^+: \dbcoh{ Y_{\lambda} } \to \dbcoh{X \modmod{} +},
  \end{displaymath}
   and a semi-orthogonal decomposition, with respect to $\Phi^+_d, \Upsilon^+$,
  \begin{displaymath}
   \dbcoh{ X \modmod{} + } = \langle \dbcoh{Y_{\lambda}}(-t(\mathfrak{K}^-) + d), \ldots, \dbcoh{Y_{\lambda}}(-t(\mathfrak{K}^+) + d -1), \dbcoh{X \modmod{} -} \rangle,
  \end{displaymath}
  where $\dbcoh{Y_{\lambda}}(l)$ is the image of $\Upsilon^+$ tensored with $\mathcal O(l \chi)$ where $\chi$ is a character such that $\chi \circ \lambda(\alpha) = \alpha$.
  \item If $\mu(\omega^{-1}_X,\lambda,x) = 0$, there is an equivalence,
  \begin{displaymath}
   \Phi^+_d : \dbcoh{ X \modmod{} - } \to \dbcoh{X \modmod{} +}.
  \end{displaymath}
  \item If $\mu(\omega^{-1}_X,\lambda,x) < 0$, there exists fully-faithful functors,
  \begin{displaymath}
   \Phi^-_d : \dbcoh{ X \modmod{} + } \to \dbcoh{X \modmod{} -}
  \end{displaymath}
  and
  \begin{displaymath}
   \Upsilon^-: \dbcoh{ Y_{\lambda} } \to \dbcoh{X \modmod{} -},
  \end{displaymath}
  and a semi-orthogonal decomposition, with respect to $\Phi^-_d, \Upsilon^-$,
  \begin{displaymath}
   \dbcoh{ X \modmod{} - } = \langle \dbcoh{Y_{\lambda}}(t(\mathfrak{K}^-)-d+1), \ldots, \dbcoh{Y_{\lambda}}(t(\mathfrak{K}^+)-d), \dbcoh{X \modmod{} +} \rangle,
  \end{displaymath}
  where $\dbcoh{Y_{\lambda}}(l)$ is the image of $\Upsilon^-$ tensored with $\mathcal O(l \chi)$ where $\chi$ is a character such that $\chi \circ \lambda(\alpha) = \alpha^{-1}$.
 \end{enumerate}
\end{theorem}

\begin{proof}
 Corollary \ref{corollary: Isik} allows us to replace our derived categories of sheaves by derived categories of factorizations. 

 Proposition \ref{proposition: variation for toric var gives elem wall crossing} guarantees we can apply Theorem \ref{theorem: elementary wall crossing}. Since $G$ is Abelian, Lemma \ref{lemma: wall compositions are same} states that $\dbcoh{[Z^0_{\lambda}/C(\lambda)]}_m \cong \dbcoh{Y_{\lambda}}$ for all $m \in \Z$. Note that we can twist by $(m)$ on $Y_{\lambda}$ or $X$ when defining $\Upsilon_m$. 
 
 We can apply Lemma \ref{lemma: formula for mu} to compute $-t(\mathfrak{K}^+) + t(\mathfrak{K}^-)$ via $\mu(\omega_X^{-1},\lambda,x)$. Remark \ref{remark: walls as GIT quotients} notes that $Y_{\lambda}$ is the GIT quotient of the fixed locus, $X^{\lambda}$, by $C(\lambda)/\lambda = G/\lambda = G_{\lambda}$ using the character, $\chi_0$. 
\end{proof}

\begin{remark} \label{remark: which side is the anti-canonical on}
 There is a particularly simple method for determining the ``larger'' derived category in Theorem \ref{theorem: SOD of derived categories for toric variations}. By this, we mean the direction of the semi-orthogonal decomposition. {\it The larger derived category, the one being decomposed, always lies in the same side of the wall as the anti-canonical bundle}.
\end{remark}

Let us use Theorem \ref{theorem: SOD of derived categories for toric variations} to recover the following result of Kawamata, \cite{Kaw06,Kaw12}.
\begin{theorem} \label{theorem: full exc coll on toric var}
 Let $X \modmod{\chi} G$ be a projective DM toric stack and assume that $\chi$ lies in the interior of a chamber of the GKZ fan. The derived category, $\dbcoh{X \modmod{\chi} G}$, admits a full exceptional collection.
\end{theorem}

\begin{proof}
 We use Theorem \ref{theorem: SOD of derived categories for toric variations}, induction, and the following fact: for a one-parameter subgroup, $\lambda$, the GKZ fan of the action of $G_{\lambda}$ on $X^{\lambda}$ is isomorphic to the restriction of the GKZ fan, for the action of $G$ on $X$, to the hyperplane defined by $\lambda$. This fact can be checked by working through the construction of the GKZ fan, see \cite{GKZ} or \cite{CLS}, or, noting that, by Proposition \ref{proposition: affine GIT fan}, the relative interiors of the cones of the GKZ fan are exactly the subsets of constant semi-stable locus. 
 
 We use induction on the dimension of $X$. When $X$ has dimension zero, $G$ is finite. The derived category is simply the derived category of finite-dimensional representations of $G$. As $k$ has characteristic zero, any representation is completely reducible. The finite collection of simple representations of $G$ is an exceptional collection by Schur's Lemma ($k$ is algebraically closed). 
 
 Now assume we have proven the statement whenever the dimension of $X$ is strictly less than $n$ and assume that $X = \mathbb{A}^n$. Let $\chi \in C$ lie in the interior. Let $C_{-K}$ be a chamber containing $\omega_{X}^{-1}$. Choose a straight line path, $\gamma: [0,1] \to \widehat{G}_{\mathbb{R}}$, such that,
\begin{itemize}
 \item $\gamma(0) = \omega_X^{-1}$,
 \item $\gamma([0,1]) \cap \op{Int} C$ is nonempty,
 \item if $[t_i,t_e] = \gamma([0,1]) \cap C$, then there is some $\delta > 0$ so that $\gamma((t_e-\delta,1])$ does not intersect any cone of codimension $2$,
 \item $\gamma(1)$ lies in the interior of the empty chamber.
\end{itemize}
\sidenote{{\color{red} Picture? }}
 To see that such an $\gamma$ exists, fix an Euclidean metric on $\widehat{G}_{\mathbb{R}} \cong \R^r$. Choose some $\epsilon > 0$ so that, for any $\chi' \in B_{\epsilon}(0) \cap \op{Int} C$, the straight line path, $r: [0,\infty) \to \widehat{G}_{\mathbb{R}}$, defined by $r(0) = \omega_X^{-1}$ and $r(1) = \chi'$, has $r(t)$ lying in the complement of the secondary fan for $t >> 0$. Since we have assumed projectivity, the secondary fan is strongly convex by Proposition 14.1.3 of \cite{CLS}. Therefore, such an $\epsilon$ exists. Let $S$ be a small sphere centered at $\omega_X^{-1}$. The image of the radial projection of the union of all cones of codimension $\geq 2$ onto $S$ has codimension $\geq 1$, while the radial projection of $B_{\epsilon}(0) \cap \op{Int} C$ is an open subset. Thus, if we choose a general point, $\chi'$, in $B_{\epsilon}(0) \cap \op{Int} C$, then the straight line path from $\omega_X^{-1}$ through $\chi'$ will satisfy all the conditions after rescaling.

 The path, $\gamma$, passes through a finite number of walls, $M_1, \ldots, M_s$, on its way through the chambers, $C =: C_0, \ldots, C_s := C_{\emptyset}$, where $C_{\emptyset}$ denotes the empty chamber. At each wall, $M_i$, we take the one-parameter subgroup, $\lambda_i$, defining the hyperplane containing $M_i$ normalized so that $\lambda_i$ pairs non-negatively with $C_i$. By Theorem \ref{theorem: SOD of derived categories for toric variations}, we have a semi-orthogonal decomposition,
 \begin{displaymath}
  \dbcoh{X \modmod{\chi_{i-1}} G} = \langle \dbcoh{Y_{\lambda_i}},\ldots,\dbcoh{Y_{\lambda_i}}(-\mu_i - 1), \dbcoh{X \modmod{\chi_i} G} \rangle,
 \end{displaymath}
 where $\mu_i = \mu(\omega_X^{-1},\lambda_i,x) < 0$ for $x \in Z_{\lambda_i}^0$. Hence, combining these semi-orthogonal decompositions, we obtain a semi-orthogonal decomposition,
 \begin{gather*}
  \dbcoh{X \modmod{\chi} G} = \langle \dbcoh{Y_{\lambda_1}}, \ldots , \dbcoh{Y_{\lambda_1}}(-\mu_1 - 1) , \ldots ,\\  \dbcoh{Y_{\lambda_s}}, \ldots, \dbcoh{Y_{\lambda_s}}(-\mu_s - 1) \rangle.
 \end{gather*}
 Note, $\op{dim }(X^{\lambda_i}) < \op{dim }(X)$ for all $i$. Applying the induction hypothesis, we conclude that $\dbcoh{X}$ possesses a full exceptional collection.
\end{proof}

\begin{definition}
 Let $P$ be a property of a triangulated category. We say that $P$ \textbf{descends under semi-orthogonal decomposition} if whenever $\mathcal A$ is an admissible subcategory of $\mathcal T$ and $P$ is true for $\mathcal T$, then $P$ is also true for $\mathcal A$.
\end{definition}

We may run the argument of Theorem \ref{theorem: full exc coll on toric var} backwards to show the following.

\begin{proposition} \label{proposition: SOD prop for toric}
 Let $P$ be a property of triangulated categories that descends under semi-orthogonal decomposition. If $P$ is true for the derived categories of all weakly-Fano projective toric DM stacks, then it is true for the derived categories of all projective toric DM stacks.
\end{proposition}

\begin{proof}
 Recall that that a projective toric DM stack, $X \modmod{\chi} G$, is weakly-Fano if its anti-canonical bundle is nef and big. This is equivalent to $\omega_X^{-1}$ lying in the same chamber of the GKZ fan as $\chi$. Via an argument similar to the proof of Theorem \ref{theorem: full exc coll on toric var}, for any $\chi'$ lying in the interior of a chamber, there exists a $\chi$ in the interior of a chamber containing $\omega_X^{-1}$ such that $\dbcoh{X \modmod{\chi'} G}$ is an admissible category of $\dbcoh{X \modmod{\chi} G}$. Since $P$ is true for $\dbcoh{X \modmod{\chi} G}$, it descends to $\dbcoh{X \modmod{\chi'} G}$.
\end{proof}

\begin{corollary} \label{corollary: rouquier dimension toric}
 Let $X \modmod{\chi} G$ be a projective DM toric stack of Picard rank $\leq 2$ or of dimension $\leq 2$. The Rouquier dimension of $X$ is equal to the Krull dimension of $X$.
\end{corollary}

\begin{proof}
 We refer the definition of and details about Rouquier dimension to \cite{BF}. In \cite{BF}, the Rouquier dimension of a projective, weakly-Fano DM toric stack of Picard rank $\leq 2$ or of dimension $\leq 2$ was shown to be equal to $\op{dim} X$.  
 
 We note that the Rouquier dimension cannot increase under passage to an admissible subcategory. Thus, the property ``Rouquier dimension is $\leq \op{dim} X$'' descends under semi-orthogonal decompositions. By Proposition \ref{proposition: SOD prop for toric}, the Rouquier dimension of a projective DM toric stack is $\leq \op{dim} X$. As the Rouquier dimension is always at least the dimension of the DM stack, Lemma 2.17 \cite{BF}, we reach the desired conclusion. 
\end{proof}

\begin{remark}
 In \cite{O4}, Orlov conjectured that the Rouquier dimension of any smooth projective variety equaled the Krull dimension. As evidence, he proved his conjecture for curves. Corollary \ref{corollary: rouquier dimension toric} verifies Orlov's conjecture for projective DM toric stacks of Picard rank at most $2$ or of dimension at most $2$. 
\end{remark}

\section{Exceptional collections on moduli spaces of pointed rational curves} \label{section: moduli}

In this section, we first consider GIT quotients of $(\mathbb{P}^1)^n$ by the diagonal action of $\op{PGL}_2$. Subsequently, we consider the GIT quotients of the Fulton-MacPherson compactification of $n$ points on $\P^1$, \cite{FM}, under the action of $\op{PGL}_2$. We prove the existence of full exceptional collections for many of these moduli spaces of pointed rational curves.
 
\subsection{Moduli of weighted points on a line} \label{subsection: moduli of weighted points}

Let us start with the so-called ``elementary example'' of \cite{MFK}, the diagonal action of $\op{PGL}_2$ on $P_n := (\mathbb{P}^1)^n$. For technical reasons, we will also need to consider the action of $\op{SL}_2$ on $P_n$. Most of the results in this section on the structure of the GIT fan appear in \cite{Tha96,DH98}. Set $[n] := \{1,\ldots,n\}$.

Let $\pi_i: P_n \to \P^1$ be the projection onto the $i$-th factor. The Picard group of $P_n$ equals the Neron-Severi group of $P_n$. Both are isomorphic to $\Z^n$ and generated by $\pi_i^*\mathcal O_{\P^1}(1)$ for $1 \leq i \leq n$. We set
\begin{displaymath}
 \mathcal O(\bm{d}) := \mathcal O(d_1,\ldots,d_n) := \bigotimes_{i=1}^n \pi_i^*\mathcal O_{\P^1}(d_i).
\end{displaymath}

\begin{lemma} \label{lemma: Picard group of (P1)n}
 There is an isomorphism, 
 \begin{displaymath}
  \op{Pic}^{\op{SL}_2}(P_n) = \op{NS}^{\op{SL}_2}(P_n) \cong \Z^n.
 \end{displaymath}
 The bundles, $\pi_i^*\mathcal O_{\P^1}(1)$, form a basis for $\op{Pic}^{\op{SL}_2}(P_n)$. 
 
 There is an isomorphism, 
 \begin{displaymath}
  \op{Pic}^{\op{PGL}_2}(P_n) = \op{NS}^{\op{PGL}_2}(P_n) = \{ \mathcal O(\bm{d}) \in \op{NS}^{\op{SL}_2}(P_n) \mid \sum d_i \text{ is even }\}.
 \end{displaymath}
\end{lemma}

\begin{proof}
 Two different $G$-equivariant structures on the same line bundle differ by a twist of a character. As both $\op{PGL}_2$ and $\op{SL}_2$ have no characters, an equivariant structure is unique, if it exists. It is clear that $\pi_i^*\mathcal O_{\P^1}(1)$ admits a natural $\op{SL}_2$ equivariant structure coming from the action on $\mathbb{A}^2$. 

 Any $\op{PGL}_2$ equivariant structure automatically provides a $\op{SL}_2$ equivariant structure. However, the center of $\op{SL}_2$ acts nontrivially on $\mathcal O(\bm{d})$ if $\sum d_i$ is odd. Thus, only $\mathcal O(\bm{d})$ with $\sum d_i$ even admit a $\op{PGL}_2$ equivariant structure. 
\end{proof}

Any one-parameter subgroup, $\lambda$, has exactly two fixed points on $\P^1$. Let $y^-_{\lambda}$ be the fixed point with
\begin{displaymath}
 \mu( \varOmega_{\P^1},\lambda,y^-_{\lambda}) = -1
\end{displaymath}
and let $y^+_{\lambda}$ be the fixed point with 
\begin{displaymath}
 \mu( \varOmega_{\P^1},\lambda,y^+_{\lambda}) = 1.
\end{displaymath}
Let us next compute the possible values of the Hilbert-Mumford numerical function.

\begin{lemma} \label{lemma: computing the HM function for (P1)n}
 Let $\lambda$ be a one parameter subgroup of $\op{PGL}_2$ and let $x = (x_1,\ldots,x_n) \in P_n$. Let $x^* := \lim_{t\to 0} \sigma(\lambda(\alpha),x)$ and set
 \begin{displaymath}
  I := \{i \in [n] \mid x_i = y_{\lambda}^- \}.
 \end{displaymath}
 Let $I^c$ denote the complement of $I$ in $[n]$. Then,
 \begin{displaymath}
  \mu(\mathcal O(\bm{d}),\lambda,x^*) = \frac{1}{2}\left(\sum_{i\in I} d_i - \sum_{i \in I^c} d_i\right).
 \end{displaymath}
\end{lemma}

\begin{proof} 
 Any one-parameter subgroup of $\op{PGL}_2$ is conjugate to
 \begin{displaymath}
  \begin{pmatrix} \alpha & 0 \\ 0 & 1 \end{pmatrix}  \text{ or } \begin{pmatrix} \alpha^{-1} & 0 \\ 0 & 1 \end{pmatrix}.
 \end{displaymath}

 By Lemma \ref{lemma: weights under the orbit}, we may assume that $\lambda$ is one of these one-parameter subgroups. Let us assume that 
 \begin{displaymath}
  \lambda(\alpha) = \begin{pmatrix} \alpha & 0 \\ 0 & 1 \end{pmatrix};
 \end{displaymath}
 the other case is completely analogous.
 
 The weight of $\lambda$ at $x^*$ is one half of the weight of 
 \begin{displaymath}
  \lambda^2(\alpha) = \begin{pmatrix} \alpha^2 & 0 \\ 0 & 1 \end{pmatrix} = \begin{pmatrix} \alpha & 0 \\ 0 & \alpha^{-1} \end{pmatrix} \in \op{PGL}_2.
 \end{displaymath}
 This has the pleasant property that it lifts to a one-parameter subgroup, $\widetilde{\lambda}^2$, of $\op{SL}_2$. Notice also that
 \begin{align*}
  \mu(\mathcal O_{\P^1}(1),\widetilde{\lambda}^2,[1:0]) & =  1 \\
  \mu(\mathcal O_{\P^1}(1),\widetilde{\lambda}^2,[0:1]) & = -1.
 \end{align*}

 Let $x^* = \lim_{\alpha \to 0} \sigma(\lambda(\alpha),x)$. We have
 \begin{displaymath}
  x^*_i = \lim_{\alpha \to 0} \sigma(\lambda(\alpha),x_i) = \begin{cases} y^-_{\lambda} & x_i = y^-_{\lambda} \\ y^+_{\lambda} & \text{otherwise}. \end{cases}
 \end{displaymath}
 For $\lambda$ as above, we have $y^-_{\lambda} = [1:0]$ and $y^+_{\lambda} = [0:1]$. 
 
Now using linearity of $\mu$, we compute,
 \begin{align*}
  \mu(\mathcal O(\bm{d}),\lambda,x^*) &  =   \frac{1}{2} \mu(\mathcal O(\bm{d}),\widetilde{\lambda}^2,x^*) \\
    & =    \frac{1}{2} \left(\sum_{i \in I} \mu(\mathcal O_{\P^1}(d_i),\widetilde{\lambda}^2, [1:0]) + \sum_{i \in I^c} \mu(\mathcal O_{\P^1}(d_i), \widetilde{\lambda}^2, [0:1])\right) \\
    & = \frac{1}{2} (\sum_{i \in I} d_i - \sum_{i \in I^c} d_i).
 \end{align*}
\end{proof}

\begin{lemma} \label{lemma: ss locus P1n}
 We have 
 \begin{displaymath}
  P_n^{\op{ss}}(\bm{d}) := P_n^{\op{ss}}(\mathcal O(\bm{d})) = \{ (x_1,\ldots, x_n) \mid \text{if } I \subset [n] \text{ with } x_i=x_j~ \forall i,j\in I, \text{ then } \sum_{i \in I} d_i \leq \sum_{i \in I^c} d_i \}.
 \end{displaymath}
\end{lemma}

\begin{proof}
 The Hilbert-Mumford numerical criterion, Theorem \ref{theorem: Hilbert-Mumford numerical stability}, states that
 \begin{displaymath}
   P_n^{\op{ss}}(\bm{d}) = \{ x \mid \mu(\mathcal O(\bm{d}),\lambda,x) \leq 0~ \forall \lambda \}.
 \end{displaymath}
 Lemma \ref{lemma: computing the HM function for (P1)n} gives 
 \begin{displaymath}
  \mu(\mathcal O(\bm{d}),\lambda,x) = \frac{1}{2} \left(\sum_{i\in I} d_i - \sum_{i \in I^c} d_i\right)
 \end{displaymath}
 where $I = \{i \in [n] \mid x_i = y_{\lambda}^- \}$. If we have a subset, $I \subset [n]$, with $x_i = x_j$ for all $i,j \in I$, then we can find a one-parameter subgroup $\lambda$ with $y_{\lambda}^- = x_i$ for $i \in I$. The conclusion is clear from these facts.
\end{proof}

Let $I \subseteq [n]$. Define a hyperplane, 
\begin{displaymath}
 H_I := \{ \mathcal O(\bm{d}) \mid \sum_{i \in I} d_i = \sum_{i \in I^c} d_i \} \subset \R^n,
\end{displaymath}
and half-spaces
\begin{align*}
 H_I^{\geq 0} & := \{ \mathcal O(\bm{d}) \mid \sum_{i \in I} d_i \geq \sum_{i \in I^c} d_i \} \subset \R^n \\
 H_I^{> 0} & := \{ \mathcal O(\bm{d}) \mid \sum_{i \in I} d_i > \sum_{i \in I^c} d_i \} \subset \R^n \\
 H_I^{\leq 0} & := \{ \mathcal O(\bm{d}) \mid \sum_{i \in I} d_i \leq \sum_{i \in I^c} d_i \} \subset \R^n \\
 H_I^{< 0} & := \{ \mathcal O(\bm{d}) \mid \sum_{i \in I} d_i < \sum_{i \in I^c} d_i \} \subset \R^n,
\end{align*}

Since $\op{SL}_2$ and $\op{PGL}_2$ differ by a finite group, their semi-stable loci coincide whenever the line bundle admits an equivariant structure for both groups.

\begin{corollary} \label{corollary: GIT fan for points in P1}
 The subsets of constant semi-stable locus within the ample cone are the intersections,
 \begin{displaymath}
  H_{I_1}^{> 0} \cap \cdots \cap H_{I_{j_+}}^{> 0} \cap H_{I_{j_+ + 1}}^{< 0} \cap \cdots \cap H_{I_{j_+ + j_-}}^{< 0} \cap H_{I_{j_+ + j_- + 1}} \cap \cdots \cap H_{I_{j_+ + j_- +j_0}} \cap(\R_{>0})^n.
 \end{displaymath}
 
 The closures of these subsets form a fan whose support is the ample cone. The chambers of this fan correspond to intersections with $j_0 = 0$ and the walls correspond to intersections with $j_0 = 1$.
\end{corollary}

\begin{proof}
 The first conclusion is immediate from Lemma \ref{lemma: ss locus P1n}. The remainder of the conclusions are evident.
\end{proof} 

\begin{definition}
 We shall call the fan from Corollary \ref{corollary: GIT fan for points in P1}, the \textbf{extended GIT fan}. A chamber in the extended GIT fan is called an \textbf{empty chamber} if its interior has empty semi-stable locus.
\end{definition}

\begin{lemma} \label{lemma: empty chambers for P1n}
 Assume that $d_i > 0$ for all $i$. The semi-stable locus, $P_n^{\op{ss}}(\bm{d})$, is empty if and only if there exists an $i_0$ with $d_{i_0} > \sum_{i \not = i_0} d_i$.
\end{lemma}

\begin{proof}
 We use the Hilbert-Mumford numerical criterion, Theorem \ref{theorem: Hilbert-Mumford numerical stability}. Assume that there exists an $i_0$ with $d_{i_0} > \sum_{i \not = i_0} d_i$. It is clear from Lemma \ref{lemma: ss locus P1n} that $P_n^{\op{ss}}(\bm{d}) = \emptyset$. 
 
 Assume that $d_{j} < \sum_{i \not = j} d_i$ for all $1 \leq j \leq n$. Lemma \ref{lemma: ss locus P1n} implies that a point, $x = (x_1,\ldots,x_n)$, with $x_i \not = x_j$ if $i \not = j$, lies in $P_n^{\op{ss}}(\bm{d})$.
\end{proof}

\begin{lemma} \label{lemma: P1n elem wall cross}
 Let $C_-$ and $C_+$ be two chambers of of the extended GIT fan and let $M$ be their common face. Choose, $\mathcal O(\bm{d}_-) \in \op{Int} C_-$ and $\mathcal O(\bm{d}_+) \in \op{Int} C_+$ near $M$ and let $\mathcal O(\bm{d}_0)$ be the point where the straight line path between $\mathcal O(\bm{d}_-)$ and $\mathcal O(\bm{d}_+)$ intersects $M$. The variation, $(\mathcal O(\bm{d}_-), \mathcal O(\bm{d}_+))$, satisfies the DHT condition. Consequently, there is an an elementary wall crossing,
 \begin{align*}
  P_n^{\op{ss}}(\bm{d}_0) & = P_n^{\op{ss}}(\bm{d}_+) \sqcup S_{\lambda} \\
  P_n^{\op{ss}}(\bm{d}_0) & = P_n^{\op{ss}}(\bm{d}_-) \sqcup S_{-\lambda}.
 \end{align*}
\end{lemma}

\begin{proof}
 The first condition of being a DHT variation follows from Corollary \ref{corollary: GIT fan for points in P1}. The second condition follows as the stabilizer in $\op{PGL}_2$ of any point in $(\P^1)^n$ is either $\mathbb{G}_m$ or $1$. Let $H_I$ be the hyperplane that $M$ lies in. Then, applying the Hilbert-Mumford numerical criterion, Theorem \ref{theorem: Hilbert-Mumford numerical stability}, 
 \begin{gather*}
  P_n^{\op{ss}}(\bm{d}_0)\setminus \left(  P_n^{\op{ss}}(\bm{d}_-) \cup P_n^{\op{ss}}(\bm{d}_+) \right) = \\ \{ (x_1,\ldots,x_n) \mid x_i = x_j \text{ if } i,j \in I \text{ or } i,j \in I^c \text{ and } x_i \not = x_j \text{  if } i \in I, j \in I^c\},
 \end{gather*}
 which is connected. The final statement is Theorem \ref{theorem: VGIT give elementary wall crossing}.
\end{proof}

We let
\begin{displaymath}
 P_n \modmod{\bm{d}} \op{PGL}_2 := [P_n^{\op{ss}}(\bm{d})/\op{PGL}_2]
\end{displaymath}
and 
\begin{displaymath}
 P_n \modmod{\bm{d}} \op{SL}_2 := [P_n^{\op{ss}}(\bm{d})/\op{SL}_2].
\end{displaymath}

We will prove the following proposition.

\begin{proposition} \label{proposition: exc coll SL2}
 Let $\mathcal O(\bm{d})$ lie in the interior of a chamber of the GIT fan for the action of $\op{SL}_2$ on $(\P^1)^n$. We have a full exceptional collection,
 \begin{displaymath}
  \dbcoh{P_n \modmod{\bm{d}} \op{SL}_2} = \langle \mathcal E_1,\ldots, \mathcal E_{l(\bm{d})}, \mathcal F_1,\ldots, \mathcal F_{l(\bm{d})} \rangle
 \end{displaymath}
 where the center of $\op{SL}_2$, $\Z/(2)$, acts as $\op{Id}_{\mathcal E_i}$ on $\mathcal E_i$ and acts as $-\op{Id}_{\mathcal F_i}$ on $\mathcal F_i$.
\end{proposition}

As a corollary we get the following statement.

\begin{theorem} \label{theorem: exc coll on moduli of weighted points}
 Let $\mathcal O(\bm{d})$ lie in the interior of a chamber of the GIT fan for the action of $\op{PGL}_2$ on $(\P^1)^n$. The derived category, $\dbcoh{P_n \modmod{\bm{d}} \op{PGL}_2}$, admits a full exceptional collection. 
\end{theorem}

\begin{proof}
 We have an inclusion,
 \begin{displaymath}
  \dbcoh{P_n \modmod{\bm{d}} \op{PGL}_2} \to \dbcoh{P_n \modmod{\bm{d}} \op{SL}_2},
 \end{displaymath}
 as the subcategory of complexes for which the center of $\op{SL}_2$, $\Z/(2)$, acts trivially.  Note that, in the decomposition,
 \begin{displaymath}
  \dbcoh{P_n \modmod{\bm{d}} \op{SL}_2} = \langle \mathcal E_1,\ldots, \mathcal E_{l(\bm{d})}, \mathcal F_1,\ldots, \mathcal F_{l(\bm{d})} \rangle,
 \end{displaymath}
 of Proposition \ref{proposition: exc coll SL2} 
 \begin{align*}
  \op{Hom}_{P_n \modmod{\bm{d}} \op{SL}_2}(\mathcal E_i, \mathcal F_j[t]) & = 0 \\
  \op{Hom}_{P_n \modmod{\bm{d}} \op{SL}_2}(\mathcal F_j, \mathcal E_i[t]) & = 0.
 \end{align*}
 for all $1 \leq i, j \leq l(\bm{d}), t\in \Z$. Thus, 
 \begin{displaymath}
  \dbcoh{P_n \modmod{\bm{d}} \op{PGL}_2} = \langle \mathcal E_1,\ldots, \mathcal E_{l(\bm{d})} \rangle.
 \end{displaymath}
\end{proof}

\begin{remark}
 The quotient stacks, $P_n \modmod{\bm{d}} \op{PGL}_2$, are varieties by Proposition \ref{proposition: when is the GIT quotient is a scheme}.
\end{remark}

The argument to prove Proposition \ref{proposition: exc coll SL2} will proceed by passing to the boundary of the ample cone in the GIT fan.

\begin{definition}
 A \textbf{boundary chamber} in the extended GIT fan for the action of $\op{SL}_2$ on $P_n$ is a chamber whose closure in $\op{Pic}^{\op{SL}_2}(P_n)_{\R}$ intersected the boundary of the ample cone has dimension $n-1$.
\end{definition}

\begin{lemma} \label{lemma: semi-stable locus for boundary chamber P1n}
 Let $\mathcal O(\bm{d})$ lie in the interior of a boundary chamber, $C$. Take $i_0$ so that $\overline{C} \cap \{ d_{i_0} = 0 \}$ is full-dimensional. Let $\mathcal O(\bm{d}')$ lie in the interior of $\overline{C} \cap \{ d_{i_0} = 0 \}$. We have an isomorphism,
 \begin{displaymath}
  P_n^{\op{ss}}(\bm{d}) \cong \P^1 \times P_{n-1}^{\op{ss}}(\bm{d}').
 \end{displaymath}
\end{lemma}

\begin{proof}
 Since $C$ is a boundary chamber, in an inequality, 
 \begin{displaymath}
  \sum_{i \in I} d_i \leq \sum_{i \in I^c} d_i,
 \end{displaymath}
 we can take $d_{i_0} \to 0$ without violating the inequality. Applying Lemma \ref{lemma: ss locus P1n} gives the result.
\end{proof}

\begin{proof}[Proof of Proposition \ref{proposition: exc coll SL2}] 
 We proceed by induction on $n$. The base case is $n=3$ where there are four chambers in the GIT fan. The three boundary chambers have empty semi-stable locus while the GIT quotient for non-boundary chamber is isomorphic to $[\op{pt}/(\Z/(2))]$. The base case is therefore clear.
 
 Assume the statement of Proposition \ref{proposition: exc coll SL2} holds whenever the number of points is $< n$. Let $\bm{d}$ be a point in the interior of a chamber, $C$, in the GIT fan. There exists a straight line path,
 \begin{displaymath}
  \gamma(t): [0,1] \to \op{Pic}^{\op{SL}_2}(P_n)_{\R},
 \end{displaymath}
 such that
 \begin{itemize}
  \item $\gamma(0) = \omega_{P_n}^{-1} = \mathcal O(\bm{2})$,
  \item $\gamma([0,1]) \cap \op{Int} C \not = \emptyset$,
  \item $\gamma(1)$ lies in the interior of a boundary chamber,
  \item $\gamma$ passes through no codimension $2$ walls after entering $C$.
 \end{itemize}
 Let $C =: C_0,\ldots,C_s$ be the list of chambers $\gamma$ passes through after, and including, $C$. Let $M_1,\ldots,M_s$ be the walls separating them. Pick $\bm{d}_i \in \op{Int} C_i$. Thanks to Lemma \ref{lemma: P1n elem wall cross}, we can apply Theorem \ref{theorem: elementary wall crossing} and use Lemma \ref{lemma: formula for mu} to get a semi-orthogonal decomposition,
 \begin{displaymath}
  \dbcoh{ P_n \modmod{\bm{d_{i-1}}} \op{SL}_2} = \langle \dbcoh{Y_{\lambda_i}},\ldots,\dbcoh{Y_{\lambda_i}}, \dbcoh{ P_n \modmod{\bm{d_{i}}} \op{SL}_2} \rangle.
 \end{displaymath}
 \sidenote{{\color{red} Here the notation hurts because its the same 1-ps but different Z's} {\color{blue} why not use superscripts for the numbers and just $\lambda$ in the subscript?}}
 and, consequently,
 \begin{displaymath}
  \dbcoh{ P_n \modmod{\bm{d}} \op{SL}_2} = \langle \dbcoh{Y_{\lambda_1}},\ldots,\dbcoh{Y_{\lambda_l}}, \dbcoh{ P_n \modmod{\bm{d_{l}}} \op{SL}_2} \rangle
 \end{displaymath} 
 The fixed locus of $\lambda$ consists of a finite number of points, all of which lie in the same $\op{SL}_2$-orbit while $G_{\lambda} = \Z/(2)$. Thus, $\dbcoh{Y_{\lambda,j}} = \dbcoh{[\op{pt}/(\Z/(2))]}$ admits a full exceptional collection corresponding to the irreducible representations of $\Z/(2)$. We reduce to checking that $P_n \modmod{\bm{d_{l}}} \op{SL}_2$ has a full exceptional collection of the appropriate form. 
 
 Applying Lemma \ref{lemma: semi-stable locus for boundary chamber P1n}, 
 \begin{displaymath}
  P_n^{\op{ss}}(\bm{d}_l) \cong \P^1 \times P_{n-1}^{\op{ss}}(\bm{d}')
 \end{displaymath}
 where $\bm{d}'$ lies in the interior of $\overline{C}_l \cap \{ d_i = 0 \}$ for some $1 \leq i \leq n$. As we can lift the action of $\op{SL}_2$ from $\P^1$ to $\mathbb{A}^2$, $P_n \modmod{\bm{d_{l}}} \op{SL}_2$ is the projectivization of a rank $2$ vector bundle on $P_{n-1}^{\op{ss}}\modmod{\bm{d}'} \op{SL}_2$. Orlov's theorem on derived categories of projective bundles, \cite{Orl92}, states there is a semi-orthogonal decomposition,
 \begin{displaymath}
  \dbcoh{P_n \modmod{\bm{d_{l}}} \op{SL}_2} = \langle \dbcoh{P_{n-1}^{\op{ss}}\modmod{\bm{d}'} \op{SL}_2}, \dbcoh{P_{n-1}^{\op{ss}}\modmod{\bm{d}'} \op{SL}_2} \otimes \mathcal O(1) \rangle. 
 \end{displaymath}
 By the induction hypothesis, we may write
 \begin{displaymath}
  \dbcoh{P_{n-1}^{\op{ss}} \modmod{\bm{d}'} \op{SL}_2} = \langle \mathcal E_0,\ldots,\mathcal E_{l(\bm{d}')}, \mathcal F_0,\ldots, \mathcal F_{l(\bm{d}')} \rangle
 \end{displaymath}
 where the center of $\op{SL}_2$ acts trivially on $\mathcal E_i$ and by $-\op{Id}_{\mathcal F_i}$ on $\mathcal F_i$. Thus, the center acts trivially on $\mathcal E_i$ and $\mathcal F_i(1)$ and by $-1$ on $\mathcal F_i$ and $\mathcal E_i(1)$ giving the appropriate semi-orthogonal decomposition of $\dbcoh{P_n \modmod{\bm{d_{l}}} \op{SL}_2}$ and completing the induction.
\end{proof}

\begin{remark}
 In case the reader is worried about validity of the results of \cite{Orl92} in the general setting of quotient stacks, we invite her/him to reprove and extend Orlov's theorem using Theorem \ref{theorem: elementary wall crossing}. 
\end{remark}

\subsection{Moduli of pointed rational curves}

We augment and apply some of the ideas of Section \ref{subsection: moduli of weighted points} by considering GIT quotients of the Fulton-MacPherson compactification, $\P^1[n]$, \cite{FM} by $\op{PGL}_2$.

Let us recall the construction of $\P^1[n]$ as given in \cite{KM11}. Let $F_0 = (\P^1)^n$ and for $S \subseteq [n]$ set
\begin{displaymath}
 \Sigma_0^S := \{ (x_1,\ldots,x_n) \in (\P^1)^n \mid x_i = x_j \text{ for } i,j \in S\}
\end{displaymath}
and let
\begin{displaymath}
 \Sigma^j_0 := \bigcup_{|S|=j} \Sigma^S_0.
\end{displaymath}
Now, inductively define $F_l$ as the blow up of $F_{l-1}$ along $\Sigma_{l-1}^{n-l-1}$. Let $\Sigma^S_j$ be the component of the exceptional divisor of the blow up lying over $\Sigma^S_{j-1}$ if $|S|=n-l-1$ and otherwise let it be the proper transform of $\Sigma_{j-1}^S$. 

\begin{proposition} \label{proposition: this is P1[n]}
 The varieties $F_j$ are smooth and $F_{n-2} \cong \P^1[n]$.
\end{proposition}

\begin{proof}
 This is a combination of Proposition 2.8 of \cite{Li09}, Lemma 3.1 of \cite{KM11}, and Proposition 1.8 of \cite{MM}.
\end{proof}

As we are blowing up along closed $\op{PGL}_2$ orbits, each $F_j$ inherits an action of $\op{PGL}_2$. Let $\psi_j: F_j \to F_{j-1}$ be the blow up morphism. We also inductively define $G$-equivariant $(\mathbb{Q}$-)line bundles starting with $\mathcal L_0 = \mathcal O(\bm{2})$. Then, we set
\begin{displaymath}
 \mathcal L_j = \psi_j^* \mathcal L_{j-1} \otimes \mathcal O(-\delta_j \Sigma_j^{n-l-j})
\end{displaymath}
where $\delta_j$ is a decreasing sequence of sufficiently small positive rational numbers. Y.-H. Kiem and H.-B. Moon identify the GIT quotients for the line bundles, $\mathcal L_j$, as Hassett moduli spaces. The following definition is due to Hassett \cite{Has03}. 

\begin{definition}
 Fix $a_1,\ldots,a_n \in \mathbb{Q}$ with $0 < a_i \leq 1$ and $\sum a_i > 2$. An \textbf{$n$-pointed $\bm{a}=(a_1,\ldots,a_n)$-stable rational curve} is a
 \begin{itemize}
  \item a nodal connected curve, $C$, of arithmetic genus $0$ and 
  \item points, $p_1,\ldots,p_n$, in the smooth locus of $C$
 \end{itemize}
 such that 
 \begin{itemize}
  \item if $p_{i_1} = \cdots = p_{i_r}$, then $\sum_{j=1}^r a_{i_j} \leq 1$,
  \item and $K_C + a_1p_1 + \cdots +a_np_n$ is ample.
 \end{itemize}
\end{definition}

\begin{theorem}
 There is a smooth, projective variety, $\overline{M}_{0,\bm{a}}$, representing the moduli problem of $n$-pointed $\bm{a}$-stable rational curves.
\end{theorem}

\begin{proof}
 This is Theorem 2.1 of \cite{Has03}.
\end{proof}

In the case that $a_1=\cdots=a_n=\epsilon$, we will denote $\overline{M}_{0,\bm{a}}$ by $\overline{M}_{0,n \cdot \epsilon}$ and call it a \textbf{symmetrically-weighted} moduli space. Note that $\overline{M}_{0,n \cdot 1} = \overline{M}_{0,n}$, the Deligne-Grothendieck-Knudsen-Mumford moduli space of stable $n$-pointed rational curves.

\begin{theorem} \label{theorem: moduli of stable curves as GIT quotient}
 Let $m = \lfloor n/2 \rfloor$. For $1 \leq j \leq m-2$, the GIT quotient, $F_{n-m+j} \modmod{\mathcal L_{n-m+j}} \op{PGL}_2$, is isomorphic to $\overline{M}_{0,n \cdot \epsilon_j}$ with weights $\epsilon_j = (\epsilon_j,\ldots,\epsilon_j)$ for $\frac{1}{m+1-j} < \epsilon_j \leq \frac{1}{m-j}$. In particular, 
 \begin{displaymath}
  F_{n-2} \modmod{\mathcal L_{n-2}} \op{PGL}_2 \cong  F_{n-1} \modmod{\mathcal L_{n-1}} \op{PGL}_2 \cong \overline{M}_{0,n}.
 \end{displaymath}
\end{theorem}

\begin{proof}
 This is Theorem 4.1 of \cite{KM11} for the underlying coarse moduli spaces of our GIT quotient stacks. However, by Proposition \ref{proposition: when is the GIT quotient is a scheme}, each quotient stack, $F_j \modmod{\mathcal L_j} \op{PGL}_2$, is a variety for all $j$ if $n$ is odd and for $j > m$ if $n$ even. 
\end{proof}

Let
\begin{displaymath}
 \psi_{j,l} := \psi_j \circ \cdots \circ \psi_{l+1} : F_j \to F_l
\end{displaymath}
so $\psi_{j,l}^* \mathcal O_{F_l}(\Sigma_l^S) = \mathcal O_{F_j}(\Sigma_j^S)$ when $|S| \geq n-l+1$. 

\begin{lemma} \label{lemma: G-Pic of P1[n]}
 We have an isomorphism,
 \begin{displaymath}
  \op{Pic}^{\op{PGL}_2}(F_j) = \op{NS}^{\op{PGL}_2}(F_j) = \psi^*_{j,0}\op{NS}^{\op{PGL}_2}(F_0) \oplus \bigoplus_{\substack{S \subseteq [n] \\ |S| \geq n-j-1 }} \Z \cdot \mathcal O(\Sigma^S_j)
 \end{displaymath}
\end{lemma}

\begin{proof}
 As $\op{PGL}_2$ has no characters, any two equivariant structures on the same line bundle must coincide. We are left to determine which line bundles admit $\op{PGL}_2$-equivariant structures. As $F_j$ is an iterated blow up, we have
 \begin{displaymath}
  \op{Pic}(F_j) = \psi^*_{j,0}\op{Pic}(F_0) \oplus \bigoplus_{\substack{S \subseteq [n] \\ |S| \geq n-j-1 }} \Z \cdot \mathcal O(\Sigma^S_j).
 \end{displaymath}
 The line bundles, $\mathcal O(\Sigma^S_j)$, admit $\op{PGL}_2$-equivariant structures given by the derivative of the $\op{PGL}_2$ action restricted to the normal bundles of the $\Sigma^S_l$ for $|S| \geq n-l+1$.
\end{proof}

We let 
\begin{displaymath}
 \mathcal O(\bm{d} + \bm{a}) := \psi^*_{j,0} \mathcal O_{F_0}(\bm{d}) \otimes \mathcal O( \sum_S a_S \Sigma^S_j) \in \op{Pic}^{\op{PGL}_2}(F_j).
\end{displaymath}

\begin{lemma} \label{lemma: P1[n] values of HM function}
 Let $\lambda$ be a one-parameter subgroup of $\op{PGL}_2$ and let $x \in F_j$. Set $\bar{x} = \psi_{j,0}(x)$ and 
 \begin{displaymath}
  I := \{ i \in [n] \mid \bar{x}_i = y^-_{\lambda} \}.
 \end{displaymath}
 Let $x^* = \lim_{\alpha \to 0} \sigma(\lambda(\alpha),x)$. We have an equality,
 \begin{displaymath}
  \mu( \mathcal O(\bm{d} + \bm{a}), \lambda, x^*) = \frac{1}{2}\left(\sum_{i \in I} d_i - \sum_{i \in I^c} d_i\right) - \sum_{\substack{S \subseteq I \\ |S| \geq n-j-1 }} a_S + \sum_{\substack{S \subseteq I^c \\ |S| \geq n-j-1 }} a_S.
 \end{displaymath}
\end{lemma}

\begin{proof}
 Due to the linearity of $\mu$, we have
 \begin{displaymath}
  \mu( \mathcal O(\bm{d} + \bm{a}),\lambda,x^*) = \mu(\mathcal O_{F_j}(\bm{d}),\lambda,x^*) + \sum_{|S| \geq n-j-1 } a_S \mu(\mathcal O_{F_j}(\Sigma^S_j),\lambda,x^*).
 \end{displaymath}
 By Lemma \ref{lemma: weights under the orbit}, $\mu$ respects pullbacks,
 \begin{displaymath}
  \mu(\mathcal O_{F_j}(\bm{d}),\lambda,x^*) = \mu(\mathcal O_{F_0}(\bm{d}),\lambda,\bar{x}^*) .
 \end{displaymath}
 Applying Lemma \ref{lemma: computing the HM function for (P1)n}, we then get
 \begin{displaymath}
  \mu(\mathcal O_{F_j}(\bm{d}),\lambda,x^*) = \frac{1}{2}\left(\sum_{i \in I} d_i - \sum_{i \in I^c} d_i\right) .
 \end{displaymath}

\sidenote{{\color{red} blow-up vs blow up?}}
 We are left to compute $\mu(\mathcal O(\Sigma^S_j),\lambda,x^*)$. Again, since $\mu$ respects pullbacks, we may assume that $|S|=n-j+1$. We may factor the blow up of $F_{j-1}$ at $\Sigma^{n-j+1}_{j-1}$ by blowing up $\Sigma_{j-1}^S$ first and then blowing up the strict transforms of the remaining $\Sigma^{S'}_{j-1}$ for $|S'| = n-j+1$. Thus, we may replace $F_j$ by the blow up of $F_{j-1}$ at $\Sigma_{j-1}^S$ in our computation. Let us denote this blow up by $F^S_j$, let $\pi: F^S_j \to F_{j-1}$ be the associated morphism, and denote the exceptional locus by $\Sigma_j^S$.
 
 Let $x^* := \lim_{\alpha \to 0} \sigma(\lambda(\alpha),x)$. If $x^* \not \in \Sigma_j^S$, then we have an isomorphism of $\lambda$-equivariant vector spaces, $\op{V}(\mathcal O(\Sigma^S_j)_{x^*}) \cong \op{V}(\mathcal O_{x^*})$, so
 \begin{displaymath}
  \mu(\mathcal O(\Sigma^S_j),\lambda,x^*) = 0.
 \end{displaymath}
 Assume that $x^* \in \Sigma_j^S$. This is true if and only if $\bar{x}^* \in \Sigma_0^S$ which is true if and only if $S \subseteq I$ or $S \subseteq I^c$. Assume that $S \subseteq I$. The argument in the case, $S \subseteq I^c$, will be completely analogous. Write $x^* = (\pi(x^*),[v])$ with $v \in N_{\Sigma_{j-1}^S,\pi(x^*)}$ and $[v] \in \P(N_{\Sigma_{j-1}^S,\pi(x^*)})$. Recall that $N_{\Sigma_{j-1}^S,\pi(x^*)} = \op{V}(\mathcal N_{\Sigma_{j-1}^S,\pi(x^*)}^{\vee})$ denotes the geometric normal bundle to $\Sigma_{j-1}^S$ at $\pi(x^*)$.
 
 Then, $\op{V}(\mathcal O(\Sigma_j^S))_{x^*} = (kv)^{\vee}$. As $x^*$ is a fixed point, $v$ must be an eigenvector for $\lambda$, $\sigma(\lambda(\alpha),v) = \alpha^{l}v$. Then 
 \begin{displaymath}
  \mu(\mathcal O(\Sigma^S_j),\lambda,x) = -l.
 \end{displaymath}
 We need to determine the weights of $\lambda$ on $N_{\Sigma_{j-1}^S,\pi(x^*)}$. The subvariety, $\Sigma_{j-1}^S$, is the strict transform of the smooth variety, $\Sigma_0^S$, so we have a $\lambda$-equivariant isomorphism,
 \begin{displaymath}
  d\psi_{j-1,0}: N_{\Sigma_{j-1}^S,\pi(x^*)} \to N_{\Sigma_0^S,\bar{x}^*}.
 \end{displaymath}
 We can compute the weights on $N_{\Sigma_0^S,\bar{x}^*}$. We split the tangent space
 \begin{displaymath}
  T_{(\P^1)^n,\bar{x}^*} = \bigoplus_{i=1}^n T_{\P^1,\bar{x}^*_i}.
 \end{displaymath}
 We let 
 \begin{align*}
  T_{\P^1,\bar{x}^*_i}^{\oplus I} & := \bigoplus_{i \in I} T_{\P^1,\bar{x}^*_i} \subset T_{(\P^1)^n,\bar{x}^*} \\
  T_{\P^1,\bar{x}^*_i}^{\oplus I^c} & := \bigoplus_{i \in I^c} T_{\P^1,\bar{x}^*_i} \subset T_{(\P^1)^n,\bar{x}^*}.
 \end{align*}
 We have
 \begin{displaymath}
  T_{\Sigma^S_0,\bar{x}^*} = \{ (v_1,\ldots,v_n) \in T_{(\P^1)^n,\bar{x}^*} \mid v_i = v_j~ \forall i,j \in S \}.
 \end{displaymath}
 Note that $T_{\Sigma^S_0,\bar{x}^*} \subseteq T_{\P^1,\bar{x}^*_i}^{\oplus I^c}$ as $S \subseteq I$. So there is a $\lambda$-equivariant surjection, $T_{\P^1,\bar{x}^*_i}^{\oplus I} \to N_{\Sigma^S_0,\bar{x}^*}$. Therefore,  $N_{\Sigma^S_0,\bar{x}^*}$ only has weight $1$. Combining everything, we have 
 \begin{displaymath}
  \mu(\mathcal O(\Sigma_j^S),\lambda,x) = \begin{cases} 0 & x \not \in \Sigma_j^S \\ -1 & S \subseteq I \\ 1 & S \subseteq I^c, \end{cases}
 \end{displaymath}
 finishing the computation.
\end{proof}

\begin{corollary} \label{corollary: fixed locus of lambda FM}
 The fixed locus of $\lambda$ on $F_j$ is the inverse image of the fixed locus of $\lambda$ on $F_0$ under $\psi_{j,0}: F_j \to F_0$. 
\end{corollary}

\begin{proof}
 In the proof of Lemma \ref{lemma: P1[n] values of HM function}, we showed that the whole fiber of a fixed point under a single blow up is itself fixed. The total map is a composition of these blow ups.
\end{proof}

\begin{lemma} \label{lemma: ss locus FM P1}
 Assume that $\mathcal O(\bm{d} + \bm{a})$ is ample. Then,
 \begin{gather*}
  F_j^{\op{ss}}(\bm{d}+\bm{a}):=  F_j^{\op{ss}}(\mathcal O(\bm{d}+\bm{a})) = \{ x \in F_j \mid \text{if } I \subset [n] \text{ with } \psi_{j,0}(x)_i=\psi_{j,0}(x)_{i'}~ \forall i,i'\in I, \text{ then } \\ \frac{1}{2}\sum_{i \in I} d_i + \sum_{\substack{S \subseteq I^c \\ |S| \geq n-j-1 }} a_S \leq \frac{1}{2}\sum_{i \in I^c} d_i + \sum_{\substack{S \subseteq I \\ |S| \geq n-j-1 }} a_S\}.
 \end{gather*}
\end{lemma}

\begin{proof}
 This is an immediate combination of the Hilbert-Mumford numerical criterion, Theorem \ref{theorem: Hilbert-Mumford numerical stability}, and Lemma \ref{lemma: P1[n] values of HM function}. 
\end{proof}

Define a hyperplane, 
\begin{displaymath}
 H_{j,I} := \{ \mathcal O(\bm{d}+\bm{a}) \mid \frac{1}{2}\sum_{i \in I} d_i + \sum_{\substack{S \subseteq I^c \\ |S| \geq n-j-1 }} a_S = \frac{1}{2}\sum_{i \in I^c} d_i + \sum_{\substack{S \subseteq I \\ |S| \geq n-j-1 }} a_S\} \subset \op{NS}^{\op{PGL}_2}(F_j)_{\R},
\end{displaymath}
and half-spaces
\begin{align*}
 H_{j,I}^{\geq} := \{ \mathcal O(\bm{d}+\bm{a}) \mid \frac{1}{2}\sum_{i \in I} d_i + \sum_{\substack{S \subseteq I^c \\ |S| \geq n-j-1 }} a_S \geq \frac{1}{2}\sum_{i \in I^c} d_i + \sum_{\substack{S \subseteq I \\ |S| \geq n-j-1 }} a_S\} \subset \op{NS}^{\op{PGL}_2}(F_j)_{\R} \\
 H_{j,I}^{>} := \{ \mathcal O(\bm{d}+\bm{a}) \mid \frac{1}{2}\sum_{i \in I} d_i + \sum_{\substack{S \subseteq I^c \\ |S| \geq n-j-1 }} a_S > \frac{1}{2}\sum_{i \in I^c} d_i + \sum_{\substack{S \subseteq I \\ |S| \geq n-j-1 }} a_S\} \subset \op{NS}^{\op{PGL}_2}(F_j)_{\R} \\
 H_{j,I}^{\leq} := \{ \mathcal O(\bm{d}+\bm{a}) \mid \frac{1}{2}\sum_{i \in I} d_i + \sum_{\substack{S \subseteq I^c \\ |S| \geq n-j-1 }} a_S \leq \frac{1}{2}\sum_{i \in I^c} d_i + \sum_{\substack{S \subseteq I \\ |S| \geq n-j-1 }} a_S\} \subset \op{NS}^{\op{PGL}_2}(F_j)_{\R} \\
 H_{j,I}^{<} := \{ \mathcal O(\bm{d}+\bm{a}) \mid \frac{1}{2}\sum_{i \in I} d_i + \sum_{\substack{S \subseteq I^c \\ |S| \geq n-j-1 }} a_S < \frac{1}{2}\sum_{i \in I^c} d_i + \sum_{\substack{S \subseteq I \\ |S| \geq n-j-1 }} a_S\} \subset \op{NS}^{\op{PGL}_2}(F_j)_{\R},
\end{align*}

\begin{corollary} \label{corollary: extended GIT fan for FM}
 The subsets of constant semi-stable locus within the ample cone, $\op{Amp}(F_j)_{\R}$, are the intersections,
 \begin{displaymath}
  H_{j,I_1}^{> 0} \cap \cdots \cap H_{j,I_{j_+}}^{> 0} \cap H_{j,I_{l_+ + 1}}^{< 0} \cap \cdots \cap H_{j,I_{l_+ + l_-}}^{< 0} \cap H_{j,I_{l_+ + l_- + 1}} \cap \cdots \cap H_{j,I_{l_+ + l_- + l_0}} \cap \op{Amp}(F_j)_{\R}.
 \end{displaymath}
 
 The closures of these subsets form a fan whose support is the ample cone. The chambers of this fan correspond to intersections with $l_0 = 0$ and the walls correspond to intersections with $l_0 = 1$.
\end{corollary}

\begin{proof}
 The first statement is an immediate consequence of Lemma \ref{lemma: ss locus FM P1}. The final statements are clear.
\end{proof}

\begin{definition}
 We call the fan of Corollary \ref{corollary: extended GIT fan for FM} the \textbf{extended GIT fan} for the action of $\op{PGL}_2$ on $F_j$. 

 We say that $\mathcal O(\bm{d}+\bm{a}) \in \op{NS}^{\op{PGL}_2}(F_j)_{\R} \cap \op{Amp}(F_j)_{\R}$ \textbf{slides into the abyss} if there exists a continuous path, 
 \begin{displaymath}
  \gamma: [0,1] \to \op{NS}^{\op{PGL}_2}(F_j)_{\R} \cap \op{Amp}(F_j)_{\R},
 \end{displaymath}
 satisfying
 \begin{itemize}
  \item $\gamma(0) = \mathcal O(\bm{d}+\bm{a})$,
  \item $\gamma(1)$ has empty semi-stable locus,
  \item $\gamma$ passes through no cones of codimension $\leq 2$ in the extended GIT fan for $F_j$ and $\gamma(t)$ lies in a codimension one cone for only finitely many $t$,
  \item Assume that $C_-$ and $C_+$ are chambers in the extended GIT fan separated by a wall, $M$, with $\gamma([t_-,t_0]) \subset C_-, \gamma(t_0) \in M,$ and $\gamma([t_0,t_+]) \subset C_+$. Let $H_{j,I}$ be the hyperplane containing $M$. If $C_- \subset H_{j,I}^{\leq 0}$ and $C_+ \subset H^{\geq 0}_{j,I}$, then $|I| \leq |I^c|$.
 \end{itemize}
\end{definition}

We set
\begin{displaymath}
 F_j \modmod{\bm{d}+\bm{a}} \op{PGL}_2 := [F_j^{\op{ss}}(\bm{d}+\bm{a})/\op{PGL}_2].
\end{displaymath}

\begin{lemma} \label{lemma: elem wall crossing in ext GIT fan for FM} 
 Let $C_-$ and $C_+$ be adjacent chambers in the extended GIT fan of $F_j$ and let $M$ be the common face of $C_+$ and $C_-$. Choose $\mathcal O(\bm{d}_++\bm{a}_+) \in \op{Int} C_+, \mathcal O(\bm{d}_- + \bm{a}_-) \in \op{Int} C_-, \mathcal O(\bm{d}_0 + \bm{a}_0) \in M$. There is an elementary wall crossing,
 \begin{align*}
  F_j^{\op{ss}}(\bm{d}_0+\bm{a}_0) & = F_j^{\op{ss}}(\bm{d}_+ + \bm{a}_+) \sqcup S_{\lambda} \\
  F_j^{\op{ss}}(\bm{d}_0+\bm{a}_0) & = F_j^{\op{ss}}(\bm{d}_- + \bm{a}_-) \sqcup S_{-\lambda}.
 \end{align*}
\end{lemma}

\begin{proof}
 We verify that the wall variation given by $\mathcal O(\bm{d}_++\bm{a}_+) \in \op{Int} C_+, \mathcal O(\bm{d}_- + \bm{a}_-) \in \op{Int} C_-$ close to $M$ satisfies the DHT condition and apply Theorem \ref{theorem: VGIT give elementary wall crossing}. 
 
 The first condition is satisfied as the semi-stable locus is constant in $\op{Int} C_+,\op{Int}C_-$ by Corollary \ref{corollary: extended GIT fan for FM}. The stabilizer of a point in $F_j$ is either $\mathbb{G}_m$ or $1$ so the second condition is satisfied. 
 
 Choose $I$ such that $H_{j,I}$ contains $M_0$ and $H^{\geq 0}_{j,I}$ contains $C_+$. Then, by Lemma \ref{lemma: ss locus FM P1}, 
 \begin{gather*}
  F_j^{\op{ss}}(\bm{d}_0+\bm{a}_0) \setminus \left(  F_j^{\op{ss}}(\bm{d}_+ + \bm{a}_+) \cup F_j^{\op{ss}}(\bm{d}_- + \bm{a}_-) \right) = \\ \{ x \in F_j \mid \psi_{j,0}(x)_i = \psi_{j,0}(x)_{i'} \text{ if } i,i' \in I \text { or } i,i' \in I^c \text{ and } \psi_{j,0}(x)_i \not = \psi_{j,0}(x)_{i'} \text{ if } i \in I, i' \in I^c \}.
 \end{gather*}
 This is connected as it is the inverse of a connected set in $F_0$ and the fibers of $\phi_{j,0}$ are connected. So, we satisfy the final condition.
\end{proof}

\begin{lemma} \label{lemma: abyss gives SOD}
 Assume that $\mathcal O(\bm{d}+\bm{a})$ slides into the abyss. Then, the derived category, $\dbcoh{F_j \modmod{\bm{d}+\bm{a}} \op{PGL}_2}$, admits a full exceptional collection.
\end{lemma}

\begin{proof}
 As $\mathcal O(\bm{d}+\bm{a})$ slides into the abyss, there exists chambers, $C_0,\ldots,C_t$, and walls, $M_1,\ldots,M_t$, ordered by direction along $\gamma$. Lemma \ref{lemma: elem wall crossing in ext GIT fan for FM} guarantees that we have an elementary wall crossing so we may apply Theorem \ref{theorem: elementary wall crossing} as we pass from $C_l$ through $M_{l+1}$ to $C_{l+1}$. Let $I_l \subset [n]$ be such that $M_{l} \subset H_{j,I_l}$ and $C_l \subset H^{\geq 0}_{j,I_l}$. We may use Lemma \ref{lemma: formula for mu} to compute the direction of the semi-orthogonal decomposition. Let $\lambda$ be a one-parameter subgroup and let $x$ be a fixed point in the semi-stable locus of the wall. We have
 \begin{displaymath}
  \omega_{F_j}^{-1} = \mathcal O(\bm{2}) \otimes \mathcal O(\sum_{|S| \geq n-j+1} (-|S|+2)\Sigma_j^S)
 \end{displaymath}
 and
 \begin{displaymath}
  \mu_l := \mu(\omega_{F_j}^{-1},\lambda,x) = |I_l|-|I_l^c| + \sum_{\substack{S \subseteq I_l \\ |S| \geq n-j+1}} (|S|-2) - \sum_{\substack{S \subseteq I_l^c \\ |S| \geq n-j+1}} (|S|-2).
 \end{displaymath}
 The sign of this quantity is the same as the sign of $|I_l| - |I^c_l|$, which by assumption is $\leq 0$. Thus, we have a semi-orthogonal decomposition,
 \begin{displaymath}
  \dbcoh{F_j \modmod{\bm{d}_{l-1}+\bm{a}_{l-1}} \op{PGL}_2} = \langle \dbcoh{Y_{\lambda}},\ldots,\dbcoh{Y_{\lambda}},\dbcoh{F_j \modmod{\bm{d}_{l}+\bm{a}_{l}} \op{PGL}_2} \rangle.
 \end{displaymath}
 As $G_{\lambda} = 1$, $Y_{\lambda} = Z_{\lambda}^0$ which is the fiber over a fixed point in $F_0$. This an iterated blow up of projective space along strict transforms of linear subspaces so it possesses a full exceptional collection by \cite{Orl92}. Thus, $F_j \modmod{\bm{d}_{l-1}+\bm{a}_{l-1}} \op{PGL}_2$ possesses a full exceptional collection if $F_j \modmod{\bm{d}_{l}+\bm{a}_{l}} \op{PGL}_2$ possesses a full exceptional collection. As $C_t$ has empty semi-stable locus, we can proceed by downward induction and conclude that $F_j \modmod{\bm{d}_{0}+\bm{a}_{0}} \op{PGL}_2 = F_j \modmod{\bm{d}+\bm{a}} \op{PGL}_2$ possesses a full exceptional collection.
\end{proof}

\sidenote{{\color{red} Deligne-Grothendieck-Knudsen-Mumford??? No consensus}}
\begin{theorem} \label{thm: exceptional collection stable curves}
 The derived categories of the Deligne-Grothendieck-Knudsen-Mumford moduli spaces of stable marked rational curves, $\overline{M}_{0,n}$, and the Hassett moduli spaces of stable symmetrically weighted rational curves, $\overline{M}_{0,n \cdot \epsilon}$, all admit full exceptional collections.
\end{theorem}

\begin{proof}
The basic idea is that the walls in the interior of the extended GIT fan for $F_j$ are in bijection with those for $F_0$ by Corollary~\ref{corollary: GIT fan for points in P1} and Corollary~\ref{corollary: extended GIT fan for FM}.  We argue that one may therefore chose a path in the extended GIT fan for  $F_0$ which slides into the abyss and lifts to one starting at the chamber corresponding to $\overline{M}_{0,n \cdot \epsilon}$ in the extended GIT fan for $F_j$.
 
Fix $j > \lceil n/2 \rceil$. Let $\bm{\delta} := \delta_n > \delta_{n-1} > \cdots > \delta_2 > 0$ be a decreasing sequence of positive real numbers and consider the subset, $\op{NS}^{\op{PGL_2}}(F_j)_{\R}^{\leq \bm{\delta}}$, formed by $\mathcal O(\bm{d}+\bm{a})$ with
 \begin{displaymath}
  -\delta_{|S|} \leq a_S \leq 0
 \end{displaymath}
 for all $S \subseteq [n]$ and $\mathcal O_{F_0}(\bm{d})$ ample. Let
 \begin{align*}
  \pi_j: \op{NS}^{\op{PGL}_2}(F_j)_{\R} & \to \op{NS}^{\op{PGL}_2}(F_0)_{\R} \\
  \bm{d} + \bm{a} \mapsto \bm{d}
 \end{align*}
 be the projection. Consider the images,
 \begin{align*}
  B_2^{\leq \bm{\delta}} & := \pi_j\left( \left(\bigcup_{I_1 \not = I_2 \subseteq [n]} H_{j,I_1} \cap H_{j,I_2} \right) \cap \op{Amp}(F_j)_{\R}^{\leq \bm{\delta}}\right).
 \end{align*}
 Take $\bm{\delta}$ small enough such that the GIT quotient with respect to $\bm{2} - \bm{\delta}$ is $\overline{M}_{0,n \cdot \epsilon_j}$. Such a $\bm{\delta}$ exists by Theorem \ref{theorem: moduli of stable curves as GIT quotient}. Shrink $\bm{\delta}$, if necessary, so that there is a $\bm{d_0} \in \op{Amp}(F_0)_{\R} $  and a straight line path 
 \begin{displaymath}
  \gamma(t): [0,1] \to \op{NS}^{\op{PGL}_2}(F_0)_{\R} \cap \op{Amp}(F_0)_{\R},
 \end{displaymath}
 such that
 \begin{itemize}
  \item $\bm{d}_0 - \bm{\delta}$ lies in the interior of the GIT chamber containing $\bm{2} - \bm{\delta}$,
  \item $\gamma(0) = \bm{d}_0$,
  \item $\gamma([0,1]) \cap B_2^{\leq \bm{\delta}} = \emptyset$,
  \item $\gamma(1)$ lies in the interior of an empty chamber.
  \item $\gamma(t) - \bm{\delta}$ is ample on $F_j$ for all $t \in [0,1]$
 \end{itemize}
 We claim that the new path,
 \begin{align*}
  \tilde{\gamma}: [0,1] & \to \op{NS}^{\op{PGL}_2}(F_j)  \cap \op{Amp}(F_j)_{\R} \\
  t & \mapsto \gamma(t) - \bm{\delta},
 \end{align*}
 allows $\bm{d}_0 - \bm{\delta}$ to slide into the abyss.
 
 The first condition for verifying that $\bm{d}_0 - \bm{\delta}$ slides into the abyss is clear. The second condition follows from the fact that the GIT quotient of $F_j$ at $\tilde{\gamma}(t)$ is the blow up of the GIT quotient $F_0$ by $\gamma(t)$ whenever $\tilde{\gamma}(t)$ and $\gamma(t)$ lie in interiors of chambers. The third is guaranteed by $\gamma([0,1]) \cap B_2^{\leq \bm{\delta}} = \emptyset$. For each wall that $\tilde{\gamma}$ crosses corresponding to $I$, $\gamma$ crosses the wall corresponding to $I$ in the same direction from negative to positive with respect to $H_{j,I}$ or $H_I$. As $\gamma$ is a straight line path from the anti-canonical line bundle, we satisfy the final condition by Lemma \ref{lemma: formula for mu}.
\end{proof}

\begin{corollary}
 The $\mathbb{Q}$-Chow motive of $\overline{M}_{0,n \cdot \epsilon}$ is a direct sum of tensor powers of the Lefschetz motive.
\end{corollary}

\begin{proof}
 This is a direct consequence of Theorem~\ref{thm: exceptional collection stable curves} and Theorem 1.3 of \cite{MT}, see also \cite{BB11} for $k = \C$.
\end{proof}

\begin{remark}
 What chambers of the GIT fan of $F_0$ can slide into the abyss? For $n \geq 5$, not all can. 
\end{remark}

\section{A generalization of Orlov's sigma model-LG model correspondence} \label{section: Orlov}

\subsection{Background on Orlov's theorem}

Let us recall the main result of \cite{Orl09}. Let $R$ be a (possibly non-commutative) connected $\mathbb{Z}$-graded algebra, $R = \bigoplus_{i \geq 0} R_i$ with $R_0 = k$. Let $R_+ := \bigoplus_{i > 0} R_i$. Assume that $R$ is Noetherian. Let $\op{tors}(R,\Z)$ be the full subcategory of the category of finitely-generated $\Z$-graded $R$-modules, $\op{mod}(R,\Z)$, consisting of all $M$ such that there is a $n$ with $R_+^n M = 0$. 

We let $\op{qgr }R$ be the quotient category, $\op{mod}(R,\mathbb{Z})/\op{tors}(R,\mathbb{Z})$. In addition, let $\op{D}_{\op{sg}}(R,\mathbb{Z})$ be the Verdier quotient of the category of finitely generated graded left $R$-modules by the category of perfect $R$-modules, $\op{D}^{\op{b}}(\op{mod }R,\mathbb{Z})/\op{Perf }R$, where the category of perfect $R$-modules, $\op{Perf }R$, is defined as the full subcategory consisting of bounded complexes of finite-rank graded free $R$-modules. The category, $\op{D}_{\op{sg}}(R,\mathbb{Z})$, is called the \textbf{graded singularity category} of $R$.

\begin{theorem} \label{theorem: Orlov}
 Fix $d \in \mathbb{Z}$. Assume that $R$ is Gorenstein with parameter $a \in \Z$. 
\begin{enumerate}
 \item If $a > 0$, there is a fully-faithful functor,
\begin{displaymath}
 \Phi_d: \op{D}_{\op{sg}}(R,\mathbb{Z}) \to \op{D}^{\op{b}}(\op{qgr }R),
\end{displaymath}
 and a semi-orthogonal decomposition, with respect to $\Phi_d$,
\begin{displaymath}
 \op{D}^{\op{b}}(\op{qgr }R) = \langle R(-a-d+1),\ldots, R(-d), \op{D}_{\op{sg}}(R,\mathbb{Z}) \rangle.
\end{displaymath}
 \item If $a = 0$, there is an equivalence,
\begin{displaymath}
 \Phi_d: \op{D}_{\op{sg}}(R,\mathbb{Z}) \to \op{D}^{\op{b}}(\op{qgr }R).
\end{displaymath}
 \item If $a < 0$, there is a fully-faithful functor,
\begin{displaymath}
 \Psi_d:  \op{D}^{\op{b}}(\op{qgr }R) \to \op{D}_{\op{sg}}(R,\mathbb{Z}),
\end{displaymath}
 and a semi-orthogonal decomposition, with respect to $\Psi_d$,
\begin{displaymath}
 \op{D}_{\op{sg}}(R,\mathbb{Z}) = \langle k(-d),\ldots, k(a-d+1), \op{D}^{\op{b}}(\op{qgr }R) \rangle.
\end{displaymath}
\end{enumerate}
\end{theorem}

\begin{proof}
 This is Theorem 2.5 of \cite{Orl09}.
\end{proof}

Orlov applies his theorem to the following special case. Let $f_1,\ldots,f_c$ be homogeneous regular sequence in $S= k[x_1,\ldots,x_n]$ of degrees, $d_1,\ldots, d_c$. Then, the graded algebra, $R := S/(f_1,\ldots,f_c)$, is Gorenstein with parameter $n - \sum d_i$. Let $Y$ be the associated complete intersection in $\mathbb{P}^{n-1}$. Serre's theorem, \cite{SerreFAC}, states that $\op{qgr }S$ is equivalent to $\op{coh }{Y}$.

\begin{corollary} \label{corollary: Orlov}
 With the notation as above, fix $d \in \mathbb{Z}$. Let $a = n - \sum d_j$.
\begin{enumerate}
 \item If $a > 0$, there is a fully-faithful functor,
\begin{displaymath}
 \Phi_d: \op{D}_{\op{sg}}(R,\mathbb{Z}) \to \dbcoh{Y},
\end{displaymath}
 and a semi-orthogonal decomposition, with respect to $\Phi_d$,
\begin{displaymath}
 \dbcoh{X} = \langle \mathcal O_{Y}(a-d+1),\ldots, \mathcal O_{Y}(-d), \op{D}_{\op{sg}}(R,\mathbb{Z}) \rangle.
\end{displaymath}
 \item If $a = 0$, there is an equivalence,
\begin{displaymath}
 \Phi_d: \op{D}_{\op{sg}}(R,\mathbb{Z}) \to \dbcoh{Y}.
\end{displaymath}
 \item If $a < 0$, there is a fully-faithful functor,
\begin{displaymath}
 \Psi_d:  \dbcoh{Y} \to \op{D}_{\op{sg}}(R,\mathbb{Z}),
\end{displaymath}
 and a semi-orthogonal decomposition, with respect to $\Psi_d$,
\begin{displaymath}
 \op{D}_{\op{sg}}(R,\mathbb{Z}) = \langle k(-d),\ldots, k(\mu-d+1), \dbcoh{Y} \rangle.
\end{displaymath}
\end{enumerate}
\end{corollary}

\begin{proof}
 This is Theorem 2.13 of \cite{Orl09}.
\end{proof}

\subsection{VGIT and Orlov's theorem}

Corollary \ref{corollary: Orlov} is a very powerful result relating commutative algebra and projective geometry in a novel way. Applications of the result include constructing generators of derived categories of projective hypersurfaces \cite{BFK10} and providing a new proof and an extension \cite{CT,BFK11} of Griffiths' famous \cite{Gri} description of the primitive cohomology of a projective hypersurface in terms of the Jacobian ring of its defining function. As such, it is natural to seek a generalization of Corollary \ref{corollary: Orlov}. In this final section, we show how to recover the full statement of Corollary \ref{corollary: Orlov} using Theorem \ref{theorem: elementary wall crossing} in combination with Theorem \ref{theorem: Isik} of Isik and Shipman.

We use Theorem \ref{theorem: Isik} to replace the categories appearing in Corollary \ref{corollary: Orlov} by gauged LG models. Let $f_1,\ldots,f_c$ be homogeneous polynomials in $S= k[x_1,\ldots,x_n]$ of degrees, $d_1,\ldots, d_c$ such that $f_1,\ldots,f_c$ is a complete intersection. Let $R = S/(f_1,\ldots,f_c)$. We consider the ring,
\begin{displaymath}
 S[\bm{u}] := S[u_1,\ldots,u_c],
\end{displaymath}
with a $\mathbb{G}_m$ action where $u_i$ has degree $-d_i$. In addition, we allow another copy of $\mathbb{G}_m$ to act on $S[\bm{u}]$. The elements of $S$ have degree zero with respect to this additional action while the $u_i$ have degree $1$. Following the language of physics, we call this additional action, R-symmetry. We also include the potential,
\begin{displaymath}
 w := u_1f_1 + \cdots + u_cf_c \in S[\bm{u}].
\end{displaymath}
Let $X := \op{Spec} S[\bm{u}]$, $Z(u) := \op{Spec }S$, and $Z(x) := \op{Spec }k[u_1,\ldots,u_c]$.

\begin{lemma} \label{lemma: Isik theorem for affine CI}
 There is an exact equivalence of triangulated categories,
 \begin{displaymath}
  \mathfrak I: \op{D}^{\op{b}}(\op{mod }R,\mathbb{Z}) \to \dbcoh{[X/\mathbb{G}_m^2],w},
 \end{displaymath}
 that restricts to equivalences,
 \begin{align*}
  \op{D}^{\op{b}}(\op{tors }R,\mathbb{Z}) & \cong \op{D}^{\op{b}}_{Z(x)}(\op{coh }[X/\mathbb{G}_m^2],w) \\
  \op{Perf }R & \cong \op{D}^{\op{b}}_{Z(u)}(\op{coh }[X/\mathbb{G}_m^2],w).
 \end{align*}
\end{lemma}

\begin{proof}
 The equivalence, $\mathfrak I$, comes from Theorem \ref{theorem: Isik}. Under this equivalence, the $R$-module, $k(i)$, corresponds to the factorization,
 \begin{center}
  \begin{tikzpicture}[description/.style={fill=white,inner sep=2pt}]
   \matrix (m) [matrix of math nodes, row sep=3em, column sep=3em, text height=1.5ex, text depth=0.25ex]
    {  0 & \mathcal O_{Z(x)}(i), \\ };
   \path[->,font=\scriptsize]
    (m-1-1) edge[out=30,in=150] node[above] {$0$} (m-1-2)
    (m-1-2) edge[out=210, in=330] node[below] {$0$} (m-1-1);
  \end{tikzpicture}
 \end{center}
 while $R(i)$ corresponds to 
 \begin{center}
 \begin{tikzpicture}[description/.style={fill=white,inner sep=2pt}]
  \matrix (m) [matrix of math nodes, row sep=3em, column sep=3em, text height=1.5ex, text depth=0.25ex]
    {  0 & \mathcal O_{Z(u)}(i). \\ };
  \path[->,font=\scriptsize]
    (m-1-1) edge[out=30,in=150] node[above] {$0$} (m-1-2)
    (m-1-2) edge[out=210, in=330] node[below] {$0$} (m-1-1);
  \end{tikzpicture}
 \end{center}
 Here, we twist the original $\mathbb{G}_m$ structure and not the R-symmetry. 
\end{proof}

Let $X_+ := X \setminus Z(x)$, $X_- := X \setminus Z(u)$, and denote the projective complete intersection determined by $f_1,\ldots,f_c$ by $Y$.

\begin{corollary} \label{corollary: Iski theorem for affine CI}
 There are exact equivalences,
 \begin{align*}
  \dbcoh{Y} & \cong \dbcoh{[X_+/\mathbb{G}_m^2],w|_{X_+}} \\ 
  \op{D}_{\op{sg}}(R,\mathbb{Z}) & \cong \dbcoh{[X_-/\mathbb{G}_m^2],w|_{X_-}}.
 \end{align*}
\end{corollary}

\begin{proof}
 There are equivalences,
 \begin{align*}
  \dbcoh{[X/\mathbb{G}_m^2],w}/\op{D}^{\op{b}}_{Z(x)}(\op{coh }[X/\mathbb{G}_m^2],w) & \cong \dbcoh{[X_+/\mathbb{G}_m^2],w|_{X_+}} \\
  \dbcoh{[X/\mathbb{G}_m^2],w}/\op{D}^{\op{b}}_{Z(u)}(\op{coh }[X/\mathbb{G}_m^2],w) & \cong \dbcoh{[X_-/\mathbb{G}_m^2],w|_{X_-}}.
 \end{align*}
 While, by \cite{SerreFAC}, 
 \begin{displaymath}
  \op{D}^{\op{b}}(\op{mod }R,\mathbb{Z})/\op{D}^{\op{b}}(\op{tors }R,\mathbb{Z}) \cong \dbcoh{Y}
 \end{displaymath}
 and
 \begin{displaymath}
  \op{D}_{\op{sg}}(R,\mathbb{Z}) := \op{D}^{\op{b}}(\op{mod }R,\mathbb{Z})/\op{Perf }R.
 \end{displaymath}
 Lemma \ref{lemma: Isik theorem for affine CI} gives the conclusion.
\end{proof}

Consider the one-parameter subgroup,
\begin{align*}
 \lambda: \mathbb{G}_m & \to \mathbb{G}_m^2 \\
 \alpha & \mapsto (\alpha,1).
\end{align*}
Then,
\begin{align*}
 Z_{\lambda}^0 & = \{0\} \in X \\
 S_{\lambda} & = Z_{\lambda} = Z(x) \\
 S_{-\lambda} & = Z_{-\lambda} = Z(u).
\end{align*}
Thus, we have an elementary wall crossing,
\begin{align*}
 X & = X_+ \sqcup S_{\lambda} \\
 X & = X_- \sqcup S_{-\lambda},
\end{align*}
with $t(\mathfrak{K}^+) = -n$ and $t(\mathfrak{K}^-) = -\sum d_i$. 

\begin{proof}[Another proof of Corollary \ref{corollary: Orlov}]
 Note that $a = -t(\mathfrak{K}^+) + t(\mathfrak{K}^-)$ so the three cases of Theorem \ref{theorem: elementary wall crossing} are exactly the cases $a > 0$, $a = 0$, and $a < 0$. Let us first assume that $a \geq 0$. Theorem \ref{theorem: elementary wall crossing} states that we have fully-faithful functors,
 \begin{displaymath}
  \Phi^+_d: \dcoh{[X_-/G],w|_{X_-}} \to \dcoh{[X_+/G],w|_{X_+}},
 \end{displaymath}
 and, for $-t(\mathfrak{K}^-) + d \leq j \leq -t(\mathfrak{K}^+) + d -1$,
 \begin{displaymath}
  \Upsilon_j^+: \dcoh{Y_{\lambda},w_{\lambda}} \to \dcoh{[X_+/G],w|_{X_+}},
 \end{displaymath}
 and a semi-orthogonal decomposition,
 \begin{displaymath}
  \dcoh{[X_+/G],w|_{X_+}} = \langle \Upsilon^+_{-t(\mathfrak{K}^-)+d}, \ldots, \Upsilon^+_{-t(\mathfrak{K}^+)+d-1}, \Phi^+_d \rangle,
 \end{displaymath}
 where $G = \mathbb{G}_m^2$, $Y_{\lambda} = [Z_{\lambda}^0/G_{\lambda}]$, and $G_{\lambda} \cong \mathbb{G}_m$. As noted above, $Z_{\lambda}^0 = \op{pt}$ and $w|_{Z^0_{\lambda}} = 0$ so 
 \begin{displaymath}
  \dcoh{Y_{\lambda},w_{\lambda}} \cong \dcoh{[\op{pt}/\mathbb{G}_m],0} \cong \dbcoh{\op{pt}}. 
 \end{displaymath}
 Under this equivalence and $\Upsilon_l^+$, the image of $\mathcal O_{\op{pt}}$ is the factorization, 
 \begin{center}
  \begin{tikzpicture}[description/.style={fill=white,inner sep=2pt}]
   \matrix (m) [matrix of math nodes, row sep=3em, column sep=3em, text height=1.5ex, text depth=0.25ex]
    {  0 & \mathcal O_{Z(u)}(l), \\ };
   \path[->,font=\scriptsize]
    (m-1-1) edge[out=30,in=150] node[above] {$0$} (m-1-2)
    (m-1-2) edge[out=210, in=330] node[below] {$0$} (m-1-1);
  \end{tikzpicture}
 \end{center}
 which, under the equivalences of Lemma \ref{lemma: Isik theorem for affine CI} and Corollary \ref{corollary: Iski theorem for affine CI}, is $\mathcal O_Y(l)$. Thus, up to reindexing, we recover the $a \geq 0$ cases of Corollary \ref{corollary: Orlov} in full. The $a < 0$ case is entirely analogous.
\end{proof}

\begin{remark}
 The VGIT approach to Orlov's Theorem extends to complete intersections in GIT quotients. A large class of GIT quotients to consider is the class of toric DM stacks. However, in general, within the secondary fan of an arbitrary toric DM stack, there is no affine point, i.e. a chamber such that the GIT quotient is an affine variety.  On the other hand, there is a simple algebraic criteria for existence of such a chamber. This was conjectured by Herbst and proven by P. Clarke and J. Guffin in \cite{CG}.
\end{remark}

\begin{remark}
 Orlov's full result, Theorem \ref{theorem: Orlov}, is, of course, much more general. The existence of a such a general result suggests that one can mimic some of the techniques for GIT in the case of a linear algebraic group acting on a smooth and proper dg-category. It would be interesting to fully realize this statement.
\end{remark}

%%%%%%%%%%%%%%%%%%%%%%%%%%%%%%%%%%%%%%%%%%%%%%%%%%%%%%%%%%%%%%%%%%%%%%%%%%%%

\end{document}